\documentclass{article}

\usepackage{geometry}
\usepackage{microtype}
\usepackage{graphicx}
\usepackage{wrapfig}
\usepackage{subfigure}
\usepackage{booktabs}
\usepackage{comment}
\usepackage{authblk}
\usepackage{bm}
\usepackage{placeins}

\usepackage{hyperref}
\usepackage[numbers]{natbib}

\usepackage{latexsym,pstricks} 
\usepackage{comment}

\usepackage{amsthm,amsfonts}
\usepackage{amsmath}
\usepackage{amssymb}
\usepackage{mathtools}
\usepackage{mathrsfs}
\usepackage{dsfont}

\newcommand{\Esp}[1]{\mathbb{E} \left[ #1\right]}
\newcommand{\Prob}[1]{\mathbb{P} \left( #1 \right)}
\newcommand{\ind}[1]{\mathds{1}_{#1}}

\newcommand{\fn}{f_n}
\newcommand{\forest}{f_{M,n}}
\newcommand{\infforest}{f_{\infty,n}}
\newcommand{\vfcrf}{f_{\infty,n}^{\mathrm{VF}}}
\newcommand{\medrf}{f_{\infty,n}^{\mathrm{MedRF}}}

\newcommand{\T}{\Theta}
\newcommand{\E}{\mathds{E}}

\renewcommand{\P}{\mathds{P}}
\newcommand{\N}{\mathds{N}}

\newcommand{\bX}{{X}}
\newcommand{\bx}{{x}}
\newcommand{\x}{{x}}
\newcommand{\X}{{X}}

\newcommand{\K}{K_k^{cc}(\x, \X)}
\newcommand{\multi}{ \sum\limits_{\substack{k_1,\hdots,k_d \\ \sum_{j=1}^d k_j = k }}}
\newcommand{\integer}[1]{\lfloor #1 \rfloor}

\newcommand{\diff}{\mbox{\normalfont d}}

\usepackage[colorinlistoftodos,textwidth=2.3cm]{todonotes}

\usepackage[normalem]{ulem}
\definecolor{Vert}{RGB}{0,128,0}
\newcommand{\replace}[2]{{#2}}

\usepackage[capitalize,noabbrev]{cleveref}

\theoremstyle{plain}
\newtheorem{theorem}{Theorem}[section]
\newtheorem{proposition}[theorem]{Proposition}
\newtheorem{lemma}[theorem]{Lemma}
\newtheorem{corollary}[theorem]{Corollary}
\theoremstyle{definition}
\newtheorem{definition}[theorem]{Definition}

\theoremstyle{remark}
\newtheorem{remark}[theorem]{Remark}

\title{Is interpolation benign for random forest regression?}

\author[1]{Ludovic Arnould}
\author[1,2]{Claire Boyer}
\author[3]{Erwan Scornet}

\affil[1]{LPSM, Sorbonne Universit\'e, Paris, France}
\affil[2]{MOKAPLAN, INRIA Paris}
\affil[3]{CMAP, Ecole Polytechnique, Paris, France}
\date{}

\begin{document}

\maketitle

\begin{abstract}
Statistical wisdom suggests that very complex models, interpolating training data, will be poor at predicting unseen examples.
Yet, this aphorism has been recently challenged by the identification of benign overfitting regimes, specially studied in the case of parametric models: generalization capabilities may be preserved despite model high complexity.
While it is widely known that fully-grown decision trees interpolate and, in turn, have bad predictive performances, the same behavior is yet to be analyzed for Random Forests (RF).
In this paper, we study the trade-off between interpolation and consistency for several types of RF algorithms. Theoretically, we prove that interpolation regimes and consistency cannot be achieved simultaneously for several non-adaptive RF.
Since adaptivity seems to be the cornerstone to bring together interpolation and consistency, we study interpolating Median RF which are proved to be consistent \replace{}{in the interpolating regime. This is the first result conciliating interpolation and consistency for RF, highlighting that the averaging effect introduced by feature randomization is a key mechanism, sufficient to ensure the consistency in the interpolation regime and beyond}.
Numerical experiments show that Breiman's RF are consistent while exactly interpolating, when no bootstrap step is involved.
We theoretically control the size of the interpolation area, which converges fast enough to zero, giving a necessary condition for exact interpolation and consistency to occur in conjunction.
\end{abstract}

\section{Introduction}

Random Forests \citep[RF, ][]{breiman2001random}  have proven to be very efficient algorithms, especially on tabular data sets. As any machine learning (ML) algorithm, Random Forests and Decision Trees have been analyzed and used according to the overfitting-underfitting trade-off. Regularization parameters have been introduced in order to control the variance while still reducing the bias. 
For instance, one can increase the variety of the constructed trees (by playing either with bootstrap samples or feature subsampling) or control the tree structure (by limiting either the number of points falling within each leaf or the maximum depth of all trees).

However, the paradigm stating that high model complexity leads to bad generalization capacity has been recently challenged: in particular, deeper and larger neural networks still empirically exhibit high predictive performances \citep{goodfellow2016deep}. In such situations, overfitting can be qualified as "benign":
complex models, possibly leading to interpolation of the training examples, still generalize well on unseen data \citep{bartlett2021deep}.

Regarding parametric methods, benign overfitting has been exhibited and well understood in linear regression \citep{bartlett2020benign,tsigler2020benign,liang2020multiple} and investigated in the context of neural networks \citep{belkin2019reconciling}.
Many researchers currently study the \emph{implicit bias} or \emph{implicit regularization} of stochastic gradient (SGD) strategies used during neural network training: the optimization of an over-parametrized one-hidden-layer neural network via SGD will converge to a minimum of minimal norm with good generalization properties in a regression setting \citep{bach2021gradient}, or with maximal margin in a classification setting \citep{chizat2020implicit}.
%This phenomenon, first empirically observed, is now abundantly studied from a theoretical perspective. Most papers focus on the parametric setting (especially linear regression and two-layer neural networks) where it is easier to quantify the \textit{complexity} of the model in terms of number or norm of the model's parameters. Learning the weights of an over-parametrized neural network via stochastic gradient descent produces an \textit{implicit regularisation} effect \cite{bach2021gradient}.

Regarding non-parametric methods, practitioners have noticed the good performances of high-depth RFs for a long time (by default, several ML libraries such as the popular Scikit-Learn grow trees until pure leaves are reached). More recently, the use of interpolating (or very deep) trees for boosting and bagging methods has been discussed by \cite{tang2018random} and \cite{wyner2017explaining}. 
While \cite{tang2018random} criticize the relevancy of interpolating random forests, \citet{wyner2017explaining} believe that the \textit{self-averaging} process at hand in RF (or in boosting methods) also produces an implicit regularization that prevents the interpolating algorithm from overfitting. Note that the regularization properties of RF have also been studied in the light of their complexity \citep{buschjager2021there} and tree depth \citep{zhou2021trees}. 
{This phenomenon can be} put in parallel with the results proved in \cite{devroye1998hilbert} and \cite{belkin2019does} where they show that an interpolating kernel method using a singular kernel (similar to $K(x) =||x||^{-\alpha} \mathds{1}_{||x|| \leq 1}$) is consistent, reaching minimax convergence rate for $\beta$-Hölder regular functions. \replace{}{More recently, \citet{wang2022consistent} showed the consistency of interpolating kernel methods, defined on Riemannian manifolds, whose kernels can be written as  weighted random partition kernels on the sphere (similarly to the kernel random forest methods defined in Section \ref{sec:KeRF}).}

\noindent \textbf{Contributions and outline} 
In this paper, we study the trade-off between interpolation and consistency {in the context of regression}, 
\replace{}{for different types of RF:
\begin{itemize}
    \item Centered RF (Section \ref{sec:CRF}). We prove theoretically that interpolation regimes and consistency cannot be achieved simultaneously for non-adaptive centered RF . The major problem arises from empty cells in tree partitions. Therefore, we also study a slightly modified Centered RF that does not take into account empty cells;
    \item Kernel RF (Section \ref{sec:KeRF}). We then study a more refined version of the CRF, the Kernel Random Forest (KeRF), built by averaging over all connected data points. By neglecting empty cells, this method is consistent for larger tree depths, but does not meet the exact interpolation requirement yet;
    \item Median RF (Section \ref{sec:AdaCRF}). Since adaptivity seems to be the cornerstone to conciliate interpolation and consistency, we study the interpolating Median RF, which is proved to be consistent \replace{}{in the exact interpolation regime . For the first time, it is shown that the averaging effect of the feature randomization inside RF (without boostrap) is sufficient to "average the noise out" (interpolating trees being sensitive to the noise), i.e.\ to decrease the variance towards 0. The bias of interpolating trees can be still classically controlled;}
    \item Breiman RF (Section \ref{sec:BreimanRF}). Numerical experiments show that Breiman RF are consistent when exactly interpolating, i.e.\ when the whole data set is used to build each fully-grown tree (no bootstrap). It seems that the key randomization mechanism at work in RF is sufficient to reach consistency in spite of interpolation. Finally, we prove that the volume of the interpolation zone (where noise sensitivity is maximum) for an infinite Breiman RF tends to 0 at an exponential rate in the dimension $d$. This supports the idea that the decay of the interpolation volume could be fast enough to retrieve consistency despite interpolation.
\end{itemize}
}

%We also prove that its variance decreases towards 0 in an asymptotic  high-dimensional setting (note that a control of the bias can be also established provided extra assumptions on the considered model). This highlights that a better understanding of the effect of the randomization of splitting directions on the forest partition might be the key to theoretically conciliate  exact interpolation and consistency
%\replace{}{In order to better understand the averaging effect of the RF, we also prove that the variance of the Median RF decreases towards 0 in a high-dimensional setting where tree diversity (and thus the consequent averaging effect) is enhanced}. 
 %\replace{If bootstrap is used instead, we numerically show that Breiman RF are consistent but do not interpolate anymore: however each weak learner in the forest is inconsistent while being an interpolator. }{}
%If bootstrap is often considered as a way to prevent overfitting when dealing with interpolators, 
\begin{figure}[h]
    \centering
    \includegraphics[width=0.7\textwidth]{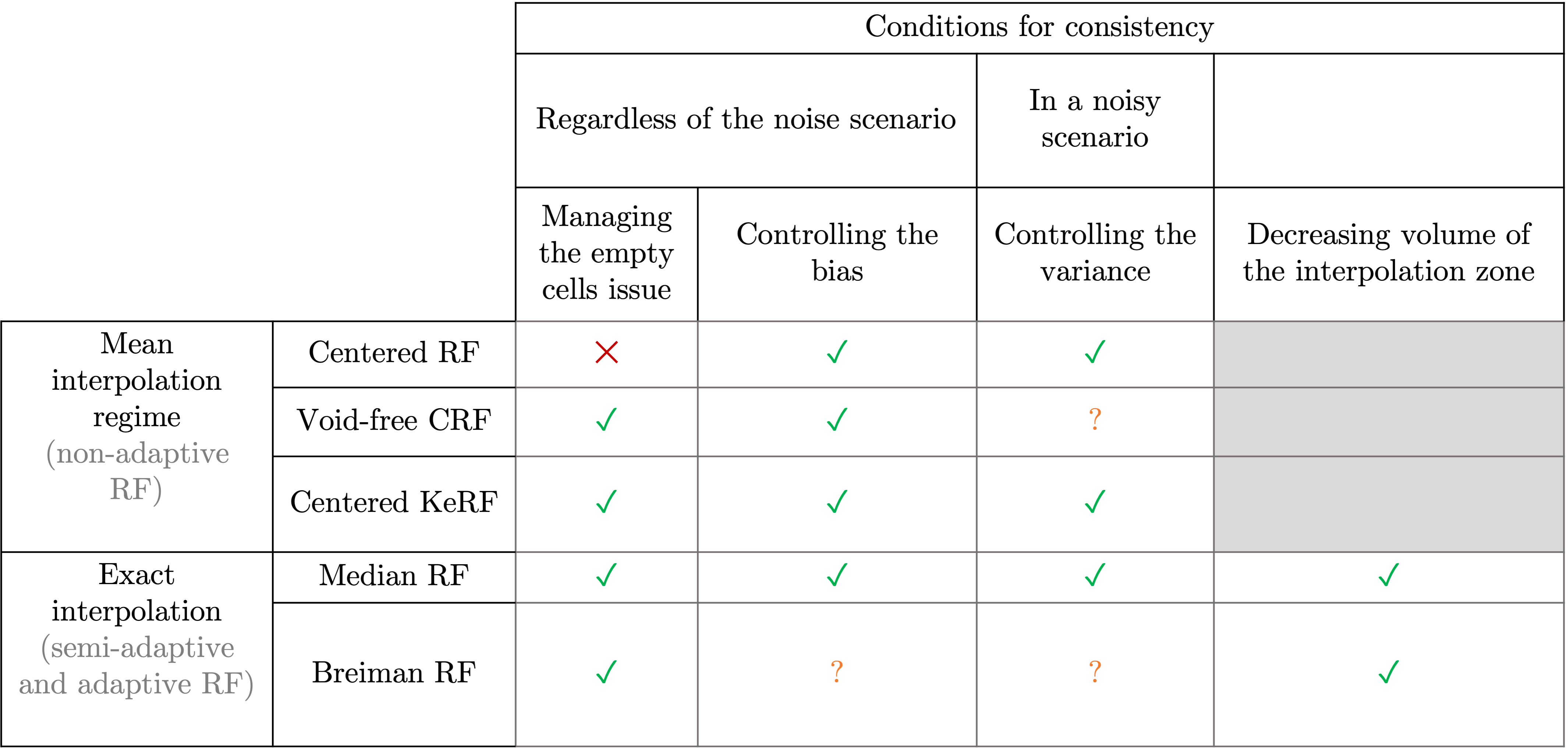}
    \caption{Summary of theoretical contributions.}
    \label{fig:contribution_summary}
\end{figure}
{Please refer to Figure \ref{fig:contribution_summary} for an overview of theoretical contributions.} 
All proofs and details on numerical experiments are given in Appendix \ref{sec:proofs} and \ref{sec:exp_setting}.

% \es{à diviser par 2}
 %In this paper, we aim at understanding the mechanisms at hand allowing RFs composed of \textit{strong} learners (interpolating trees here) to achieve good performances or consistency from both a theoretical and an empirical points of view. 
%To this end, we introduce several RF models, both adaptive and non-adaptive. We begin by proving that non-adaptive RFs struggle reaching \textit{exact} interpolation and that it can be harmful for consistency (Section \ref{sec:CRF}).
%We also show that non-adaptive methods are still consistent when verifying a weaker notion of interpolation (Section \ref{sec:KeRF}).
% In Section \ref{sec:AdaCRF}, we present an adaptive and interpolating method that successfully attain consistency in a noiseless setting.
% Finally, we prove that the volume of an infinite Breiman RF tends to 0 at a polynomial rate in the number of samples $n$ and an exponential rate in the dimension $d$ (Section \ref{sec:BreimanRF}). We then discuss how the randomness bringing diversity between the trees of the RF counterbalances the potential overfitting of individual trees in line with \cite{wyner2017explaining} and \cite{mentch2019randomization}.
%\noindent \textbf{Related work.}   

\section{Setting}
\label{sec:setting}

\paragraph{Framework} In a general non-parametric regression framework, we assume to be given a \emph{training set} $\mathcal{D}_n := ((X_1, Y_1), ...,$  $ (X_n, Y_n))$, composed of i.i.d.\ copies of the generic random variable $(X,Y)$, where the input $X$ is assumed \textit{throughout the paper} to be uniformly distributed over $ [0,1]^d$, and $Y \in \mathds{R}$ is the output. The underlying model is assumed to satisfy $Y = f^\star(X) + \varepsilon$, where $f^\star(x) = \Esp{Y |X = x}$ is the regression function and $\varepsilon$ a random noise satisfying, almost surely, $\Esp{\varepsilon | X } = 0$ and $\mathbb{V}[\varepsilon|X] \leq \sigma^2 < \infty$, for some $\sigma^2 \geq 0$. Given an input $x \in [0,1]^d$, the goal is to estimate the associated response $f^\star(x)$.
We measure the performance of an estimator $\fn$ via its \textit{excess risk}, defined as $\mathcal{R}(\fn) := \Esp{(\fn(X) - f^\star(X))^2}$, and its consistency property. 
%The asymptotic performance of an estimator $\fn$ is assessed via its \textit{consistency}:
\begin{definition}[Consistency]
An estimator $f_n$ is \textbf{consistent} when  $\lim\limits_{n \to \infty} \mathcal{R}(\fn) = 0$.
\end{definition}

%via its Given an estimator , we are interested in evaluating
%\es{compacter un peu?}

\paragraph{Estimator} A Random Forest (RF) is a predictor consis\-ting of a collection of $M$ randomized trees \citep[see][for details about decision trees]{breiman1984classification}. To build a forest, we generate $M \in \mathds{N}^{\star}$  independent random variables $(\Theta_1, \hdots, \Theta_M)$, distributed as a generic random variable $\Theta$, independent of $\mathcal{D}_n$. In our setting, $\Theta_j$ actually represents the successive  random splitting directions and the resampling data mechanism in the $j$-th tree. The predicted value at the query point $x$ given by the $j$-th tree is defined as
% \begin{align*}
%     m_n(x, \Theta_j, \mathcal{D}_n) = \displaystyle \sum_{i=1}^n \frac{\ind{X_i \in A_n(x, \Theta_j, \mathcal{D}_n)} Y_i}{N_n(x, \Theta_j, \mathcal{D}_n)} \ind{\mathcal{E}}
% \end{align*}
\begin{align*}
    \fn(x, \Theta_j) = \displaystyle \sum_{i=1}^n \frac{\ind{X_i \in A_n(x, \Theta_j)} Y_i}{N_n(x, \Theta_j)} \ind{N_n(x, \Theta_j)>0}~ ,
\end{align*}
where $A_n(x, \Theta_j)$ is the cell containing $x$ and $N_n(x, \Theta_j)$ is the number of points falling into $A_n(x, \Theta_j)$. The (finite) forest estimate then results from the aggregation of $M$ trees:
\begin{align*}
    \forest (x, \boldsymbol{\Theta}_M) = \frac{1}{M} \displaystyle \sum_{m=1}^M \fn(x,\Theta_m)~, 
\end{align*}
where $\boldsymbol{\Theta}_M := (\Theta_1, ..., \Theta_M)$. 
By letting $M$ tending to infinity, we can consider the \emph{infinite} forest estimate, $\infforest(x) = \mathds{E}_\Theta[\fn(x, \Theta)],$ which has also played an important role in the theoretical understanding of random forests \citep[see][for more details]{scornet2016asymptotics}.
Here, $\mathds{E}_\Theta$ denotes the expectation w.r.t.\ $\Theta$, conditional on $\mathcal{D}_n$.   

Several random forests have been proposed depending on the type of randomness they contain (what $\Theta$ represents) and the type of decision trees they aggregate. Breiman forest is one of the most widely used RF, which exhibits excellent predictive performances. Unfortunately, its behavior is difficult to theoretically analyze, because of the numerous complex mechanisms involved in the predictive process (data resampling, data-dependent splits, split randomization). 
Therefore, in this paper, we simultaneously study the consistency and interpolation properties of different simplified versions of RF, both adaptive (i.e.\ when trees are built in a data-dependent manner) and non-adaptive. 
%the classical CRF and the Kernel Centered Random Forest (KeRF) \cite{scornet2016random}) and adaptive.

All forests include a \textit{depth} parameter, denoted $k_n$, which limits the maximum length of each branch in a tree, thus limiting the number of leaves (up to $2^{k_n}$). In this work, we analyze how the tuning of $k_n$ allows us to 
%use the depth parameter to control the number of leaves in each tree and to
adjust the \textit{consistency} and \textit{interpolation} characteristics of the forest.  
The classical notion of (exact) interpolation is defined below. 
%or \textit{exact interpolation} can be defined as follows: 
\begin{definition}[(Exact) interpolation] \label{def:exact_interpolation}
An estimator $\fn$ is said to \textit{interpolate} if for all training data $(X_i, Y_i)$, we have $\fn(X_i) = Y_i$ almost surely. 
\end{definition}

Recall that the prediction of a single tree at a point $x$ is given by the average of all $Y_i$ such that $X_i$ is contained in the leaf of $x$. Therefore, each tree within a forest can be parameterized in order to interpolate: it is sufficient to grow the tree until pure leaves (i.e.\ leaves containing labels of the same values) are reached. In any regression model with continuous random noise, we have $Y_i \neq Y_j$ for all $i\neq j$ almost surely. Therefore, an interpolating tree is a tree that contains at most one point per leaf.   

As the final prediction of the random forest is made by averaging the predictions of all its trees, if all trees interpolate, the random forest interpolates as well. Consequently, throughout all the theoretical analysis, we consider RF built without sub-sampling: each tree is built using the whole dataset instead of bootstrap samples as in standard RF. We will discuss the empirical effect of bootstrap in Section \ref{sec:BreimanRF}.   

% \cb{virer la remarque étant donné ce qui est raconté plus haut ? }

\begin{remark}
In a classification setting, it is possible to obtain pure leaves  with more than one point per cell (see \cite{mentch2019randomization} for more details).
%In a classification setting, in view of estimating the probabilities of being in a class, interpolation occurs as soon as there is no diversity within each leaf, i.e.\ each leaf is pure containing points from a single class. Indeed, consider a degenerated setting such as $X$ uniformly distributed on $[0,1]^d$ and $Y \sim \mathcal{B}(p), p\in [1/2, 1)$ independent from $X$, then as soon as there are $t > 1$ points in a cell, the probability of not interpolating is greater than $1- ((1-p)^t + p^t) > 0$. In such a setting, interpolation a.s.\ occurs when there is at most one data point per cell, similarly to the regression setting. For a more detailed analysis of the classification setting, see \cite{mentch2019randomization}.
\end{remark}

%The Adaptive Centered Random Forest (AdaCRF) introduced in Section \ref{sec:AdaCRF} easily verifies the interpolation property, we study its consistency property in the same section. On the opposite, non-adaptive forests such as the original Centered Random Forest (CRF) can hardly reach exact interpolation. 

\section{Centered RF}
\label{sec:CRF}

We start our analysis of interpolation and consistency of RF with the simple yet widely studied Centered Random Forest \citep[CRF, see][]{biau2012analysis}. CRF are ensemble methods said to be non-adaptive since trees are built independently of the data: 
%We will prove that within the general regression framework introduced in Section \ref{sec:setting} the interpolation property comes at the cost of their consistency.
%A Centered RF is built without taking into account the dataset: 
at each step of a centered tree construction, a feature is uniformly chosen among all  possible $d$ features and the split along the chosen feature is made at the center of the current cell. Then, the trees are aggregated to produce a CRF. 
\replace{}{Although simpler, the study of the mechanisms at hand in non-adaptive RF already provides good insights about the inner behaviour of more general RF.}

\subsection{Interpolation in CRF} 
%For CRF, forest interpolation is equivalent to tree interpolation, as shown below. 
\begin{lemma} \label{lem:forest_interp_iff_tree}
The CRF $\forest^{\mathrm{CRF}}$ interpolates if and only if all trees that form the CRF interpolate.
\end{lemma}

Since CRF construction is non-adaptive, it is impossible to enforce exactly one observation per leaf. Hence trees do not interpolate and in turn, the interpolation regime (De\-fi\-ni\-tion~\ref{def:exact_interpolation}) cannot be satisfied for CRF. This leads us to examine a weaker notion of interpolation in probability.

\begin{proposition}[Probability of interpolation for a centered tree]
\label{prop:tree_proba_interpolation}
    Denote $\mathcal{I}_T$ the event ``a centered tree of depth $k_n$ interpolates the training data". Then, for all $n\geq 3$, fixing $k_n = \lfloor {\log_2(\alpha_n n)} \rfloor$, with $\alpha_n \in \mathds{N}\setminus \{0,1\}$, one has
    %\replace{such that $\log_2(\alpha n) \in \mathds{N}$}{}, 
    \begin{align*}
        e^{-\frac{n}{\alpha_n - 1}}
        \leq 
        \Prob{\mathcal{I}_T}
        \leq 
        e^{-\frac{n}{2(\alpha_n+1)}}.
    \end{align*}
\end{proposition}

%See Appendix \ref{proof:ctree_proba_interpolation} for the proof. 
%Choosing $\alpha_n=1$, i.e.\ $k_n=\lfloor {\log_2 n} \rfloor$, is not sufficient to meet exact interpolation in the infinite sample size: indeed, the upper-bound on the probability of exact interpolation goes to $0$ as $n$ tends to infinity.
According to Proposition~\ref{prop:tree_proba_interpolation}, the probability that a tree interpolates tends to one if and only if $k_n = \lfloor {\log_2 (\alpha_n n)}\rfloor$ with $\alpha_n =\omega (n)$\footnote{i.e.\ $\alpha_n$ asymptotically dominates $n$.}. Consequently, the regime $\alpha_n =\omega (n)$ completely characterizes the interpolation of a centered tree.  
%The lower bound tells us that the choice $k_n = \lfloor {\log_2 (\alpha_n n)}\rfloor$ with $\alpha_n =\omega (n)$\footnote{i.e.\ $\alpha_n$ asymptotically dominates $n$.} makes the centered tree interpolate with probability one asymptotically, when $n$ tends to infinity. in the infinite sample size. This gives a necessary and sufficient condition for the interpolation of a centered tree in probability. 
Proposition \ref{prop:tree_proba_interpolation} can be in turn used to control the interpolation probability of a centered RF.

\begin{corollary}[Probability of interpolation for a CRF]
\label{cor:interpolation_CRF}
We denote by $\mathcal{I}_F$ the event ``a centered forest $\forest^{\mathrm{CRF}}(.,\boldsymbol{\Theta}_M)$ interpolates". Then, for $k_n = \integer{\log_2(\alpha_n n)}$ with $\alpha_n \geq 1$,
\begin{align}
\Prob{\mathcal{I}_F} \leq e^{-\frac{n}{2(\alpha_n + 1)}}. 
\end{align}
\end{corollary}
%See Appendix \ref{proof:crf_proba_interpolation} for the proof. 
%Similarly, the upper bound on the probability of interpolation in Corollary \ref{cor:interpolation_CRF} will tend to 1 if and only if $\alpha=\omega (n/2)$. 
Therefore, the condition $\alpha_n= \omega(n)$ (corresponding to the interpolation of a single centered tree with high probability) is necessary to ensure that w.h.p., the RF interpolates. %to Since a centered tree interpolates provided that $\alpha_n= \omega(n)$, 
%(with $k_n =  \log_2(\alpha_n n)$)
Our analysis stresses that a tree depth of at least $k_n=2\log_2(n)$ is required to obtain tree/forest interpolation. 
%Corollary~\ref{cor:interpolation_CRF} gives a necessary condition for CRF interpolation: 
%\replace{}{choosing $\alpha_n=n^{1+\eta}$ for $\eta >0$ leads to a depth $k_n = (2+\eta)\log_2(n)$. 
%This highlights the hinge value of the depth $k_n$ at $2\log_2(n)$ to meet the necessary condition of the CRF interpolation. }

In fact, choosing $k_n$ of the order of $\log_2(n)$ characterizes another type of interpolation regime. %as developed in what follows.
To see this, consider a centered tree of depth $k$, whose leaves are denoted $L_1,\dots,L_{2^k}$. The number of points falling into the leaf $L_i$ is denoted $N_n(L_i)$. Since $X$ is uniformly distributed over $[0,1]^d$, then, for all $i =1,\hdots , 2^k,$
\begin{align}
%\label{eq:proba_point_leaf_crf}
 \quad \Prob{X \in L_i} = \frac{1}{2^k} \quad
    \text{and} \quad \Esp{N_n(L_i)} = \frac{n}{2^k}. \label{eq:esp_point_leaf_crf}
\end{align}
%Therefore, if there are more leaves than the sample size $n$, the tree is expected to interpolate \textit{in expectation}, as specified by the following definition. %indeed, the mean number of points per cell is less than 1 (by Equation \ref{eq:esp_point_leaf_crf}). 
%This property is easier to satisfy than exact interpolation; it also allows more flexibility in the choice of the depth $k_n$ in order to reach consistency. 
%This motivates the introduction of the following weaker notion of interpolation.

\begin{definition}[Mean interpolation regime] A CRF $\forest^{\mathrm{CRF}}$ satisfies the {\emph{mean interpolation regime}} when each tree of $\forest$ has at least $n$ leaves, i.e.\  if and only if $k_n\geq \log_2 n$. 
% Indeed, when there are more leaves than the number of points, there is at most one point per leaf in expectation: for all $i\in \{1,\hdots , 2^k\}$, $\Esp{N_n(L_i)}\leq 1$.
\end{definition}

%%%%%%%%%%%%%%%%%%%%
%This phenomenon leads us to consider the following weaker notion of interpolation. 
%We will prove in Section \ref{sec:KeRF} the consistency of a non-adaptive tree-based method, the Kernel centered Random Forest (KeRF), under this regime of interpolation. \\

%The mean interpolation regime is met for CRF if and only if $k_n \geq \log_2 n$. 
By Equation \eqref{eq:esp_point_leaf_crf}, the mean interpolation regime implies that for all leaves $L_i$, $\Esp{N_n(L_i)}\leq 1$:
%, i.e., each leaf contains at most one point in expectation. 
one could say that trees interpolate in expectation, in the mean interpolation regime. 
%Equivalently, is met as soon as the depth $k$ of a tree is greater than $\log_2 n$ (by ).
%As the mean interpolation is less restrictive than the exact one (which occurs when each leaf contains at most one point), 
%this actually gives a lower-bound on the required depth for exact interpolation. 

%The upper bound of the probability that the CRF interpolates behaves similarly w.r.t $n$ and $\alpha$. Therefore, it is necessary (but maybe not sufficient) that $\alpha_n n \gg n$ for the CRF to interpolate with high probability, hence producing an infinite number of empty leaves which would increase the bias of the estimator as shown in the next proposition.
%This means that the trees should be 
%\label{section:CRF_inconsistency}

\subsection{Inconsistency of the standard CRF} 

In both interpolation regimes (mean and in probability), trees need to be very deep, with a growing number of empty cells as $n$ tends to infinity, eventually damaging the consistency of the overall CRF.

\begin{proposition}
\label{prop:crf_inconsistency}
Suppose that $\Esp{f^\star(X)^2} > 0$ and set $\alpha >0$. Then the infinite Centered Random Forest $\infforest^{\mathrm{CRF}}$ of depth {$k_n \geq {\log_2 \alpha n}$} is inconsistent.
\end{proposition}
%The proof is given in Appendix \ref{proof:crf_inconsitency}. 
Proposition~\ref{prop:crf_inconsistency} emphasizes the poor generalization capacities of the interpolating CRF (under any interpolating regime), which could be expected given its non-adaptive construction.  Indeed, the non-consistency of the CRF stems from the fact that the probability for a random point $X$ to fall in an empty cell does not converge to zero, introducing an irreducible bias in the excess risk.

\subsection{Consistency of void-free CRF under the mean interpolation regime} 
Since limiting the impact of empty cells seems crucial for consistency, we study a CRF that averages over non-empty cells only, which we call the \textit{Void-Free CRF}. Note that predictions in empty leaves are arbitrary set to 0. Denoting $\Lambda_n(x,\boldsymbol{\Theta}_M)$ the number of non-empty leaves containing $x$ in the forest with trees $\Theta_1, \hdots, \Theta_M$, the void-free CRF is written as 
\begin{align*}
    f_{M, n}^{\mathrm{VF}}(x,\boldsymbol{\Theta}_M) &= \frac{1}{\Lambda_n(x,\boldsymbol{\Theta}_M)}\sum_{m=1}^M f_n(x,\Theta_m) \ind{N_n(x,\Theta_m) >0}.
\end{align*}

The problematic terms that arise in the theoretical derivations of classical CRF vs.\ void-free CRF are of different natures: the probability $\Prob{N_n(X,\boldsymbol{\Theta}_M) = 0}$ of falling into an empty leaf in a random tree of an (infinite) CRF compared to the probability %$\Prob{\Prob{ N_n(X,\Theta) = 0 | X, \mathcal{D}_n } =1} =
$\mathbb{P} \left[{\forall m \in \{1, \hdots, M\}, N_n(X, \Theta_m) =0} \right]$ of falling into empty leaves in all trees in the (infinite) CRF. Lemma~\ref{lem:agg_crf_empty_cell} below controls this last term. 
\begin{lemma} \label{lem:agg_crf_empty_cell}
Consider a finite void-free CRF $f_{M, n}^{\mathrm{VF}}( \cdot ,\boldsymbol{\Theta}_M)$ of depth $k \in \mathds{N}$. Let $x \in [0,1]^d$ and denote $\mathcal{E}_{M,n}(x)$ the event ``for all $m \in \{1, \hdots, M\}, N_n(x, \Theta_m) =0$''.
%(i.e.\ $x$ falls into an empty cell in all trees). 
Then,
\begin{align}
    \Prob{\mathcal{E}_{M,n}(x)} \leq e^{-\frac{kn}{2^{k+1}}} + e^{-M d^{-k}}.
\end{align}
%It follows that for an infinite forest $f_{\infty, n}^{\mathrm{VF}}$, 
%\begin{align}
%    \Prob{\mathcal{E}_{\infty,n}(x)} \leq e^{\frac{- kn }{2^{k+1}}}.
%\end{align}
Consequently, if $k= \integer{\log_2(n)}$ and $\replace{}{M_n} = \omega(n^{\log_2 d})$, then $\lim\limits_{n \to \infty} \Prob{\replace{}{\mathcal{E}_{M_n,n}}(x)} = 0.  $
%\begin{align}
%\lim\limits_{n \to \infty} \Prob{\mathcal{E}_{\infty,n}} = 0.    
%\end{align}
\end{lemma}
As previously, the infinite void-free CRF is defined as 
%\begin{align*}
 $   f_{\infty, n}^{\mathrm{VF}}(x) = \mathbb{E}_{\Theta} \left[f_n(x,\Theta) | N_n(x,\Theta)>0 \right]
    %&=  \mathbb{E}_{\Theta} \left[ \frac{f_n(x,\Theta) \ind{N_n(x,\Theta)>0}}{\mathbb{P}_\Theta\left(N_n(x,\Theta)>0\right)} \right]
    $.
%\end{align*}
\begin{theorem} \label{th:agg_crf_consistency} \replace{}{Assume} that $f^\star$ has bounded partial derivatives. Then, the infinite void-free-CRF of depth $k=\integer{\log_2 n}$ is consistent in a noiseless setting ($\sigma=0$), and, for all $n >1$,
%\begin{align*}
 %   \mathcal{R}\left(f_{\infty, n}^{\mathrm{VF}}(X)\right) &\leq 16d \displaystyle \sum_{j=1}^d ||\partial f_j^\star ||_\infty^2 \left( n^{2\alpha \log(2) \log\left(1-\frac{1}{2d}\right)} + e^{-\frac{n^{1-\alpha}}{2}}\right) \\
 %   &+ 2 n^{-\frac{1}{\log(2)}}.
%\end{align*}
\begin{align*}
    & \mathcal{R}\left(f_{\infty, n}^{\mathrm{VF}}(X)\right) 
      \leq & C_d  \left( \frac{n}{\log_2 n} \right)^{2 \log_2 \left( 1 - \frac{1}{2d}\right)}  
     + \left(  C_d +2 \right) n^{-1/(2 \ln 2)}, 
\end{align*}
where $C_d = 4d  \left(\sum_{j=1}^d ||\partial f_j^\star ||_\infty^2 \right).$
\end{theorem}

The overall rate is of order $O\left(n^{2 \log(1-1/2d)}\right)$ which is a typical approximation rate for CRF, see \citet{klusowski2021sharp}. As a matter of fact, Theorem \ref{th:agg_crf_consistency} highlights that empty cells do not limit the performance of the void-free-CRF in the mean interpolation regime.

%Compared to the classical CRF that aggregates all leaves, void-free CRF only combine non-empty leaves. 

%The first term is shown to be of constant order (Proposition~\ref{prop:crf_inconsistency}), whereas the second term tends to zero (Lemma~\ref{lem:agg_crf_empty_cell}), which offers a more precise control of the void-free CRF, compared to classical CRF. This highlights the importance of the aggregation scheme for the RF consistency.

%Neglecting empty cells allows CRF to be consistent in the mean interpolation regime. 
However, this construction introduces a conditioning over $N_n(x,\Theta)>0 $ that prevents us from bounding the variance in the case of noisy samples. Therefore, in the next section, we analyze Centered Kernel RF (KeRF) with a different aggregation rule (empty cells still being neglected).
%empty cells are not taken into account for the prediction at $x$ (unless all cells containing $x$ are empty across all trees in the forest).
%for any tree of the forest, the cell containing $x$ is empty). 
%Regarding adaptive random forests, their adaptivity feature will inherently limit the number of empty cells as shown in Section \ref{sec:AdaCRF}.

\section{Centered kernel RF}
\label{sec:KeRF}

As formalized in \cite{geurts2006extremely} and developed in \cite{arlot2014analysis}, slightly modifying the aggregation rule of tree estimates provides a kernel-type estimator. Instead of averaging the predictions of all centered trees, the construction of a Kernel RF (KeRF) is performed by growing all centered trees and then averaging along all points contained in the leaves in which $x$ falls, i.e.\ 
\begin{align*}
    \forest^{\mathrm{KeRF}}(x, \boldsymbol{\Theta}_M) :=  \frac{\sum_{i=1}^n Y_i \sum_{m=1}^M \ind{X_i \in A_n(x,\Theta_m)}}{\sum_{i=1}^n \sum_{m=1}^M \ind{X_i \in A_n(x,\Theta_m)}}.
\end{align*}

One of the benefits of this construction is to limit the influence of empty cells, which can be harmful both for consistency and interpolation (see Section \ref{sec:CRF}). As earlier, the infinite KeRF is defined as, 
\begin{align*}
%\forall x \in [0,1]^d, \qquad
\infforest^{\mathrm{KeRF}}(x) = \frac{\sum_{i=1}^n Y_i K_{n}(x,X_i) }{\sum_{i=1}^n K_{n}(x,X_{i})}, 
\end{align*}
where $K_{n}(x,z) = \P_{\T}\left[ z \in A_n(x, \Theta) \right]$ is the probability that $x$ and $z$ are in the same cell w.r.t.\ a tree built according to $\Theta$ \citep[see][for details]{scornet2016random}.

%Letting $K_{M, n}$ be the connection function of the $M$ finite forest defined by
%\begin{align*}
%    K_{M, n}(\x,\z) := \frac{1}{M} \sum_{j=1}^M \mathds{1}_{ \z \in A_n(\x, \Theta_j)}, %\label{definition_noyau_fini}
%\end{align*}
%\cite{scornet2016random} shows that the KeRF can be rewritten as
%\begin{align*}
%    \forest^{\mathrm{KeRF}}(\x, \boldsymbol{\Theta}_M) = \frac{ \sum_{i=1}^n Y_i K_{M, n}(\x,{\bf X}_i) }{ \sum_{i=1}^n  K_{M, n}(\x,{\bf X}_{i})}, % \label{weightedapproximation}
%\end{align*}
%hence the name of kernel RF. In addition, it is shown that 
%\begin{align*}
%\lim\limits_{M \to \infty} K_{M, n}(x,z) := K_{n}(x,z),
%\end{align*}
%where $K_{n}(x,z) = \P_{\T}\left[ z \in A_n(x, \Theta) \right]$ which is the probability that $x$ and $z$ are in the same cell w.r.t.\ a tree built according to $\Theta$. %When $M$ tends to $\infty$, this empirical probability tends to the theoretical probability of connection of the infinite KeRF, so that 
%Consequently, for all $x \in [0,1]^d$, the infinite KeRF reads as
%\begin{align*}
%\infforest^{\mathrm{KeRF}}(x) = \frac{\sum_{i=1}^n Y_i K_{n}(x,X_i) }{\sum_{i=1}^n K_{n}(x,X_{i})}.
%\end{align*}
%where $K_{k}(x,z)$ is its connection function evaluated between $x$ and $z$.

%\subsection{Interpolation Conditions}

\paragraph{Interpolation conditions} Since KeRF aggregates centered trees as CRF (but in a different way), the results of Section \ref{sec:CRF} can be extended to KeRF: $(i)$  the mean interpolation regime is met for centered trees (hence for KeRF) when $k_n \geq \log_2n$; $(ii)$ a necessary condition to attain the KeRF interpolation in probability is  $k_n>2\log_2(n)$. One can note that the depths required for both interpolation regimes are still large, leading to as many empty cells for KeRF as for classical CRF but the aggregation rule is such that they are not taken into account in KeRF predictions, which gives hope that consistency could be preserved.
%\es{à compacter? enelver si pas interpolation en proba}

%\subsection{Consistency}

\paragraph{Consistency}
We study the convergence of the centered KeRF under the \emph{mean interpolation regime}. 
To this end, we consider extra hypotheses on the noise and on the regularity of $f^\star$. 
% \begin{description}
% \item[(H)] The additive noise $\varepsilon$ is Gaussian with a finite variance $\sigma^2 < \infty$ and $f^\star$ is Lipschitz continuous.
% \end{description}

%\begin{align}
%\label{hyp:setting}
%\tag{H}
%%Y = f^\star(X) + \varepsilon,
%end{align}
%where $\varepsilon$ is a centered Gaussian noise, independent of $\bX$, with finite variance $\sigma^2<\infty$. Moreover, $m$ is Lipschitz. The gaussian hypothesis is made to control the distribution queue when upper bounding the approximation error in \ref{proof:th_consistency_kerf_approx}.

%\cb{should be ok for $k_n=\Theta(\log_2 (n))$, i.e. s'il existe a et b tel $a\log_2(n) \leq k_n \leq b \log_2(n)$, alors telle vitesse de CV }
%\la{En inf on peut descendre jusqu'a la profondeur choisie par Erwan par exemple, mais en sup $b\log_2n$ est trop gros pour tout $b$, on peut rajouter ni facteur multiplicatif ni puissance.}
\begin{theorem}
\label{theoreme_consistency_centred_forest_approximation}
%Assume that \textbf{(H)} is satisfied. 
Assume that $f^\star$ is Lipschitz continuous and that the additive noise $\varepsilon$ is a centered Gaussian variable independent from $X$ with finite variance $\sigma^2$.
Then, the risk of the {infinite} centered KeRF of depth $k_n = \integer{\log_2(n)}$  verifies, for all $d > 5$, for all $n$ large enough,
%\begin{align*}%
%   \mathcal{R} &(\widetilde{m}_{\infty,n}(x)) \leq C_4 \Big(\log(n)^2 w_n + \frac{\log(n)^2 n^{-1/d}}{2}  \\
%   &+  n^{-\log_2(d-2)} (\log n)^{d+5/2}\Big)^{1/3} +  C_4 n^{2\log(1-\frac{1}{2d})}
%\end{align*}
\begin{align*}
    \mathcal{R}(\infforest^{\mathrm{KeRF}})
    %&\leq \E \Big[ \infforest^{\mathrm{KeRF}}(\x) - f^\star(x) \Big]^2 
    &\leq 8L^2 d^2 n^{2\log_2(1-\frac{1}{d})} 
     + C_d (\log_2 n)^{-\frac{d-5}{6}} (\log_2(\log_2 n))^{d/3},
    %\left(\log(n)^{-(d-5)/6} +  \log(n)^{2/3} n^{-1/3d}  \right. \\
    %&+ \left. \log(n)^{-\frac{d-11}{6}} + n^{-\log_2(d-2)} (\log n)^{d+5/2} \right)
    %&= O \left(\log(n)^{-\frac{d-5}{6}} \right)
\end{align*}
where $C_d>0$ is a constant dependent on $\sigma^2$ and made explicit in the proof. 
%with $C_d >0$ a constant depending on $\sigma, d, \|f^\star\|_{\infty}$. %and $w_n \sim \log(n)^{-(d-1)/2}$.
\end{theorem}

%See Appendix \ref{proof:th_consistency_kerf_approx} for t
Theorem~\ref{theoreme_consistency_centred_forest_approximation} states that the infinite centered KeRF estimator is consistent as soon as $d > 5$, with a slow convergence rate of $\log(n)^{-(d-5)/6}$. The proof is based on the general paradigm of bias-variance trade-off and is adapted from \cite{scornet2016random}. 
At first sight, one might think that the rate becomes better as the dimension $d$ increases. However, the constant term highly depends on the dimension, so that the established bound should be regarded for a fixed $d$. 

%Remark that although the term limiting the rate decreases as $d$ increases, this is not the case for the other terms (such as $C_d$ for instance). Furthermore, as the constant is exponential in $d$, the bound is only valid for $d$ fixed.

Choosing $k_n =\integer{\log_2(n)}$ in Theorem \ref{theoreme_consistency_centred_forest_approximation} allows us to have a mean interpolation regime concomitant with consistency for KeRF, therefore highlighting that consistency and mean interpolation are compatible.
This is not the case for CRF for which the mean interpolation regime forbids convergence (Proposition \ref{prop:crf_inconsistency}).
If a ``mean'' overfitting regime is benign for the consistency of KeRF, 
it seems to be nonetheless malignant for the convergence rate.
Indeed, \citet{lin2006random} provides a lower bound on the convergence rate of a deep non-adaptive RF (such as the CRF), scaling in $(\log n)^{-(d-1)}$.
This leads us to believe that the convergence rate we obtain in Theorem~\ref{theoreme_consistency_centred_forest_approximation} is marginally improvable.
%has a convergence rate at least greater than $t \cdot \log(n)^{-(d-1)}$ where $t$ is the number of points within each leaf (in our case $t \approx 1$). This lower-bound is actually dictated by the variance term. Therefore, the \textit{optimal} convergence rate for an interpolating CRF (and KeRF) is slower than $(\log n)^{-(d-1)}$. 
%One may argue that the KeRF is expected to perform better in interpolating regimes than the standard CRF,  both in terms of bias (as the empty cells do not count) and variance (as the prediction at a point $x$ is made by averaging along all the training points connected to $x$ within all trees).
%This certainly suggests that the rate in Theorem \ref{theoreme_consistency_centred_forest_approximation} is marginally improvable.

Interpolation of kernel estimators has been recently studied with singular kernel by 
\cite{belkin2019does}. Since KeRF are kernel estimators, one can wonder how sharp is our bound (Theorem \ref{theoreme_consistency_centred_forest_approximation}) compared to that of \cite{belkin2019does}, which is minimax.
Due to the spikiness of the singular kernel studied in \cite{belkin2019does}, interpolation arises for any kernel bandwidth.  
The latter can be then tuned to reach minimax rates of consistency. 
The story is totally different for KeRF since interpolation occurs only for specific tree depths $k_n \geq \log(n)$ (where the depth parameter is closely related to the bandwidth of classical kernel estimates). Less latitude for choosing the depth then leads to sub-optimal rates of consistency (see Theorem \ref{theoreme_consistency_centred_forest_approximation}). 
Of course, a better rate of consistency in $O(n^{1/(3+d\log 2)})$ could be obtained as in \cite{scornet2016random} when optimizing this depth parameter, but leaving the interpolation world.

%\label{sec:kerf_empirical}
We numerically assess the performance of KeRF in the mean interpolation regime (see Appendix \ref{sec:exp_setting}).

\section{Semi-Adaptive RF: Median RF}
\label{sec:AdaCRF}
So far, consistency has been analyzed in the mean interpolation regime. What about consistency with exact RF interpolation? 
To analyze this phenomenon, we thus introduce semi-adaptive RF, Median RF, whose constructions depend on 
the training inputs $X_i$'s (and not on the outputs $Y_i$'s).

The Median RF, studied e.g.~in \cite{duroux2018impact, klusowski2021sharp}, is composed of median trees that first randomly choose the direction to cut over and then cut at the median of the data points contained in the current cells. \replace{}{In our analysis, for any cell containing $n_c$ observations, the median is set as the middle of the segment of two consecutive order statistics: $X_{(n_c/2)}$ and $X_{(n_c/2+1)}$ for an even number of observations, $X_{(\frac{n_c-1}{2})}$ and $X_{(\frac{n_c+1}{2})}$
otherwise.}

\subsection{Consistency}

\replace{}{In order to obtain consistency for an adaptive RF, one needs to control two terms: the bias and the variance terms. On the one hand, the bias is roughly controlled by the diameter of the leaf times the supremum of the derivatives of $f^\star$ in the leaf. In the interpolating regime, the depth is maximum so the diameter of the leaf is minimum and therefore the bias is smoothly upper bounded.}

{On the other hand, a ``low" depth regime is usually required to control the variance term, so that each leaf of each tree contains an infinite number of points when $n$ tends to $+\infty$. This directly ``averages the noise out" and decreases the variance towards 0 within each tree. However, in the interpolating case, each leaf contains only one point and we can only rely on the averaging effect of the RF, induced by the random splitting mechanism, to upper bound the variance.} 
\replace{}{ Studying the effect of the random splitting mechanism in full generality remains challenging. However,
as increasing the dimension also increases the diversity of the trees within the RF, it should naturally be easier to control the variance of an interpolating Median RF in an asymptotic high-dimensional setting, as we prove and discuss in Appendix \ref{subsec:var_high_dim}.}

\replace{}{The following theorem establishes the consistency of the interpolating Median RF in the general setting of noisy data and fixed input dimension.}

\begin{theorem}
\label{th:Median_RF_consistency}
\replace{}{Suppose that $f^\star$ has bounded partial derivatives and that $n$ is a power of two. % $\partial_j f^\star$ for all $j$
Then, the infinite interpolating Median RF $\infforest^{\mathrm{MedRF}}$ is consistent and verifies:
\begin{align*}
    \mathcal{R}\left(\medrf\right) &\leq C_{1} d \left(\sum_{\ell =1}^d||\partial_\ell f^\star||_\infty^2 \right) \left( 1 - \frac{3}{4d} \right)^{\log_2 n}  + \sigma^2  C_{2,d} (\log_2 n)^{- (d-1)/2},
\end{align*}
where 
$C_{1}$ and $C_{2,d}$ are explicit constants, the former being independent of the dimension $d$ (see the proof for the exact computations). 
%$ \leq 256 \exp \left( \frac{ 55 + \sqrt{5}}{2 - \sqrt{2}} \right)$ and $C_{2,d} = \left(32~\exp \left(  \frac{10}{\sqrt{2}-1} \right)~\right)^d d^{d/2}$. 
}
\end{theorem}
\replace{}{The control of the bias term follows the general approach used in \cite{duroux2018impact} with substantial technical refinements. On the other hand, we propose a more general approach for the control of the variance inspired by \cite{biau2012analysis,klusowski2021sharp}, where we derive explicit bounds specifically designed for Median RF.
Note that the consistency achieved by Median RF cannot be obtained for CRF under the interpolation regime due to the non-negligible probability of falling into empty cells (see Proposition \ref{prop:tree_proba_interpolation}).
%and \ref{prop:inconst_adacrf} resp.\ 
% and AdaCRF).
}

\replace{}{Theorem \ref{th:Median_RF_consistency} is the first result ensuring consistency of RF despite exact interpolation. It is even more impressive considering that bootstrap is off so that the averaging process in the RF is only due to feature subsampling. 
%This result clarifies the RF behavior by showing the power of the simple yet very efficient averaging procedure at hand in the RF.
More specifically, when dealing with interpolating trees, the variance reduction does not come from averaging many points in the leaf of a given tree anymore (since the tree depth is no longer limited), but results from averaging single points from the leaves of different trees. } 

\replace{}{
% Commentaire sur la vitesse de convergence
If interpolation remains compatible with consistency in the case of Median RF, it nevertheless damages the convergence rate. Indeed, it has been proved  that, for all $\alpha$ small enough, the convergence rate of Median RF with trees of depth $k = (1 - \alpha) \log_2(n)$ is $n^{-\alpha}$ \citep[see Theorem 3 in][]{klusowski2021sharp}. In the case of interpolating Median RF, Theorem~\ref{th:Median_RF_consistency} highlights a phase transition when $k = \log_2(n)$, as the convergence rate is driven by the variance term, which is of order $(\log_2 n)^{- (d-1)/2}$. While being very slow, this rate is close to the lower bound $(\log_2 n)^{- (d-1)}$ established for non-adaptive interpolating RF \citep[][]{lin2006random}. Actually, by assuming $\log_2(n) \geq d$, our proof can be directly modified so that our upper bound matches the lower bound of \citet{lin2006random} \citep[using the second statement of Lemma S.1 in][instead of the first one]{klusowski2021sharp}. 

Note that Theorem \ref{th:Median_RF_consistency} does not contradict Proposition 1 in \cite{tang2018random}, as the condition therein is not proved to be satisfied for interpolating median RF (nor for interpolating CRF). 
}

We also provide numerical experiments (resp.~Section \ref{sec:interp_zone_adacrf} and \ref{sec:median_rf_empirical}) that illustrate the consistency of the interpolating Median RF.
%On the theoretical side, when dealing with interpolating trees, the variance reduction cannot come from averaging many points in the leaf of a given tree anymore, but should result from averaging single points from the leaves of different trees. Hence, it requires to understand the geometry of the tree partitions and their intersections, which is an arduous task for generic adaptive RFs.
%even those with data-independent cuts.
%However, we reach a favorable conclusion regarding the variance reduction at work in interpolating median (and therefore semi-adaptive) RF, in the case of an asymptotic high-dimensional setting.

\subsection{Volume of the interpolation area}
\label{sec:interp_zone_adacrf}

%To go one step further in the analysis of Median RF, we introduce the notion of \textit{interpolation area} defined below.
In this section, we aim at quantifying the volume of the interpolation area of a Median RF, which is a prerequisite for the RF consistency. To pursue our analysis, we first give a rigorous definition of the interpolation area.   
\begin{definition}
The \textbf{\textit{interpolation area}} is the subspace of $[0,1]^d$ where the forest prediction depends only on one training point. For a given forest $\forest(., \boldsymbol{\Theta}_M)$, the interpolation area is denoted by\footnote{the symbol $\exists!$ means ``there exists a unique''.} 
\begin{align*}
\mathcal{A}(\forest(.,\boldsymbol{\Theta}_M)) 
&= \left\{x \in [0,1]^d, \exists! X_i \in \mathcal{D}_n, X_i \in \bigcap_{m=1}^M A_n(x, {\Theta}_m) \right\}.
\end{align*}
\end{definition}
The interpolation zone is highly dependent on both the geometry of the training points $X_i$'s and the construction of the trees. Analyzing the interpolation area for a finite Median RF turns out to be quite a challenging task. Therefore, we focus our study on the \emph{core interpolation area} $\mathcal{A}_{min}$ written as 
%$\mathcal{A}_{min}~=~\min_{\boldsymbol{\Theta}_M} \mathcal{A}(m_{M,n}(.,\boldsymbol{\Theta}_M))$, or equivalently,
\begin{align*}
    \mathcal{A}_{min} &= \bigcap_{M \in \mathds{N}, \boldsymbol{\Theta}_M} \mathcal{A}(\forest(.,\boldsymbol{\Theta}_M)).
\end{align*}
The area $\mathcal{A}_{min}$ is the intersection of the interpolation zones of all possible forests, or equivalently of a forest containing all possible trees (and therefore all possible cuts).
As an example note that 
in the case of median trees, every cut may occur with a positive probability.
Therefore, $\mathcal{A}_{min}$ matches the volume of the interpolation area of an infinite Median RF.
In the following proposition, we control the Lebesgue measure (denoted by $\mu$) of the core interpolation area $\mathcal{A}_{min}^{\mathrm{MedRF}}$ of an infinite Median RF. 
%defined by 
%$$\mu(\mathcal{A}_{min}) = \int_{[0,1]^d} \ind{\mathcal{A}_{min}}(x) dx_1 \hdots dx_d.$$
%The interpolation zone is the union of $n$ interpolation zones, each one containing a single $X_i$. We denote $\mathcal{A}(m_{M,n}(.,\boldsymbol{\Theta}_M)) = \mathcal{A}_{X_1,\boldsymbol{\Theta}_M} \cup  ... \cup \mathcal{A}_{X_n, \boldsymbol{\Theta}_M}$ with $\mathcal{A}_{X_i, \Theta_M} = \{x \in [0,1]^d, m_{M,n}(x,\boldsymbol{\Theta}_M) = Y_i\}$.
%We are interested in computing the volume of this area in order to quantify the \textit{direct influence} of the interpolation property on the estimation of the signal. 

\begin{proposition} \label{prop:vol_interp_zone_AdaCRF}
For all $n\geq 2$, for all $d\geq 2$, consider an infinite Median RF.
Then, 
%and its interpolation zone $\mathcal{A}_{m_{\infty,n}}$.
%\begin{enumerate}
    %\item \label{prop:vol_interp_AdaCRF1} $\mathcal{A}({m_{\infty,n}}) = \mathcal{A}_{min} \hspace{0.3cm}\text{a.s}.$
    \begin{align*}
        \displaystyle 
    \mathbb{E}_{\mathcal{D}_n}\left[\mu({\mathcal{A}}_{min}^{\mathrm{MedRF}})\right] \leq 2 \left(\frac{2}{n}\right)^{d-1}. \label{prop:vol_interp_AdaCRF2} 
    \end{align*} 
%\end{enumerate}
%Proof in Appendix \ref{proof:vol_interp_zone_AdaCRF}.

\end{proposition}

%With a non-zero probability, an adaptive centered tree can be fully grown by cutting along a single axis. Therefore, the volume of the minimum interpolation zone  can be computed by considering all trees that split along a single direction, therefore leading to the first statement in Proposition \ref{prop:vol_interp_zone_AdaCRF}.
The volume of the core interpolation area of an infinite Median RF tends to 0 polynomially in $n$ and exponentially in $d$.
%, and so does $\mathbb{E}_{\mathcal{D}_n} \left[\Prob{X\in \mathcal{A}_{min} } \right]$.

\begin{remark}
\label{rem:interpolationZone_medianRF}
Apart from a very restricted zone, the prediction of a Median RF 
%involves one single training point only in a negligible area and 
mostly relies on more than one training point.
More specifically, this is a \textit{necessary} condition for consistency: the volume of the area where the prediction involves only a finite number of points (\textit{a fortiori} the interpolation zone) should tend to 0. 
Indeed, by decomposing the risk as $R(f_n(X) \mathds{1}_{X \in{\mathcal{A}}_{min}^{\mathrm{MedRF}}} ) + R(f_n(X) \mathds{1}_{X \notin {\mathcal{A}}_{min}^{\mathrm{MedRF}}})$,  the first term is at least of the order $\sigma^2 \mu({\mathcal{A}}_{min}^{\mathrm{MedRF}})$. Therefore, 
it is not possible to cancel out the noise of the training dataset when only a finite number of points is used for the prediction. The noise in such an area remains of order $\sigma^2$.
Proposition \ref{prop:vol_interp_zone_AdaCRF} portends the predominant \textit{self-averaging} property of adaptive RF, and hence underpins the idea of good capabilities of Median RF in interpolation regimes.
\end{remark}

%except the first one which supports the claim that AdaCRF is consistent although converging at a slow rate. 

\section{Breiman RF}
\label{sec:BreimanRF}

The widely-used Breiman RF is composed of several CART \citep{breiman1984classification}, each one trained on a bootstrap sample, and for which the successive splitting directions and thresholds are chosen at each step (among a random subset of directions) in order to minimize the CART criterion. 
Breiman RF exhibit excellent predictive performance even if their adaptivity to the data remains a real hurdle to their theoretical analysis. 
%On the one hand, since their construction depends both on the input and output variables, this architecture is highly adaptive to the data. 
%On the other hand, 

From the interpolation perspective, each CART being trained on a bootstrap sample, the RF interpolation is not ensured when considering fully-grown trees. Indeed, a tree cannot interpolate a point that is not chosen in the bootstrap step. 
%These last two aspects represent a real hurdle for the theoretical analysis of Breiman forests in interpolation regimes. 
For this reason, we focus our study on the volume of interpolation areas for Breiman RF without bootstrap
and then analyze their empirical behavior in interpolating regimes through a battery of numerical experiments.

\paragraph{Interpolation}
As a Breiman RF is built using both the $X_i$'s and the $Y_i$'s, it is difficult to determine the depth necessary to reach the interpolation state. Depending on the data, the latter can be of the order $k \approx \log_2(n)$ in the best case, if each cut creates approximately two groups of the same size), or $k \approx n$ in the worst case, if only one point is separated from the others at each step \citep[low signal-to-noise ratios situations, see e.g.,][]{ishwaran2015effect}.
Note that by omitting the bootstrap in the RF construction, the interpolation of Breiman RF directly results from aggregating fully-grown trees. 

\paragraph{Volume of the interpolation zone}

As shown in the next proposition, the volume of the core interpolation area of Breiman RF tends to $0$ as $n$ tends to infinity.

\begin{proposition}
\label{prop:interpol_limit_zone} 
Consider an infinite Breiman forest constructed without bootstrap. 
Suppose that for a given configuration of the training data, all cuts have a probability strictly greater than 0 to appear.
Then, the volume of the minimal interpolation zone verifies
    \begin{align*}
        \Esp{\mu(\mathcal{A}_{min})} \leq \frac{1}{n^{d-1}} \left(1- 2^{-n} \right)^d.
    \end{align*}
\end{proposition}

%The proof is given in Appendix \ref{proof:interp_zone_breiman}.
Similarly to the Median RF, the bound on the interpolation volume for a Breiman forest enjoys the same order of decay, improved by a constant exponential in the dimension. 
Since predictions cannot be accurate in the interpolation area in a noisy setting, it is necessary that the volume of this area decreases to zero in order to ensure the RF consistency (see Remark \ref{rem:interpolationZone_medianRF}). 
Proposition \ref{prop:interpol_limit_zone} therefore suggests the good generalization properties of Breiman RF in interpolation regimes, as several training points are mostly used for prediction.
%Since the RF prediction suffers from a high sensitivity to the noise in the interpolation area, 

Setting the number of eligible features for splitting to $1$ is sufficient to ensure the hypothesis on cuts in Proposition \ref{prop:interpol_limit_zone}: one can obtain a tree in which all splits are performed along a single direction. This is a minor modification to the original algorithm,  easy to implement since most ML libraries have a ``max-feature" (as {scikit-learn} in Python) or ``mtry" (in {R}) parameter that can be set to 1.

In Appendix \ref{sec:exp_setting}, we numerically evaluate the volume of the interpolation zone and compare it to the theoretical bounds in Proposition \ref{prop:interpol_limit_zone}. 

%Similarly to the Adaptive Centered Random Forests, the volume of the interpolation zone of infinite Breiman RF decreases at a polynomial rate in $n$ and an exponential one in $d$. The same conclusions regarding its generalisation capacity and noise robustness hold: the influence of each data point remains local and interpolating should not harm the estimation of the signal globally.

\paragraph{Empirical study of consistency}
%For completeness, we also numerically assess the fast decay of the interpolation volume when the number $M$ of trees in the forest grows for a fixed sample size, see Figure \ref{fig:volume_interp_zone_tree_varying} in Appendix . 
 We now present an empirical study of Breiman RF consistency in interpolation regimes. 
In the theoretical analysis, we have focused on a specific type of Breiman RF (without bootstrap and a \textit{max-features} parameter equal to 1). We now examine the characteristics of Breiman forests with their default parameters and study the regularization processes that limit the noise sensitivity in the interpolation regime.
%To begin with, it is well known {\color{red} \textbf{[REF]}} that a single interpolating CART is inconsistent (see Figure \ref{fig:cart_inconsistency} in Appendix \ref{sec:exp_setting}).
%A popular saying says tree that interpolates, prediction that deteriorates.
In order to reach a better estimation of the regression function, Breiman RF average several CARTs while introducing randomness in the construction of each tree to diversify them. The first randomization comes from the bootstrap: each tree is trained on a bootstrap sample (selecting $n$ observations out of the $n$ original ones, with replacement). 
The other randomization results from a random selection of splitting directions: at each node, a subset of $\{1,\hdots, d\}$ of size max-features is randomly selected and the CART criterion is optimized along these directions only (setting \textit{max-features} to $1$ provides the maximum diversity whereas setting it to $d$ results in the construction of a unique tree).

The benefit of these two aspects in the construction of the Breiman RF is numerically analyzed when using interpolating Breiman trees. 
In Figure \ref{fig:default_breimadn_consistency}, we measure the excess risk of two RFs with {2000} trees and max-depth$=\texttt{None}$, where for the first one, bootstrap is used and the max-features parameter is set to 1, whereas the second one excludes bootstrap and sets the  max-features parameter to $\lceil d/3 \rceil$ (default value in \texttt{randomForest} in \texttt{R}).

\begin{figure}[ht]%[13]{r}{0.48\textwidth}
    \centering
    \includegraphics[width=0.8\textwidth]{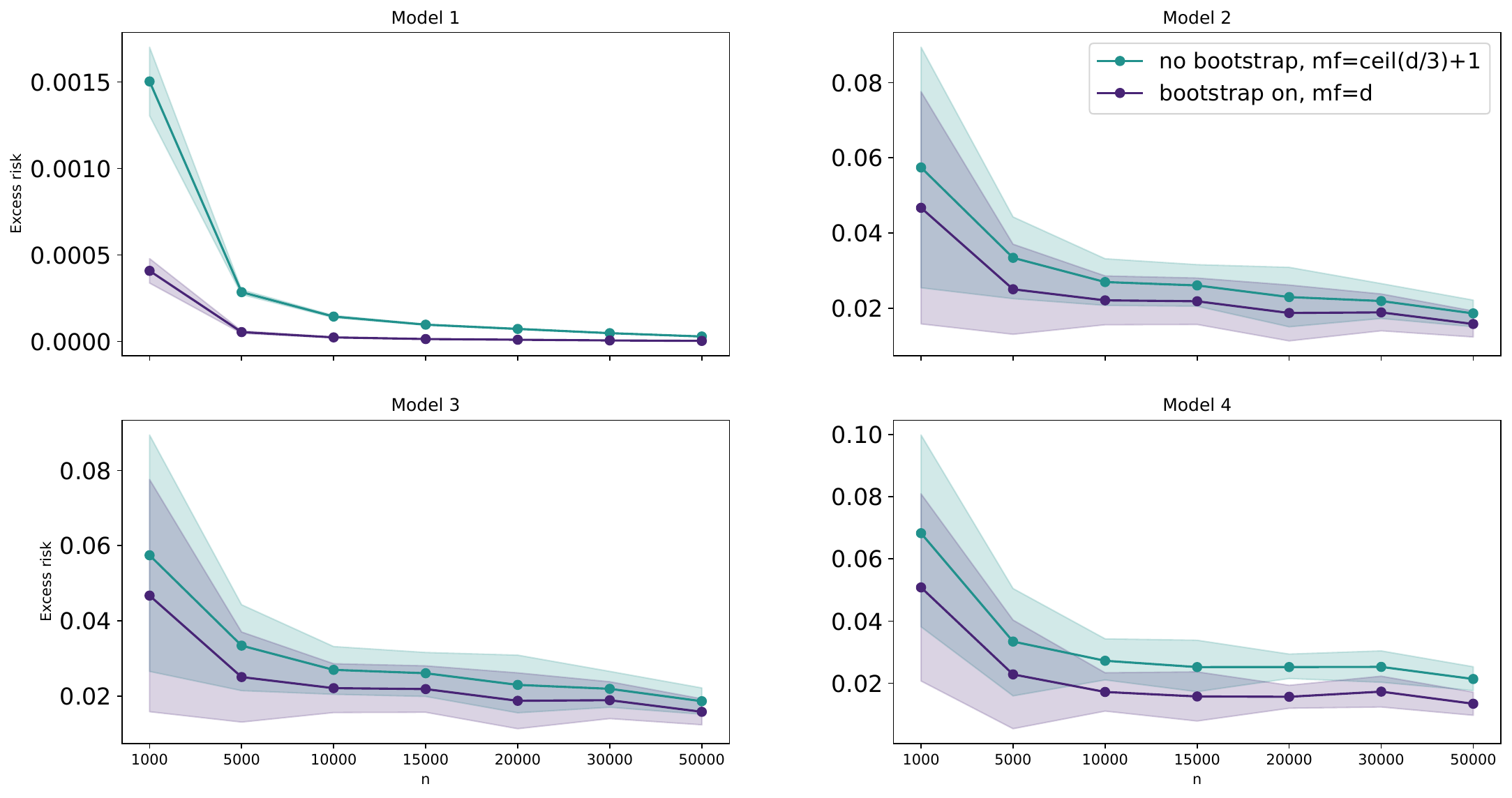}
    \caption{Consistency of two Breiman RF: excess risk w.r.t.\ sample size $n$.  %max-features$=d$ for the ``bootstrap on" RF, bootstrap off for the ``max-features$=\lceil d/3\rceil$" RF
    Mean over 10 tries (bold lines) and mean $\pm$ std (filled zone), when using 2000 trees per forest, and  max-depth$=$None. See Appendix \ref{sec:exp_setting} for the model definitions.}
    \label{fig:default_breimadn_consistency}
\end{figure}

In Figure \ref{fig:default_breimadn_consistency}, we observe that the excess risk decreases to $0$ for all models and for both forests. Indeed, each randomizing process alone induces enough diversity across trees for the self-averaging property to be efficient, resulting in the consistency of the overall forests \citep[see also][for insights about tree diversity in random forests]{scornet2016asymptotics, mentch2019randomization, mourtada2020minimax}.
%These results are in line with the message of \cite{} who explain that the success of RF estimators heavily relies on the randomisation processes at hand in the construction of the RF.

However, when using bootstrap, consistency comes at the cost of leaving the interpolation regime, as only $2/3$ of the data are used in average to build each tree (see Figures \ref{fig:breiman_train_interp_quant_n_var}, \ref{fig:breiman_train_interp_quant_tree_var} in Section \ref{sec:interp_loss_quant} for more details about the forest  non-interpolation).
In regards of this internal sampling selection, the aggregation of interpolating bagged trees results in smoothing the decision process of the entire forest, providing thereby a consistent but not interpolating estimate.

In turn, Breiman RF built with \textit{max-features}$=\lceil d/3\rceil$ seems consistent while preserving its interpolating behavior. 
%Diversity between the trees is only introduced by setting the max-feature parameter to $\lceil d/3 \rceil$ which seems to be a sufficient \textit{regularizer}: enough diversity between the trees allows the \textit{self-averaging} property of RF to be fully effective. 
Within this configuration, the final RF still interpolates the data but the volume of the interpolation zone is very small as shown in Figure \ref{fig:volume_interp_mf_tierce}. 
This is in line with the vision of a \textit{locally spiky} estimator developed in \cite{wyner2017explaining} and \cite{bartlett2021deep}. Indeed, the influence of the averaging effect is locally null near the data training points, but increases with the distance from these points.
Note that bootstrap and feature subsampling act differently. Bootstrap smoothens predictions by averaging different observations, even at points of the training set, which leads to an empty interpolation area. On the other hand, feature subsampling increases tree partition diversity, which reduces but does not annihilate the interpolation area of the overall forest.

\begin{remark}
One of the advantage of using deep (interpolating) trees is that it allows the RF to build more diversified trees. Indeed, the number of possible trees roughly grows exponentially with regard to the depth (also depending on $n,d$ and the \texttt{max-features} parameter). Especially when \texttt{max-features} is low, this should improve the averaging effect of the RF which is of particular interest when dealing with noisy data.
\end{remark}

In this regard, Breiman RF with \textit{max-features}$=\lceil d/3\rceil$ are similar to interpolating \textit{spiky} non-singular kernel methods, as studied in  \cite{belkin2019does}, %except for the tuned parameter. 
except for the leeway allowed for the hyperparameters tuning. 
Indeed, as underlined for non-adaptive centered forests, the depth $k_n$ (i.e.\ the tuned parameter) is constrained to a strict range to ensure both consistency and interpolation. This is not the case for singular kernel methods, as they interpolate regardless of the window parameter value. 

\section{Conclusion}

In this paper, we study both empirically and theoretically the tradeoff between interpolation and consistency of different types of random forests: when dealing with non-adaptive RF (CRF), empty cells prevent consistency; so that aggregating only non-empty leaves (void-free CRF) leads to convergence rates, only in a noiseless scenario. In a noisy setting, the kernel RF aggregates leaves differently (also avoiding empty ones). For kernel RF, we establish a (slow) consistency rate in the mean interpolation regime. 
%As shown in the first part, non-adaptive RF (forests whose construction is independent of the inputs) cannot achieve the exact interpolation regime. 
%We therefore created and analyzed a twisted adaptive version of CRF (called AdaCRF) whose construction is too rigid to reach consistency and exact interpolation simultaneously. 
\replace{}{We then study semi-adaptive RF that are closer to those used in practice and that present the advantage of being able to exactly interpolate the training data. The convergence of the median RF in the exact interpolation regime is established, showing the power of such architecture (even when used without bootstrap).}
%We also provide theoretical insights on the convergence of its variance term in a (noisy) high-dimensional setting. In this case, the split randomization indeed enhances the variability across trees.
Our study also shows that a prerequisite for consistency is that the minimal interpolation zone tends to zero as $n$ tends to infinity. We theoretically analyze this quantity for median and Breiman forests, emphasizing that interpolation might occur in conjunction with consistency if the volume of such areas vanishes fast enough. An experimental study supports the concomitance of consistency and interpolation in Breiman RF, when no bootstrap step is involved. 
%In this context, the consistency most probably results from the split randomization and occurs despite the exact interpolation regime.

Contrary to Nadaraya-Watson methods involving singular kernels that interpolate regardless of the bandwidth parameter, RF interpolate only for a specific choice of the depth, thus restricting the regime in which interpolation and consistency occur in concordance. 
Overall, most simple RF versions were relevant to study RF consistency when the tree depth was limited but are not actually sufficient to handle deeper trees corresponding to interpolation regimes. For adaptive forests, increasing the tree depth towards the interpolation regime results in a reduced bias, and the variance reduction phenomenon only results from the split randomization effect. The higher the dimension, the more diversified the trees, the stronger the averaging effect and the variance reduction.
Analyzing the strength of this phenomenon, which highly depends on the very shape of tree partitions, is the cornerstone to prove the consistency of adaptive RF in a general regression setting.
%\replace{}{It was successful in the case of Median RF but the case of Breiman RF remains very challenging.}
%
%Still, it remains a challenge to control this "variance reduction" as it requires to analyze complex tree partitions and their intersections in the case of adaptive RF.
%
We believe that interpolation remains benign for the consistency of adaptive RF, but can damage their convergence rate (this was the case for KeRF in the mean interpolation regime \replace{}{and for Median RF in the exact interpolation regime}), at least when bootstrap is not used.

%Breiman forests is also shown empirically to be consistent and interpolate when 
%{\color{gray}This analysis sheds a new light on the use of strong learners (i.e.\ fully grown trees) combined with the regularization mechanisms at hand in the RF. }
%In particular, we show that interpolation is harmless in the case of adaptive methods when the \textit{self-averaging} process in the forest is sufficient to restrain the interpolation effect to a local influence.

%Indeed, we prove that the Median RF reaches consistency and exact interpolation regimes in a noiseless scenario.\replace{}{ This is the first result to prove that consistency and interpolation are not irreconcilable for such powerful learners}. 
%This results from a fast decrease of the interpolation area, which limits the negative impact of interpolation on the overall consistency of the method.
%The Adaptive CRF which achieves a good compromise between data-dependence and simplicity of the construction should also be helpful for the theoretical analysis of more complex adaptive RFs such as Breiman RF.

The analysis of the interpolation zone of RF introduced in this article is an important tool for the understanding of RF prediction in interpolation regimes. Indeed the volume of the interpolation area is actually a roundabout way to measure the diversity in the constructed trees: if this volume is high, all trees end up building similar partitions. This diversity measure could also be used as a regularization tool to reduce the RF complexity by keeping only the most uncorrelated trees (in terms of partition) in a PCA fashion.    %\replace{}{We think that it could also be used as a way to measure the complexity of the RF.}

%Studying the impact of max-feature on the interpolation area volume is also a promising way to gain a better understanding of the role of this parameter in random forests. 
%This notion could be extended to study areas in which $q \in \{1, \hdots, n\}$ training points are averaged in the RF prediction, which could be used in turn to theoretically capture consistency of adaptive RF.

%\section*{Acknowledgements}

%The Ph.D of Ludovic Arnould is funded by SCAI (Sorbonne Center for Artificial Intelligence).

\bibliographystyle{plainnat}
\bibliography{biblio}

%%%%%%%%%%%%%%%%%%%%%%%%%%%%%%%%%%%%%%%%%%%%%%%%%%%%%%%%%%%%%%%%%%%%%%%%%%%%%%%
%%%%%%%%%%%%%%%%%%%%%%%%%%%%%%%%%%%%%%%%%%%%%%%%%%%%%%%%%%%%%%%%%%%%%%%%%%%%%%%
% APPENDIX
%%%%%%%%%%%%%%%%%%%%%%%%%%%%%%%%%%%%%%%%%%%%%%%%%%%%%%%%%%%%%%%%%%%%%%%%%%%%%%%
%%%%%%%%%%%%%%%%%%%%%%%%%%%%%%%%%%%%%%%%%%%%%%%%%%%%%%%%%%%%%%%%%%%%%%%%%%%%%%%
\newpage
\appendix
\onecolumn

\section{Summary of contributions}

\begin{figure}[h]
    \centering
    \includegraphics[clip, width=0.9\textwidth]{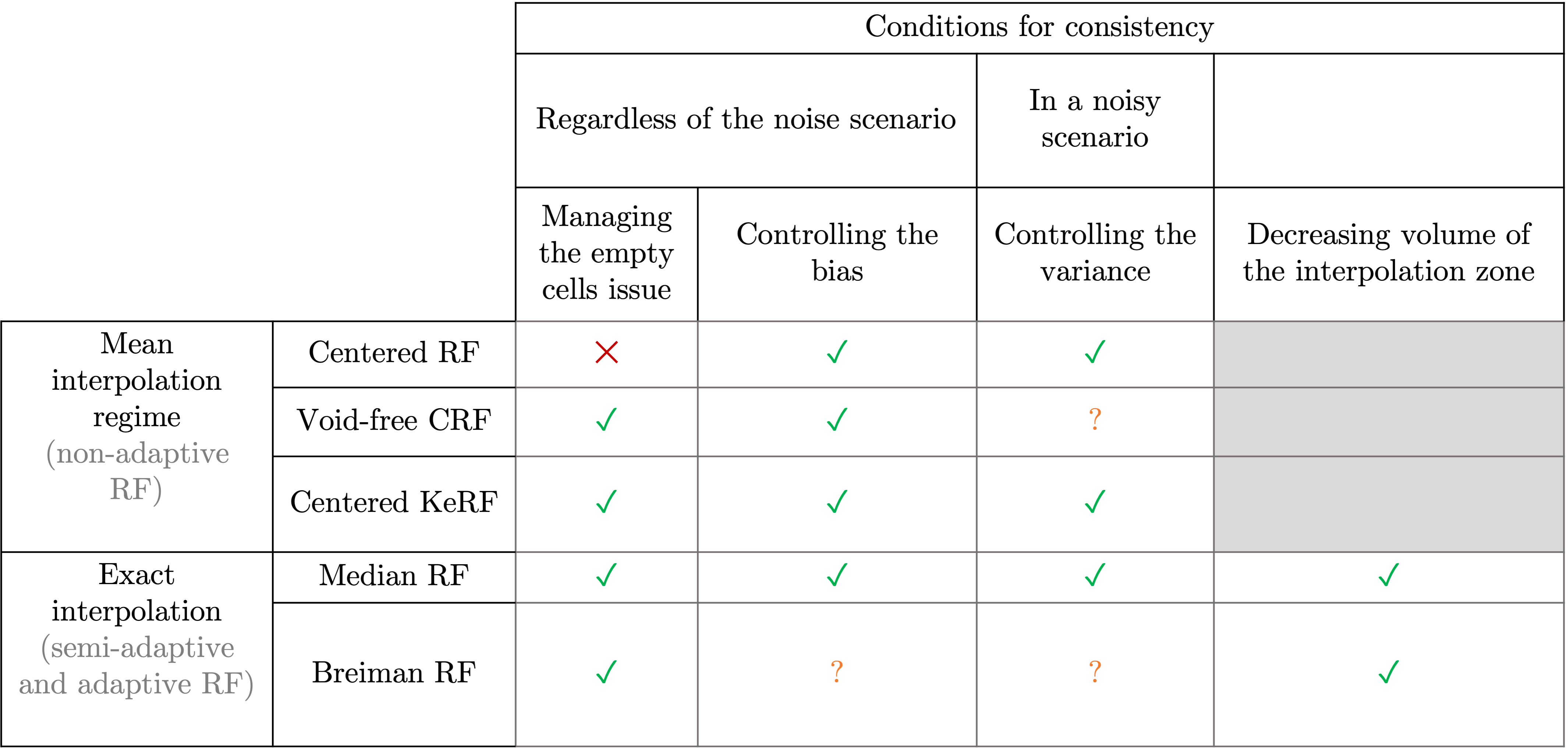}
    \caption{Summary of theoretical contributions}
\end{figure}

%%%%%%%%%%%%%%%%%%%%%%%%%%%%%%%%%%%%%%%%%%%%%%%%%%%%%%%%%%%%%%%%%%%%%%%%%
%%%%%%%%%%%%%%%%%%%%%%%%%%%%%%%%%%%%%%%%%%%%%%%%%%%%%%%%%%%%%%%%%%%%%%%%%

\section{Proofs}
\label{sec:proofs}

\subsection{Reminders and notations}

\textbf{Tree and RF estimator:}
We recall the prediction of the given by the $j$-th tree of the RF at point $x$:
% \begin{align*}
%     m_n(x, \Theta_j, \mathcal{D}_n) = \displaystyle \sum_{i=1}^n \frac{\ind{X_i \in A_n(x, \Theta_j, \mathcal{D}_n)} Y_i}{N_n(x, \Theta_j, \mathcal{D}_n)} \ind{\mathcal{E}}
% \end{align*}
\begin{align*}
    \fn(x, \Theta_j) = \displaystyle \sum_{i=1}^n \frac{\ind{X_i \in A_n(x, \Theta_j)} Y_i}{N_n(x, \Theta_j)} \ind{N_n(x, \Theta_j)>0}~ ,
\end{align*}
where $A_n(x, \Theta_j)$ is the cell containing $x$ and $N_n(x, \Theta_j)$ is the number of points falling into $A_n(x, \Theta_j)$. It is also written as follows:
\begin{align*}
    \fn(x, \Theta_j) = \displaystyle \sum_{i=1}^n W_{ni}(x,\Theta_j) Y_i,
\end{align*}
where $W_{ni}(x,\Theta_j) = \frac{\ind{X_i \in A_n(x, \Theta_j)} }{N_n(x, \Theta_j)} \ind{N_n(x, \Theta_j)>0}$.
The (finite) forest estimate then results from the aggregation of $M$ trees:
\begin{align*}
    \forest (x, \boldsymbol{\Theta}_M) = \frac{1}{M} \displaystyle \sum_{m=1}^M \fn(x,\Theta_m)~, 
\end{align*}
where $\boldsymbol{\Theta}_M := (\Theta_1, ..., \Theta_M)$.

\subsection{Proofs of Section \ref{sec:CRF} (Centered RF)}

\subsubsection{Proof of Lemma \ref{lem:forest_interp_iff_tree} (Link between tree and forest interpolation)}
%\begin{proof}[Proof of Corollary \ref{cor:interpolation_CRF}]

\label{proof:forest_interp_iff_tree}

First, it is clear that if all trees of a forest interpolate, the forest interpolates. Now, suppose that the forest $f_{M,n}^{\mathrm{CRF}}$ interpolates a training point 
$ X_s, s \in \{1,\hdots,n\}$.  
%Suppose that a tree $f_r$ in the forest does not interpolate for a given point $ X_s, s \in \{1,\hdots,n\}$, we write $f_r(X_s, \Theta_r) = Y_s + \xi$, $\xi \neq 0$. 
Then, by definition of $f_{M,n}^{\mathrm{CRF}}$,
\begin{align*}
        f_{M,n}^{\mathrm{CRF}}(X_s,\boldsymbol{\Theta_M}) &= \frac{1}{M}  \displaystyle \sum_{j=1}^M \sum_{i=1}^n Y_i  W_{n i}(X_s, \Theta_j) \\
        & = \sum_{i=1}^n Y_i \left( \frac{1}{M}  \displaystyle \sum_{j=1}^M   W_{n i}(X_s, \Theta_j) \right)\\
        & = Y_s,
        %&=  \frac{1}{M}  \displaystyle \sum_{j=1, j\neq r}^M \sum_{i=1}^n (f^\star(X_i) + \varepsilon_i)  W_{n i}(X_s, \Theta_j) + \frac{1}{M} (Y_s + \xi) 
\end{align*}
where $W_{n i}(X_s, \Theta_j) :=  \frac{\ind{X_i \in A_n(X_s, \Theta_j)}}{N_n(X_s, \Theta_j)} \ind{N_n(X_s, \Theta_j)>0}$.
Consequently, 
\begin{align}
  & f_{M,n}^{\mathrm{CRF}}(X_s,\boldsymbol{\Theta_M}) = Y_s\\
  & \Longleftrightarrow \quad  
  Y_s \left( \frac{1}{M}  \displaystyle \sum_{j=1}^M   W_{n s}(X_s, \Theta_j) - 1\right) + \sum_{i\neq s} Y_i \left( \frac{1}{M}  \displaystyle \sum_{j=1}^M   W_{n i}(X_s, \Theta_j) \right) = 0. \label{proof:forest_tree_interpolation_CRF}
\end{align}
For \eqref{proof:forest_tree_interpolation_CRF} to hold almost surely, it is necessary that it holds conditional on $X_1, \hdots, X_n, \Theta_1, \hdots, \Theta_M$. Since, for all $ j \in \{1, \hdots, M\}$, the terms $W_{ni}(X_s, \Theta_j)$ are measurable with respect to $X_1, \hdots, X_n, \Theta_1, \hdots, \Theta_M$ and $Y_s$ is independent of $(Y_i, i \neq s)$ given
$X_1, \hdots, X_n, \Theta_1, \hdots, \Theta_M$, equality \eqref{proof:forest_tree_interpolation_CRF} leads to, for all $i \neq s$, 
\begin{align}
 \frac{1}{M}  \displaystyle \sum_{j=1}^M   W_{n s}(X_s, \Theta_j) =1,  \quad \textrm{and} \quad  \frac{1}{M}  \displaystyle \sum_{j=1}^M   W_{n i}(X_s, \Theta_j) = 0.
\end{align}
Since all weights $W_{ni}(X, \Theta)$ take values in $[0,1]$, we have, for all $j \in \{1, \hdots, M\}$ and for all $i \neq s$
\begin{align}
W_{ns}(X_s, \Theta_j) = 1 \quad \textrm{and} \quad W_{ni}(X_s, \Theta_j) = 0.
\end{align}
Finally, for all $j \in \{1, \hdots, M\}$, the prediction of the $j$th tree at $X_s$ is given by 
\begin{align}
f_{n}^{\mathrm{CRF}}(X_s,\Theta_j) & = \sum_{i=1}^n W_{ni}(X_s, \Theta_j) Y_i \\
& = Y_s,
\end{align}
and therefore all trees of the forest interpolate the point $X_s$.

\subsubsection{Proof of Proposition \ref{prop:tree_proba_interpolation} (Probability of interpolation for a centered
tree)}
%\begin{proof}[Proof of Proposition \ref{prop:tree_proba_interpolation}]

\label{proof:ctree_proba_interpolation}
As all the leaves have the same volume and the data points are independent and uniformly distributed, having at most one point per leaf is equivalent to distribute $n$ balls into $2^k$ boxes containing at most one point with $2^k \geq n$ as can be seen on Figure \ref{fig:tree_balls}.
\begin{figure}[h]
    \centering
    \includegraphics[trim={2cm 15cm 2.5cm 3.6cm}, clip, width=0.6\textwidth]{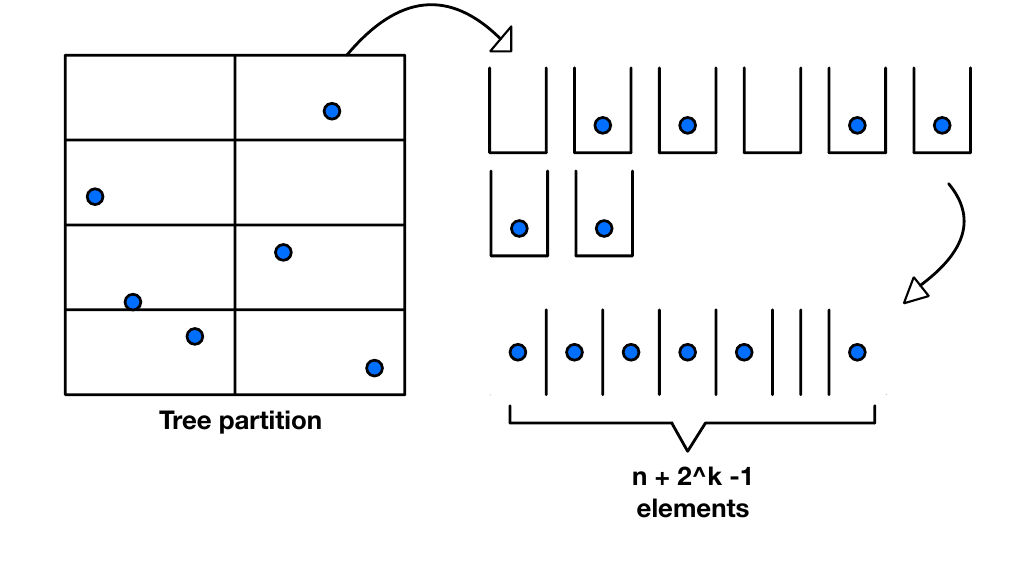}
    \caption{Computing the interpolation probability (depth $k=3$, $n=6$)}
    \label{fig:tree_balls}
\end{figure}
Recalling that $\mathcal{I}_T$ is the event ``a centered tree of depth $k_n$ interpolates the training data", we have 
\begin{align*}
\Prob{\mathcal{I}_T}
&= \frac{\binom{2^k}{n}}{\binom{n+2^k-1}{n}} \\
&= \frac{2^k!}{(2^k-n)!n!} \frac{n!(2^k-1)!}{(n+2^k-1)!} \\
& = \frac{2^k  \times (2^k - 1) \times \hdots \times (2^k - n +1)}{(2^k+n-1) \times (2^k + n - 2) \times \hdots 2^k}.
\end{align*}
If we have $k = \log_2 (\alpha_n n) \in \mathds{N}$, we have
\begin{align*}
\Prob{\mathcal{I}_T}
&= \frac{\alpha_n n}{(\alpha_n+1)n - 1} \cdot \frac{\alpha_n n-1}{(\alpha_n+1)n-2} \hdots \frac{(\alpha_n-1)n+1}{\alpha_n n}.
\end{align*}

In the general case where $k = \integer{\log_2 (\alpha_n n)}$, that is $\alpha_n n /2 \leq 2^k \leq \alpha_n n$, we can lower bound the probability of the event $\mathcal{I}_T$ as 
%\es{je ne suis pas sûr de la ligne suivante (j'ai plus de batterie!!)}
\begin{align*}
    \Prob{\mathcal{I}_T}& = \frac{2^k  \times (2^k - 1) \times \hdots \times (2^k - n +1)}{(2^k+n-1) \times (2^k + n - 2) \times \hdots 2^k}  \geq \left(\frac{2^k - n +1}{2^k + n - 1}\right)^n  \geq \left(\frac{2^k - n }{2^k + n }\right)^n \\
    & \geq \exp \left( n \log \left( \frac{2^k - n }{2^k + n }\right)\right) 
    \geq \exp \left( n \log \left(1 - \frac{2n }{2^k + n }\right)\right) 
    \geq \exp \left( - n  \left( \frac{2 }{\frac{2^k}{n} - 1 }\right)\right)\\
    & \geq \exp \left( -   \left( \frac{4n}{\alpha_n - 2 }\right)\right)\\
    \end{align*}
since $\log (1 - x) \geq - x / (1-x)$ and provided that $\alpha_n >2$ for the last inequality.
To upper bound the probability, note that, for all $r \in \{1, \hdots, \lfloor n/2\rfloor \}$
\begin{align*}
 \frac{2^k -n+r}{2^k + n -r } & \leq   \frac{2^k -n+\frac{n}{2}}{2^k + n -\frac{n}{2} - 1 } \leq \frac{2^k -\frac{n}{2}}{2^k + \frac{n}{2} - 1 }, 
\end{align*}
and, for all $r \in \{1,...,n\}$,
\begin{align*}
 \frac{2^k -n+r}{2^k + n -r } & \leq   1. 
\end{align*}
Therefore, one can also upper bound the probability as 
\begin{align*}
    \Prob{\mathcal{I}_T}& = \frac{2^k  \times (2^k - 1) \times \hdots \times (2^k - n +1)}{(2^k+n-1) \times (2^k + n - 2) \times \hdots 2^k} \\
    & \leq \left( \frac{2^k -\frac{n}{2}}{2^k + \frac{n}{2} - 1 }\right)^{\lfloor n/2\rfloor} \\
    & \leq \exp \left( \left\lfloor \frac{n}{2} \right\rfloor \log \left( 1 - \frac{n-1}{2^k + \frac{n}{2} - 1 } \right)\right)\\
    & \leq \exp \left( - \left\lfloor \frac{n}{2} \right\rfloor  \left( \frac{\frac{n}{2}}{2^k + \frac{n}{2} - 1 } \right)\right)\\
    & \leq \exp \left( - \left\lfloor \frac{n}{2} \right\rfloor  \left( \frac{\frac{1}{2}}{\frac{2^k}{n} + \frac{1}{2} } \right)\right)\\
    & \leq \exp \left( - \left\lfloor \frac{n}{2} \right\rfloor  \left( \frac{1}{2 \alpha_n  + 1 } \right)\right),
    \end{align*}
for all $n \geq 2$.
Finally, for all $n \geq 2$, and for all $\alpha_n > 2$, 
\begin{align*}
    \exp \left( -    \frac{4n}{\alpha_n - 2 }\right) \leq \Prob{\mathcal{I}_T} \leq \exp \left( - \left\lfloor \frac{n}{2} \right\rfloor  \left( \frac{1}{2 \alpha_n  + 1 } \right)\right).
\end{align*}
% sz\textcolor{blue}{calcul suivant à supprimer}
%     \begin{align*}
%     & \geq \frac{\alpha_n n}{(\alpha_n+1)n} \cdot \frac{\alpha_n n-1}{(\alpha_n+1)n-1} \hdots \frac{(\alpha_n-1)n+1}{\alpha_n n + 1} \\
%     &\geq \left(\frac{\alpha_n -1}{\alpha_n} \right)^n \\
%     &\geq e^{n \log(1-\frac{1}{\alpha_n})}  \\
%     &\geq e^{-\frac{n}{\alpha_n - 1}} \xrightarrow[\alpha_n \to \infty]{} 1.
%     %&= \frac{\alpha_n}{(\alpha_n+1)} \cdot \frac{\alpha_n -1/n}{(\alpha_n+1)-1/n} \hdots \frac{(\alpha_n-1)+1/n}{\alpha_n }
% \end{align*}
%     \begin{align*}
%         \frac{\alpha_n n -r}{(\alpha_n+1)n -r -1} \leq \frac{\alpha_n + 1/2}{\alpha_n + 1}.
%     \end{align*}
% \begin{align*}
%  \frac{2^k -r}{2^k + n -r -1} \leq \frac{2^k  + n/2}{2^k + n}   
% \end{align*} 
%     The computation of the upper bound is similar, note that 
%     It follows that
%     \begin{align*}
%         \Prob{\mathcal{I}_T} &\leq \left(\frac{\alpha_n + 1/2}{\alpha_n + 1} \right) n \\
%         &\leq e^{n \log(\frac{\alpha_n + 1/2}{\alpha_n + 1}) } \\
%         &\leq e^{-\frac{n}{2(\alpha_n +1)}}.
%     \end{align*}

% %\end{proof}

%%%%%%%%%%%%%%%%%%%%%%%%%%

\subsubsection{Proof of Corollary \ref{cor:interpolation_CRF} (Probability of interpolation for a CRF)}

As it is necessary for all trees to interpolation for the forest to interpolate, the probability that the forest interpolates is smaller than the probability that a single tree interpolates.

%%%%%%%%%%%%%%%%%%%%%%%%%%%%%%%%%%%%%%%%%%%%%%%%%%%%%%%%%%%%%%%%%%%%%%%%%%%%%%%%%%%%%%%%%%%%%%

%%%%%%%%%%%%%%%%%%%%%%%%%%%%

\subsubsection{Proof of Proposition \ref{prop:crf_inconsistency} (CRF inconsistency)}
%\begin{proof}[Proof of Proposition \ref{prop:crf_inconsistency}]
\label{proof:crf_inconsitency}
Let $\infforest^{\mathrm{CRF}}$ be an infinite CRF with each tree of depth $k_n \geq \log_2 (\alpha_n n)$, that is each tree has at  least $\alpha_n n$ leaves, with $\alpha_n n > 1$. Let $X$ be uniformly distributed on $[0,1]^d$. We write $\Bar{f}_{n,\infty}^{\mathrm{CRF}}(X) = \Esp{\infforest^{\mathrm{CRF}}(X) | X, X_1, ..., X_n}$. Then, denoting $\mathcal{E}$ the event ``$N_{n,\infty}(X) = 0$" (or equivalently, ``X falls into a non-empty leaf"),
\begin{align}
    \mathcal{R}(\infforest^{\mathrm{CRF}}(X)) &= \Esp{\left( \infforest^{\mathrm{CRF}}(X) - f^\star(X)\right)^2} \\
    &\geq \Esp{\left(\Bar{f}_{n,\infty}^{\mathrm{CRF}}(X) - f^\star(X) \right)^2} \\
    &= \Esp{\left(\displaystyle \sum_{i=1}^n \mathds{E}_\Theta\left[W_{ni}(X, \Theta) f^\star(X_i)\right] - \left(\ind{\mathcal{E}} + \ind{\mathcal{E}^c} \right)f^\star(X) \right)^2} \\
    &= \Esp{\left(\ind{\mathcal{E}^c} \displaystyle \sum_{i=1}^n \mathds{E}_\Theta\left[W_{ni}(X, \Theta) \left(f^\star(X_i) - f^\star(X) \right)\right] - \ind{\mathcal{E}}f^\star(X) \right)^2} \\
    &\geq \Esp{f^\star(X)^2 \ind{\mathcal{E}}} \\
    &\geq \Esp{f^\star(X)^2 \Prob{\mathcal{E}|X}}.
\end{align}

Besides,
\begin{align}
    \Prob{\mathcal{E}|X} &= \Prob{N_{n,\infty}(X) = 0 |X} \\
    & \geq \left(1 - \frac{1}{\alpha_n n} \right)^n,
\end{align}
and as $\log(1-1/x) \geq - \frac{1}{x-1}$ for $x>1$, 
\begin{align}
    \left(1 - \frac{1}{\alpha_n n} \right)^n &= e^{n \log\left( 1 - \frac{1}{\alpha_n n }\right)} \\
    &\geq e^{-\frac{n}{\alpha_n n - 1}}.
\end{align}
Thus, 
\begin{align}
\mathcal{R}(\infforest^{\mathrm{CRF}}(X)) & \geq e^{-\frac{n}{\alpha_n n - 1}} \Esp{f^\star(X)^2 },
\end{align}
which tends to 0 if and only if $\alpha_n$ tends to zero as  $n$ tends to infinity. Since, by assumptions, $\alpha_n$ does not tend to zero and $\Esp{f^\star(X)^2} > 0$, the infinite CRF is inconsistent. 
%\end{proof}

%%%%%%%%%%%%%%%%%%%%%%%%%%%%%%%%%%%%%%%%

%%%%%%%%%%%%%%%%%%%%%%%%%%%%%

\subsubsection{Proof of Lemma \ref{lem:agg_crf_empty_cell} (Probability of falling into an empty cell of the void-free CRF)}

Recall that  $\mathcal{E}_{M,n}(x)$ is the event ``for all $m \in \{1, \hdots, M\}, N_n(x, \Theta_m) =0$''. 
We have
\begin{align}
	\mathcal{E}_{M,n}(x) &= \bigcap_{j =1}^M \left\{ N_n(x,\Theta_j) = 0 \right\}.
\end{align}

Given a dataset, we distinguish two situations: either $x$ falls into an area where it cannot be connected to a point $X_i$ for any tree, or the dataset is such that $x$ could be connected to a point $X_i$ for a certain configuration of cuts within a tree.
We write $\mathcal{E}_{1,n}(x)$ the ($\mathcal{D}_n$-measurable) event $\{ \forall \theta, N_n(x,\theta)=0 \}$. Consequently, we have  $\mathcal{E}_{1,n}(x)^c = \{ \exists \hspace{0.1cm}\theta, N_n(x,\theta) \neq 0 \}$. Using these notations, we obtain
\begin{align}
    \Prob{\mathcal{E}_{M,n}(x)} &= \Prob{\mathcal{E}_{M,n}(x) \cap \mathcal{E}_{1,n}(x)} + \Prob{\mathcal{E}_{M,n}(X) \cap \mathcal{E}_{1,n}(x)^c} \\
    &= \Prob{\mathcal{E}_{1,n}(x)} + \Prob{\mathcal{E}_{M,n}(x) \cap \mathcal{E}_{1,n}(x)^c} \label{eq:proba_decompo}
\end{align}
where the first probability term of the second line is a probability taken over $\mathcal{D}_n$ only, since $\mathcal{E}_{1,n}(x)$ does not depend on $\Theta$. We control this probability thanks to the following Lemma.

\begin{lemma} \label{lem:proba_E1} For all $x \in [0,1]^d$, we let $\mathcal{E}_{1,n}(x)$ be the event $\{ \forall \theta, N_n(x,\theta)=0 \}$. Then, we have $$ \Prob{\mathcal{E}_{1,n}(x)} \leq e^{-\frac{n}{2^{k+1}}}.$$
\end{lemma}

\begin{proof}

Let $x \in [0,1]^d$. The event $\mathcal{E}_{1,n}(x)$ happens if all the points of the dataset fall into parts of the space that cannot connect to $x$ for any tree. In order to compute its probability, we compute the size of the \textit{connection area of $x$} for trees of depth $k$, denoted 
\begin{align}
Z_{c,k}(x) = \left\lbrace z \in [0,1]^d~: \exists \theta, z \in A_n(x,\theta) \right\rbrace.
\end{align}
We recall that trees are built independently from the dataset and that all cuts are made in the middle of the current node for a uniformly chosen feature at each step. 
%$x$ is connected to a cell obtained by cutting along feature $X^{(1)}$ $k$ times, to another cell obtained by cutting $k-1$ times along feature $X^{(1)}$ and once along another feature, etc, and similarly for all features (See Figure \textbf{[REF]}). The connection area $Z_{c,k}(x)$ is the union of all these cells.
We denote $A(k_1,...,k_d, x)$ the cell of $x$ obtained by cutting $k_j$ times along feature $X^{(j)}$ for all $j \in \{1, \hdots, d\}$. Then, the volume of the connection area $Z_{c,k}$ of $x$ is
\begin{align}
    \mu(Z_{c,k}(x)) &= \mu \left( \bigcup\limits_{\substack{0 \leq k_1, \hdots, k_d \leq k \\ \sum_j k_j = k}} A(k_1,..., k_d, x) \right) \\
    &\geq \mu \left( \bigcup\limits_{\substack{0 \leq k_1, k_2 \leq k \\ k_1 + k_2 = k}} A(k_1, k_2, 0,..., 0, x) \right).
    \label{eq:proof_vol_min}
\end{align}

By $\sigma$-additivity of $\mu$,
\begin{align}
    \notag
    \mu & \left( \bigcup\limits_{\substack{0 \leq k_1, k_2 \leq k \\ k_1 + k_2 = k}} A(k_1, k_2, 0,..., 0, x) \right) \\
    &= \mu\Big(A(k,0,...,0,x)\Big) 
    +  \sum_{j=1}^k \mu\left( A(k-j, j,0,...,0, x) \setminus \bigcup_{\ell=0}^{j-1} A(k-\ell, \ell, 0,...,0, x) \right). \label{eq:proof1}
\end{align}
Given the shape of the cells $A(k-j, j,0,...,0, x)$, for all $j \in \{1, \hdots, d\}$, we have (see Figure~\ref{fig:connection_volume})
\begin{align}
     & A(k-j, j,0,...,0, x) \setminus \bigcup_{\ell=0}^{j-1} A(k-\ell, \ell, 0,...,0, x) \nonumber \\
      =  & A(k-j, j,0,...,0, x) \setminus A(k-j+1, j-1, 0,...,0, x). \label{eq:proof2}
\end{align}
Furthermore, note that, for all $j \in \{1, \hdots, d\}$, the volume of each cell $A(k-j+1, j-1, 0,...,0, x)$ is $2^{-k}$ (since $k$ cuts have been performed). Therefore, for all $j \in \{1,\hdots, k\}$,
\begin{enumerate}
    \item $\mu(A(k-j, j,0,...,0, x)) = \mu(A(k-j+1, j-1,0,...,0, x)) = 2^{-k}$
    \item $\mu\big((A(k-j, j,0,...,0, x) \cap A(k-j+1, j-1,0,...,0, x)\big) = \frac{\mu(A(k-j, j,0,...,0, x))}{2}$ as can be seen on Figure \ref{fig:connection_volume}.
\end{enumerate}

\begin{figure}[h]
    \centering
    \includegraphics[trim={2cm 10.8cm 2.5cm 3cm}, clip, width=0.7\textwidth]{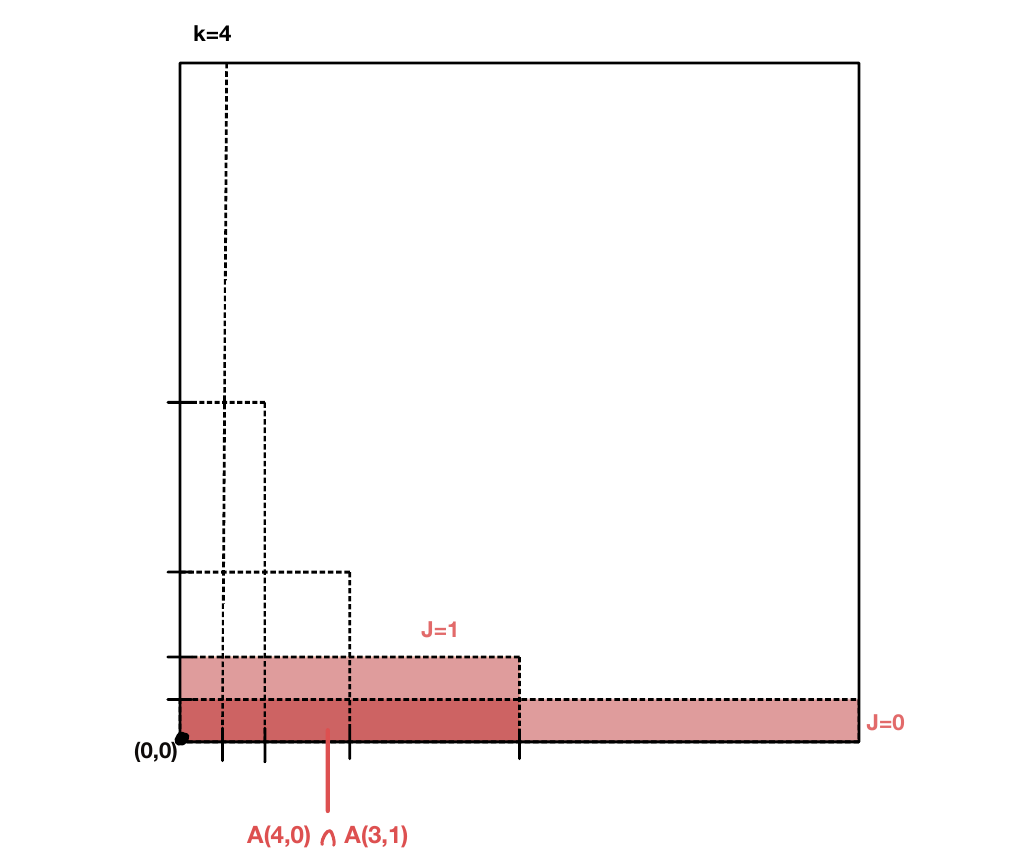}
    \caption{Volume of leaf intersection $\mu\big((A(k-j, j, x) \cap A(k-j+1, j-1, x)\big) $ in dimension 2 with $x=(0,0)$, $k=4$ cuts and $j\in \{0,1\}$.}
    \label{fig:connection_volume}
\end{figure}

We deduce from these facts that, for all $j$,
\begin{align}
\label{eq:proof3}
     \mu\big( A(k-j, j,0,...,0, x) \setminus A(k-j+1, j-1, 0,...,0, x)) 
     &= \frac{\mu(A(k-j, j,0,...,0, x))}{2}\\
     &= 2^{-(k+1)}
\end{align}

Hence, combining equations~\eqref{eq:proof1}, \eqref{eq:proof2} and \eqref{eq:proof3}, we have
\begin{align}
     \mu \left( \bigcup\limits_{\substack{0 \leq k_1, k_2 \leq k \\ k_1 + k_2 = k}} A(k_1, k_2, 0,..., 0, x) \right) & = 2^{-k} + k2^{-(k+1)}.
\end{align}
Consequently, using inequality~\eqref{eq:proof_vol_min}, 
\begin{align}
    \mu(Z_{c,k}(x)) \geq k2^{-(k+1)}.
\end{align}

Finally, as the $X_i$'s are uniformly distributed on $[0,1]^d$ and $\mathcal{E}_{1,n}(x)$ is realized when none of the $X_i$s fall into $Z_{c,k}(x)$, 
\begin{align}
    \Prob{\mathcal{E}_{1,n}(x)} &= \Prob{\forall i \in \{1, \hdots, n\}, X_i \notin Z_{c,k}(x)}\\
    & = \left(1 - \mu(Z_{c,k}(x)) \right)^n \\
    &\leq \left(1 - k2^{-(k+1)} \right)^n \\
    &= e^{n\log(1-k2^{-(k+1)})} \\
    &\leq e^{-\frac{k n}{ 2^{k+1}}}.
\end{align}

\end{proof}

Regarding the second term of \eqref{eq:proba_decompo}, we have
\begin{align}
  \Prob{\mathcal{E}_{M,n}(x) \cap \mathcal{E}_{2,n}(x)} &=  \Prob{ \left(\bigcap_{j=1}^M N_n(x,\Theta_j)= 0 \right) \bigcap \Big(\exists i \in \{1,\hdots,n\}, X_i \in Z_{c,k}(x) \Big)} \\
  &= \Esp{\Esp{\ind{\exists i \in \{1,\hdots,n\}, X_i \in Z_{c,k}(x)} \ind{\bigcap_{j=1}^M N_n(x,\Theta_j)= 0} | \mathcal{D}_n}} \\
  &= \Esp{\ind{\exists i \in \{1,\hdots,n\}, X_i \in Z_{c,k}(x)} \Prob{\bigcap_{j=1}^M N_n(x,\Theta_j)= 0 | \mathcal{D}_n}} \\
  &= \Esp{\ind{\exists i \in \{1,\hdots,n\}, X_i \in Z_{c,k}(x)} (1-p_n)^M}
\end{align}
where $p_n = \mathbb{P}_\Theta\left(N_n(x,\Theta) > 0 | \mathcal{D}_n\right)$ and where the last line is obtained by independence of the $\Theta_j$'s conditionally on $\mathcal{D}_n$.
Note that, if $\exists i \in \{1,\hdots,n\}, X_i \in Z_{c,k}(x)$, then $p_n \geq d^{-k}$ since a tree connects $x$ and a point in $Z_{c,k}(x)$ with probability at least $d^{-k}$ (i.e.\ by choosing the right cut at each step). Hence, 
\begin{align}
\ind{\exists i \in \{1,\hdots,n\}, X_i \in Z_{c,k}(x)}(1-p_n)^M \leq (1-d^{-k})^M, 
\end{align}
which leads to
\begin{align}
    \Prob{\mathcal{E}_{M,n}(x) \cap \mathcal{E}_{1,n}(x)^c} &\leq \left(1-d^{-k}\right)^M \\
    &\leq e^{-{M}{d^{-k}}}. \label{proof_eq_prob_empty_cells_void_crf}
\end{align}

Finally, gathering Lemma \ref{lem:proba_E1} and inequality~\eqref{proof_eq_prob_empty_cells_void_crf} yields
\begin{align}
    \Prob{\mathcal{E}_{M,n} (x)} \leq e^{-\frac{kn}{2^{k+1}}} + e^{-M d^{-k}}.
\end{align}

%%%%%%%%%%%%%%%%%%%%%%%%%%

%%%%%%%%%%%%%%%%%%%%%%%%%%%%%%%%%%%%%%%%%

\subsubsection{Proof of Proposition \ref{th:agg_crf_consistency} (Consistency of void-free-CRF in a noiseless setting)}

Recall that, in a noiseless setting (that is, for all $i$, $Y_i = f^{\star}(X_i)$), the risk of the Void-free CRF can be written as 
\begin{align*}
    \Esp{\left(\vfcrf(X) - f^\star(X) \right)^2 } &= \Esp{\left( \frac{\ind{\mathbb{P}_\Theta \left(N_n(X,\Theta)>0\right) >0 }}{\mathbb{P}_\Theta \left(N_n(X,\Theta) >0\right)} \sum_{i=1}^n f^\star(X_i)\mathbb{E}_\Theta[W_{ni}(X, \Theta) \ind{N_n(X,\Theta)>0}] - f^\star(X) \right)^2}.
\end{align*}

We decompose $f^\star(X)$ as
\begin{align*}
    f^\star(X) = \left(\ind{\mathbb{P}_\Theta \left(N_n(X,\Theta)>0\right) >0 } + \ind{\mathbb{P}_\Theta \left(N_n(X,\Theta)>0\right) =0 }\right)f^\star(X)
\end{align*}
in order to write
\begin{align}
    &\Esp{\left( \frac{\ind{\mathbb{P}_\Theta \left(N_n(X,\Theta)>0\right) >0 }}{\mathbb{P}_\Theta \left(N_n(X,\Theta >0)\right)} \sum_{i=1}^n f^\star(X_i)\mathbb{E}_\Theta[W_{ni}(X, \Theta) \ind{N_n(X,\Theta)>0}] - f^\star(X) \right)^2} \notag \\
    &= \Esp{\left( \frac{\ind{\mathbb{P}_\Theta \left(N_n(X,\Theta)>0\right) >0 }}{\mathbb{P}_\Theta \left(N_n(X,\Theta >0)\right)} \sum_{i=1}^n \left(f^\star(X_i) - f^\star(X)\right)\mathbb{E}_\Theta[W_{ni}(X, \Theta) \ind{N_n(X,\Theta)>0}] -  f^\star(X)\ind{\mathbb{P}_\Theta \left(N_n(X,\Theta)>0\right) =0} \right)^2} \notag \\
    &\leq 2  \Esp{\left( \frac{\ind{\mathbb{P}_\Theta \left(N_n(X,\Theta)>0\right) >0 }}{\mathbb{P}_\Theta \left(N_n(X,\Theta >0)\right)} \sum_{i=1}^n \left(f^\star(X_i) - f^\star(X)\right)\mathbb{E}_\Theta[W_{ni}(X, \Theta) \ind{N_n(X,\Theta)>0}] \right)^2} \notag \\
    & \quad + 2 \Esp{ \left(f^\star(X)\ind{\mathbb{P}_\Theta \left(N_n(X,\Theta)>0\right) =0}  \right)^2}
\end{align}

The second term of the last inequality verifies
\begin{align}
    \Esp{ \left( f^\star(X)\ind{\mathbb{P}_\Theta \left(N_n(X,\Theta)>0\right) =0}  \right)^2} & \leq ||f^\star||_\infty^2 \Prob{\mathbb{P}_\Theta \left(N_n(X,\Theta)>0\right) =0}.
\end{align}
The event $\{\mathbb{P}_\Theta \left(N_n(X,\Theta)>0\right) =0 \}$ is $(X,\mathcal{D}_n)$-measurable, it corresponds to the situation where for any $\theta$, $N_n(X,\theta)=0$, i.e.\ the dataset is such that it is impossible for a tree to connect $X$ with one of the $X_i$'s. This probability is controlled by Lemma \ref{lem:proba_E1}:
\begin{align*}
    \Prob{\mathbb{P}_\Theta \left(N_n(X,\Theta)>0\right) =0} \leq e^{-\frac{kn}{2^{k+1}}}.
\end{align*}

Denoting by $\mu \left(A_{n}^{(j)}(x, \T)\right)$ the length of the $j$th side of the cell containing $x$ and following a computation from \cite{klusowski2021sharp}, 
\begin{align}
	\sum_{i=1}^n W_{ni}(X, \Theta) |f^\star(X) - f^\star(X_i)| \ind{N_n(X,\Theta)>0} %&\leq \sum_{i=1}^n W_{ni}(X, \Theta) \ind{N_n(X,\Theta)>0} \sum_{j=1}^d||\partial_j f^\star||_\infty (b_j-a_j) \\
	& \leq \sum_{i=1}^n W_{ni}(X, \Theta) \left( \sum_{j=1}^d||\partial_j f^\star||_\infty |X_i^{(j)} - X^{(j)}| \right)  \ind{N_n(X,\Theta)>0} \\
	&\leq \sum_{i=1}^n W_{ni}(X, \Theta) \ind{N_n(X,\Theta)>0} \sum_{j=1}^d||\partial_j f^\star||_\infty (b_j-a_j) \\
	&\leq \ind{N_n(X,\Theta)>0} \sum_{j=1}^d||\partial_j f^\star||_\infty \mu \left(A_{n}^{(j)}(X, \T)\right).
\end{align}
%\begin{align*}
%    |f^\star(X) - f^\star(X_i)| &\leq \sum_{j=1}^d||\partia\ell_j f^\star||_\infty |X_i^{(j)} - X^{(j)}| \\
%    &\leq \sum_{j=1}^d||\partia\ell_j f^\star||_\infty (b_j -a_j)\\
%    \sum_{i=1}^n W_{ni}(X, \Theta) |f^\star(X) - f^\star(X_i)| \ind{N_n(X,\Theta)>0} &\leq \sum_{i=1}^n W_{ni}(X, \Theta) \ind{N_n(X,\Theta)>0} \sum_{j=1}^d||\partia\ell_j f^\star||_\infty (b_j-a_j)
%\end{align*}
%and
Therefore,
\begin{align}
    \Esp{\left(\vfcrf(X) - f^\star(X)\right)^2} &\leq 2 \Esp{\left(  \frac{1}{\mathbb{P}_\Theta(N_n(X,\Theta)>0)} \sum_{j=1}^d||\partial_j f||_\infty \mathbb{E}_\Theta\left[\ind{N_n(X,\Theta)>0}\mu \left(A_{n}^{(j)}(X, \T)\right)\right]\right)^2} \nonumber \\
    & \quad +2 e^{-\frac{kn}{2^{k+1}}} \\
    &\leq 2d \displaystyle \sum_{j=1}^d ||\partial f_j^\star ||_\infty^2 \Esp{\frac{1}{\mathbb{P}_\Theta(N_n(X,\Theta)>0)^2}\mathbb{E}_\Theta\left[\ind{N_n(X,\Theta)>0}\mu \left(A_{n}^{(j)}(X, \T)\right)\right]^2} \nonumber \\
    & \quad + 2 e^{-\frac{kn}{2^{k+1}}}. \label{eq:indep_to_cond}
\end{align}

{
Note that the length $\mu \left(A_{n}^{(j)}(X, \T)\right)$ of the $j$-th side of the cell $A_{n}(X, \T)$ and the event $\{ N_n(X,\Theta)>0\}$ are not independent conditional on $X_1,..., X_n, X$. Indeed, given the geometry of the dataset, it is possible that cutting along the $j$th direction  isolates $X$ from the dataset. Therefore its length should be computed conditional on the event $\{ N_n(X,\Theta)>0\}$.

%We denote $E_{n, \kappa}(j,  X, \Theta) := E_j(\kappa, X,n,\Theta)$ the event \{at depth $\kappa$, the node $A_{n,\kappa}(X,\Theta)$ contains at least two data points\}. 
To this aim, we denote for all $\kappa \in \mathbb{N}$, $A_{n,\kappa}(X,\Theta) $ the cell  containing $X$ at depth $\kappa$ in a centered tree built with the extra randomness $\Theta$. 
Conditional on $N_n(X,\Theta) >0$, the $j$th direction can be chosen to split along if and only if it does not isolate $X$ from the points of the dataset. 
Thus, we denote by $E_{n, \kappa}(j,  X, \Theta)$
the event "In a centered tree built with the randomized cuts $\Theta$, at depth $\kappa$, splitting the cell containing $X$ along the $j$th direction does not isolate $X$".
Then,
\begin{align}
    \mathbb{E}_\Theta\left[\ind{N_n(X,\Theta)>0}\mu \left(A_{n}^{(j)}(X, \T)\right)\right] &= \mathbb{E}_\Theta\left[\ind{N_n(X,\Theta)>0}\mu \left(A_{n}^{(j)}(X, \T)\right) (\ind{E_{n, \kappa}(j,  X, \Theta)^c}+\ind{E_{n, \kappa}(j,  X, \Theta)} )\right] \\
    &\leq \mathbb{E}_\Theta\left[\ind{N_n(X,\Theta)>0} \ind{E_{n, \kappa}(j,  X, \Theta)^c}\right]  \\
    & \quad + \mathbb{E}_\Theta\left[\ind{N_n(X,\Theta)>0} \ind{E_{n, \kappa}(j,  X, \Theta)}  \mu \left(A_{n}^{(j)}(X, \T) \right)\right],\label{eq:exp_0} 
\end{align}
since $\mu \left(A_{n}^{(j)}(X, \T) \right) \leq 1$. We denote $A_{n,\kappa}^{(j),\mathrm{left}}(X,\Theta)$ (resp.\  $A_{n,\kappa}^{(j),\mathrm{right}}(X,\Theta)$) the left (resp.\ right) daughter of the cell $A_{n, \kappa}(X, \T)$ that has been split along the $j$th direction
%. $$$  side of the cell w.r.t.\ direction $X^{(j)}$ 
(note that the whole cell is considered here, not only the projection on the $j$-th side). 
Then,
\begin{align}
    &\mathbb{E}_\Theta\left[\ind{N_n(X,\Theta)>0} \ind{E_{n, \kappa}(j,  X, \Theta)^c}\right] \nonumber \\
    &= \mathbb{P}_\Theta\left(E_{n, \kappa}(j,  X, \Theta)^c  \big| N_n(X,\Theta)>0  \right) \mathbb{P}_\Theta\left(N_n(X,\Theta)>0 \right) \\
    &= \mathbb{P}_\Theta\left( \left(N_n(A_{n,\kappa}^{(j),\mathrm{left}}(X,\Theta))=0\right) \cap \left(X \in A_{n,\kappa}^{(j),\mathrm{right}}(X,\Theta) \right) \big| N_n(X,\Theta)>0  \right) \mathbb{P}_\Theta\left(N_n(X,\Theta)>0 \right) \nonumber \\
    & \quad + \mathbb{P}_\Theta\left( \left(N_n(A_{n,\kappa}^{(j),\mathrm{right}}(X,\Theta))=0\right) \cap \left(X \in A_{n,\kappa}^{(j),\mathrm{left}}(X,\Theta) \right) \big| N_n(X,\Theta)>0  \right) \mathbb{P}_\Theta\left(N_n(X,\Theta)>0 \right)\\
  %  &= \left( \underbrace{\mathbb{P}_\Theta\left(E_{n, \kappa}(j,  X, \Theta)^c \cap \mathcal{E}(\kappa)  \big| N_n(X,\Theta)>0  \right)}_{=0} + \mathbb{P}_\Theta\left(E_{n, \kappa}(j,  X, \Theta)^c \cap \mathcal{E}^c(\kappa)  \big| N_n(X,\Theta)>0  \right) \right) \mathbb{P}_\Theta\left(N_n(X,\Theta)>0 \right) \\
  %  &= 2 \mathbb{P}_\Theta\left((N_n(A_l^j(\kappa))=0) \cap \mathcal{E}(\kappa)^c \big| N_n(X,\Theta)>0  \right) \mathbb{P}_\Theta\left(N_n(X,\Theta)>0 \right) \\
    &\leq 2\mathbb{P}_\Theta\left(N_n(A_{n,\kappa}^{(j),\mathrm{left}}(X,\Theta))=0 \big| N_n(X,\Theta)>0  \right) \mathbb{P}_\Theta\left(N_n(X,\Theta)>0 \right). \label{eq:exp_1} 
\end{align}
Moreover,
\begin{align}
    &\mathbb{E}_\Theta\left[\ind{N_n(X,\Theta)>0} \ind{E_{n, \kappa}(j,  X, \Theta)}  \mu \left(A_{n}^{(j)}(X, \T) \right)\right] \nonumber \\
    &\leq \mathbb{E}_\Theta\left[  \mu \left(A_{n}^{(j)}(X, \T)\right) \big| E_{n, \kappa}(j,  X, \Theta), N_n(X,\Theta)>0 \right]  \mathbb{P}_\Theta\left(N_n(X,\Theta)>0  \right).
    %&\leq \mathbb{E}_\Theta\left[  \mu \left(A_{n}^{(j)}(X, \T)\right) \big| E_{n, \kappa}(j,  X, \Theta), N_n(X,\Theta)>0 \right]  \mathbb{P}_\Theta\left(N_n(X,\Theta)>0  \right).
\end{align}

Denoting $K_{j,\kappa}(X,\Theta)$ the number of splits made on feature $j$ up to depth $\kappa$ to produce the cell containing $X$, we obtain
\begin{align}
    \mathbb{E}_\Theta\left[  \mu \left(A_{n}^{(j)}(X, \T)\right) \big| E_{n, \kappa}(j,  X, \Theta), N_n(X,\Theta)>0 \right] &\leq \mathbb{E}_\Theta\left[  2^{-K_{j,\kappa}(X,\Theta)} \big| E_{n, \kappa}(j,  X, \Theta), N_n(X,\Theta)>0 \right].
%&\leq \left(1 - \frac{1}{2d}\right)^\kappa.
\end{align}

We denote $\delta_j(X, \Theta) \in \{0,1\}^k$ the vector indicating at which depth the $j$th direction   is chosen for splitting, that is  $\delta_{j,\ell}(X, \Theta) = 1$ if and only if the $j$th feature is used for splitting at depth $\ell$. We have
\begin{align*}
    K_{j,\kappa}(X, \Theta) = \sum_{\ell=1}^\kappa \delta_{j,\ell}(X, \Theta).
\end{align*}
For $\ell = 1, \hdots, \kappa$, the random variables $\delta_{j,\ell}(X, \Theta)$ are distributed as Bernoulli random variables. Conditional on $E_{n, \kappa}(j,  X, \Theta)$ and $N_n(X, \Theta) >0$, we know that for all $\ell = 1, \hdots, \kappa$, the $j$th direction was eligible for splitting at level $\ell$. Therefore, the probability of selecting the $j$th direction at any level $1 \leq \ell \leq \kappa$, is $p_{\ell} \geq 1/d$ (at worst, all variables are eligible for splitting, leading to $p_{\ell} = 1/d$). Besides, conditional on  $E_{n, \kappa}(j,  X, \Theta)$  and $N_n(X, \Theta) >0$, the random variables $\delta_{j,\ell}(X, \Theta)$ are independent by construction of the centered forest. Indeed, conditional on $E_{n, \kappa}(j,  X, \Theta)$  and $N_n(X, \Theta) >0$, the $j$th direction  can be chosen up to depth $\kappa$ (independence is broken only when the direction cannot be chosen at a given depth as the following one will not be chosen either).
Then,
\begin{align}
    \mathbb{E}_\Theta\left[  2^{-K_{j,\kappa}(X,\Theta)} \big| E_{n, \kappa}(j,  X, \Theta), N_n(X,\Theta)>0 \right] &= \prod_{\ell=1}^\kappa  \mathbb{E}_\Theta\left[  2^{-\delta_{j,\ell}(X,\Theta)} \big| E_{n, \kappa}(j,  X, \Theta), N_n(X,\Theta)>0 \right] \\
    &= \prod_{\ell=1}^\kappa \left( \frac{p_\ell}{2} + (1-p_\ell) \right) \\
    &\leq \left(1 - \frac{1}{2d}\right)^\kappa. \label{eq:exp_2} 
\end{align}

Therefore, injecting Equations \eqref{eq:exp_1} and \eqref{eq:exp_2} into \eqref{eq:exp_0}, we get
\begin{align}
     \frac{\mathbb{E}_\Theta\left[\ind{N_n(X,\Theta)>0}\mu \left(A_{n}^{(j)}(X, \T)\right)\right]}{\mathbb{P}_\Theta\left(N_n(X,\Theta)>0 \right)} &\leq 2 \mathbb{P}_\Theta\left(N_n(A_{n,\kappa}^{(j),\mathrm{left}}(X,\Theta))=0 \big| N_n(X,\Theta)>0  \right) + \left(1 - \frac{1}{2d}\right)^\kappa, 
\end{align}
which implies
\begin{align}
     \left( \frac{\mathbb{E}_\Theta\left[\ind{N_n(X,\Theta)>0}\mu \left(A_{n}^{(j)}(X, \T)\right)\right]}{\mathbb{P}_\Theta\left(N_n(X,\Theta)>0 \right)} \right)^2 &\leq  4 \mathbb{P}_\Theta\left(N_n(A_{n,\kappa}^{(j),\mathrm{left}}(X,\Theta))=0 \big| N_n(X,\Theta)>0  \right) + 2 \left(1 - \frac{1}{2d}\right)^{2 \kappa}, 
\end{align}
using $(a+b)^2 \leq 2a^2+2b^2 \leq 2a^2 +2b$ if $b\leq1$. Plugging-in this expression into \eqref{eq:indep_to_cond} leads to 
\begin{align}
    \Esp{\left(\vfcrf(X) - f^\star(X)\right)^2} &\leq 4d \displaystyle \sum_{j=1}^d ||\partial f_j^\star ||_\infty^2 \left(1 - \frac{1}{2d}\right)^{2\kappa} + 2 e^{-\frac{kn}{2^{k+1}}} \nonumber \\
    & \quad + 8d \displaystyle \sum_{j=1}^d ||\partial f_j^\star ||_\infty^2 \Prob{N_n(A_{n,\kappa}^{(j),\mathrm{left}}(X,\Theta))=0 \big| N_n(X,\Theta)>0}.
\end{align}

Then,
\begin{align}
    %&\Esp{\Prob{N_n(A_{n,\kappa}^{(j),\mathrm{left}}(X,\Theta))=0 \big| N_n(X,\Theta)>0 , A_{n,\kappa}(X,\Theta)} | N_n(X,\Theta)>0} \\
    & \Prob{N_n(A_{n,\kappa}^{(j),\mathrm{left}}(X,\Theta))=0 \big| N_n(X,\Theta)>0} \\
    &= \Esp{\Prob{N_n(A_{n,\kappa}^{(j),\mathrm{left}}(X,\Theta))=0 \big| N_n(X,\Theta)>0, N_n(A_{n,\kappa}(X,\Theta)), X, \Theta } | N_n(X,\Theta)>0} \nonumber \\
    &= \Esp{2^{-N_n(A_{n,\kappa}(X,\Theta)) } | N_n(X,\Theta)>0}\\
    &\leq 2 \Esp{2^{-N_n(A_{n,\kappa}(X,\Theta))}}.
    %&=\frac{1}{2} \left(1 - \mu(A_l(\kappa))\right)^{n-1} \\
    %&= \frac{1}{2} \left(1 - 2^{-\kappa-1}\right)^{n-1} \\
    %&\leq \frac{1}{2} e^{-\frac{n-1}{2^{\kappa+1}}}.
\end{align}
The last line is obtained by making the expectation explicit and noting that $\Prob{N_n(X,\Theta)>0}^{-1} \leq 1/(1-e^{-1}) \leq 2$. 
Furthermore, conditional on $X, \Theta$, 
$N_n(A_{n,\kappa}(X,\Theta))$ is distributed as a binomial of parameters $n$ and $\mu\left(A_{n,\kappa}(X,\Theta)\right) = 2^{-\kappa}$. Thus, 
\begin{align}
    \Prob{N_n(A_{n,\kappa}^{(j),\mathrm{left}}(X,\Theta))=0 \big| N_n(X,\Theta)>0} 
    &\leq 2 \Esp{2^{-N_n(A_{n,\kappa}(X,\Theta))}}\\
    & \leq 2 \Esp{ \Esp{2^{-N_n(A_{n,\kappa}(X,\Theta))} | X, \Theta} }\\
    &\leq 2\left( 1 - \frac{\mu\left(A_{n,\kappa}(X,\Theta)\right)}{2}\right)^{n} \\
    &=  2\left( 1 - 2^{-\kappa-1}\right)^{n} \\
    &\leq 2\exp\left( -\frac{n}{2^{\kappa+1}} \right).
\end{align}

Overall,
\begin{align}
    \Esp{\left(\vfcrf(X) - f^\star(X)\right)^2} &\leq 4d \displaystyle \left(\sum_{j=1}^d ||\partial f_j^\star ||_\infty^2 \right) \left( \left(1 - \frac{1}{2d}\right)^{2\kappa} + 4 \exp \left( -\frac{n}{2^{\kappa+1}} \right) \right)  
     +  2 \exp \left( -\frac{kn}{2^{k+1}} \right).
\end{align}

%Then, Jensen inequality yields
%\begin{align*}
%    \Esp{\mathbb{E}_\Theta\left[(\mu \left(A_{n}^{(j)}(X, \T)\right))\right]^2} &\leq \Esp{\mu \left(A_{n}^{(j)}(X, \T)\right)}^2 \\
%    &= \Esp{2^{-K_j(X)}}^2 \\
%    &= \Esp{\Esp{2^{-K_j(X)}|X}}^2
%\end{align*}

%It is conditionally distributed as a binomial distribution of parameters $(k, 1/d)$. Hence,
%\begin{align*}
%    \Esp{2^{-K_j(X)} | X} &\leq \left(1-\frac{1}{2d}\right)^k.
%\end{align*}
Choosing $\kappa = \log_2(n) - \log_2(\log_2(n))$, that is $2^{\kappa} = n/(\log_2(n))$, we obtain 
\begin{align}
   \exp \left( 2 \kappa \log \left(1 - \frac{1}{2d}\right)\right)  + 4 \exp \left( -\frac{n}{2^{\kappa+1}} \right)   & \leq \left( \frac{n}{\log_2 n} \right)^{2 \log_2 \left( 1 - \frac{1}{2d}\right)} + 4 n^{-1/(2 \ln 2)}. 
\end{align}
Consequently, recalling that $k = \lfloor \log_2(n) \rfloor$, 
\begin{align}
    \Esp{\left(\vfcrf(X) - f^\star(X)\right)^2} &\leq 4d \displaystyle \left(\sum_{j=1}^d ||\partial f_j^\star ||_\infty^2 \right) \left( \left( \frac{n}{\log_2 n} \right)^{2 \log_2 \left( 1 - \frac{1}{2d}\right)} + 4 n^{-1/(2 \ln 2)} \right)  
     +  2 n^{-1/(2 \ln 2)}\\
     & \leq C_d  \left( \frac{n}{\log_2 n} \right)^{2 \log_2 \left( 1 - \frac{1}{2d}\right)}   
     + \left(  C_d +2 \right) n^{-1/(2 \ln 2)}, 
\end{align}
where $C_d = 4d \displaystyle \left(\sum_{j=1}^d ||\partial f_j^\star ||_\infty^2 \right).$
}

%Choosing $\kappa= \alpha \log_2 n $, $\alpha <1$, we obtain
%\begin{align}
%    \Esp{\left(\vfcrf(X) - f^\star(X)\right)^2} &\leq 4d \displaystyle \sum_{j=1}^d ||\partial f_j ||_\infty^2 \left( \left(1 - \frac{1}{2d}\right)^{2\alpha \log_2 n} + 
%   2 e^{-\frac{n^{1-\alpha}}{2}} \right) + 2 e^{-\frac{kn}{2^{k+1}}} \\
%    &= 4d \displaystyle \sum_{j=1}^d ||\partial f_j ||_\infty^2 \left( n^{2\alpha \log(2) \log\left(1-\frac{1}{2d}\right)} + 
%   2e^{-\frac{n^{1-\alpha}}{2}} \right) + 2 e^{-\frac{kn}{2^{k+1}}}
%\end{align}
%which ends the proof.

%To conclude,
%\begin{align}
%    \Esp{\left(\vfcrf(X) - f^\star(X)\right)^2} &\leq 2d \displaystyle \sum_{j=1}^d ||\partial f_j ||_\infty^2 \left(1-\frac{1}{2d}\right)^k + 2 e^{-\frac{kn}{2^{k+1}}}.
%\end{align}

%%%%%%%%%%%%%%%%%%%%%%%%%%%%%%%%%%%%%%%%%%%%%%%%%%%%%%%%%%%%%%%%%%%%%%%%%%%%%%%%%%%%%%%%%%%%%
%%%%%%%%%%%%%%%%%%%%%%%%%%%%%%%%%%%%%%%%%%%%%%%%%%%%%%%%%%%%%%%%%%%%%%%%%%%%%%%%%%%%%%%%%%%%%

\subsection{Proofs of Section \ref{sec:KeRF} (Theorem \ref{theoreme_consistency_centred_forest_approximation})}

In this section, we prove the consistency of the infinite KeRF estimator in the mean interpolating regime (Theorem \ref{theoreme_consistency_centred_forest_approximation}). We follow the proof given in \cite{scornet2016random} and first present two of its results.

\begin{lemma}[\cite{scornet2016random}]
\label{lemme_centred_random_forest}
Let $k \in \mathds{N}$ and consider an infinite centered random forest of depth $k$. Then, for all $\x, z \in [0,1]^d$,
\begin{align*}
K_{k}(\x,z) = \sum\limits_{\substack{k_{1},\hdots,k_{d} \\ \sum_{\ell=1}^d k_{\ell} = k}} 
\frac{k!}{k_{1}! \hdots k_{d} !} \left( \frac{1}{d}\right)^k \prod_{j=1}^d  \mathds{1}_{ \lceil 2^{k_j}x^{(j)} \rceil = \lceil 2^{k_j}z^{(j)} \rceil}.
\end{align*}
\end{lemma}

\begin{theorem}[\cite{scornet2016random}]
\label{bias_theorem_centred_forest} 
Let $f^{\star}$ be a $L$-Lipschitz function. Then, for all $k$,
\begin{align*}
\sup\limits_{\bx \in [0,1]^d} \left| \frac{\int_{[0,1]^d} k_{k}(\x, z) f^\star(z) \diff z_1 \hdots \diff z_d }{\int_{[0,1]^d} k_{k}(\x, z) \diff z_1 \hdots \diff z_d} - f^\star(\x) \right| \leq Ld \left(1 - \frac{1}{2d}\right)^k.
\end{align*}
\end{theorem}

%%%%%%%%%%%%%%%%%%%%%%%%%%%%%%%%%%%%%%%%%%%%%%%%%%%%%%%%%%%%%%%%%%%%%%%%%%%%%%%%%%%%%%%%%%%%%%%%
\begin{proof}[Proof of Theorem \ref{theoreme_consistency_centred_forest_approximation}]
\label{proof:th_consistency_kerf_approx}
Let $x \in [0,1]^d$ and recall that
\begin{align*}
\infforest^{\mathrm{KeRF}}(x) = & \frac{\sum_{i=1}^n Y_i K_k(x, X_i)}{\sum_{i=1}^n  K_k(x, X_i)}.
\end{align*}
Thus, letting 
\begin{align*}
& A_{n}(x) = \frac{1}{n}\sum_{i=1}^n \left( \frac{Y_i K_k(x, \X_i)}{\E \left[ K_k(x, \X) \right]} - \frac{\E \left[ YK_k(x, \X) \right]}{\E \left[ K_k(x, \X) \right]} \right),\\
& B_{n}(x) = \frac{1}{n} \sum_{i=1}^n \left( \frac{K_k(\x, \X_i)}{\E \left[ K_k(x, \X) \right]} - 1 \right),\\
\textrm{and}~ & M_n(x) = \frac{\E \left[ YK_k(x, \X) \right]}{\E \left[ K_k(\x, \X) \right]}, 
\end{align*}
the estimate $\infforest^{\mathrm{KeRF}}(\x)$ can be rewritten as 
\begin{align*}
\infforest^{\mathrm{KeRF}}(\x) = \frac{M_n(\x) + A_n(\x)}{1 + B_n(\x)},
\end{align*}
which leads to
\begin{align*}
\infforest^{\mathrm{KeRF}}(\x) - f^\star(\x)& = \frac{  M_n(\x) - f^\star(\x) + A_n(\x) - B_n(\x) f^\star(\x)}{1 + B_n(\x)}.
\end{align*}
According to Theorem \ref{bias_theorem_centred_forest}, we have
\begin{align*}
|M_n(\x) - f^\star(\x)| & = \left|\frac{\E \left[ f^\star(\bX) K_k(\x, \X) \right]}{\E \left[ K_k(\x, \X) \right]} + \frac{\E \left[ \varepsilon K_k(\x, \X) \right]}{\E \left[ K_k(\x, \X) \right]} - f^\star(\bx) \right|\\
& \leq  \left|\frac{\E \left[ f^\star(\bX) K_k(\x, \X) \right]}{\E \left[ K_k(\x, \X) \right]} - f^\star(\bx) \right|\\
& \leq C \left( 1 - \frac{1}{2d}\right)^k, 
\end{align*}
where $C=Ld$. Take $\alpha\in ]0, 1/2]$. Let $\mathcal{C}_{\alpha}(\bx)$ be the event $\big\lbrace |A_n(\x)|\leq \alpha \big\rbrace \cap \big\lbrace |B_n(\x)| \leq \alpha \big\rbrace$. On the event $\mathcal{C}_{\alpha}(\bx)$, we have
\begin{align*}
|\infforest^{\mathrm{KeRF}}(\x) - f^\star(\x) |^2 & \leq 8 |M_n(\x) - f^\star(\x) |^2 + 8 |A_n(\bx) - B_n(\bx) f^\star(\x) |^2 \nonumber\\
%& \leq 8 |M_n(\x) - f^\star(x) |^2 + 8\alpha^2 (1 + \|m\|_{\infty})^2 \nonumber\\
& \leq 8C^2 \left( 1 - \frac{1}{2d} \right)^{2k} + 8\alpha^2 (1 + \|f^\star\|_{\infty})^2.
\end{align*}

Thus, 
\begin{align}
\E [ |\infforest^{\mathrm{KeRF}}(\x) - f^\star(\x)|^2 \mathds{1}_{\mathcal{C}_{\alpha}(\bx)} ] & \leq 8C^2 \left( 1 - \frac{1}{2d} \right)^{2k} + 8\alpha^2 (1 + \|f^\star\|_{\infty})^2.\label{equation_proof_rate_consistency_centred}
\end{align}

Consequently, to find an upper bound on the rate of consistency of $\infforest^{\mathrm{KeRF}}$, we just need to upper bound 
\begin{align*}
\E \Big[ |\infforest^{\mathrm{KeRF}}(\x) - f^\star(x)|^2 \mathds{1}_{\mathcal{C}^c_{\alpha}(\bx)} \Big] 
& \leq \E \Big[ \Big|\max\limits_{1 \leq i \leq n} |Y_i|  + |f^\star(x)|\Big|^2 \mathds{1}_{\mathcal{C}^c_{\alpha}(\bx)} \Big] \nonumber \\
& \quad \textrm{(since  $\infforest^{\mathrm{KeRF}}$ is a local averaging estimate)}\nonumber\\
& \leq \E \Big[ \Big|2 \|f^{\star}\|_{\infty} + \max\limits_{1 \leq i \leq n} |\varepsilon_i| \Big|^2 \mathds{1}_{\mathcal{C}^c_{\alpha}(\bx)} \Big] \nonumber\\
& \leq \left( \E \left[ 2\|f^{\star}\|_{\infty} + \max\limits_{1 \leq i \leq n} |\varepsilon_i|\right]^4  \P \left[ \mathcal{C}^c_{\alpha}(\bx) \right] \right)^{1/2} \nonumber \\
& \quad \textrm{(by Cauchy-Schwarz inequality)} \nonumber \\
& \leq \left(  \left( 16 \|f^{\star}\|_{\infty}^4 + 8 \E \Big[ \max\limits_{1 \leq i \leq n} |\varepsilon_i|\Big]^4 \right) \P \left[ \mathcal{C}^c_{\alpha}(\bx) \right] \right)^{1/2}. \nonumber
\end{align*}

According to Lemma~\ref{lem:max_Gaussian_variables}, there exists a constant $C'>0$ such that, for all $n$, 
\begin{align}
\E \Big[ \max\limits_{1 \leq i \leq n} \varepsilon_i^4 \Big] \leq C' \sigma^4 (\log n)^2.
\end{align}
Thus, there exists $C''$ such that, for all $n >1$, 
\begin{align}
\E \Big[ |\infforest^{\mathrm{KeRF}}(\x) - f^\star(x)|^2 \mathds{1}_{\mathcal{C}^c_{\alpha}(\bx)} \Big] 
& \leq C'' \sigma^2 (\log n) (\P \left[ \mathcal{C}^c_{\alpha}(\bx) \right] )^{1/2}. \label{proba_centred_ineq}
\end{align}
The last probability $\P \left[ \mathcal{C}^c_{\alpha}(\bx) \right] $ can be upper bounded by using Chebyshev's inequality. Indeed, with respect to $A_n(\bx)$, 
\begin{align}
\P \big[ |A_n(\bx)| > \alpha \big] 
& \leq \frac{1}{n \alpha^2 }\E \bigg[ \frac{Y K_k(\x, \X)}{\E \left[ K_k(\x, \X) \right]} - \frac{\E \left[ YK_k(\x, \X) \right]}{\E \left[ K_k(\x, \X) \right]} \bigg]^2 \nonumber \\
& \leq \frac{1}{n \alpha^2 } \frac{1}{(\E \left[ K_k(\x, \X) \right])^2}\E \bigg[ Y^2 K_k(\x, \X)^2 \bigg] \nonumber \\
& \leq \frac{2}{n \alpha^2 } \frac{1}{(\E \left[ K_k(\x, \X) \right])^2}\bigg( \E \bigg[ f^\star(\bX)^2 K_k(\x, \X)^2 \bigg] \nonumber \\
& \qquad +\E \bigg[ \varepsilon^2 K_k(\x, \X)^2 \bigg]\bigg) \nonumber \\
&\leq \frac{2(\|f^\star\|_{\infty}^2+\sigma^2)}{n \alpha^2 } \frac{\E \left[ K_k(\x, \X)^2 \right]}{(\E \left[ K_k(\x, \X) \right])^2} \\
&= \frac{C_0}{n \alpha^2} \frac{\E \left[ K_k(\x, \X)^2 \right]}{(\E \left[ K_k(\x, \X) \right])^2}\label{eq:borne_an} 
\end{align}
with $C_0 = 2(\|f^\star\|_{\infty}^2+\sigma^2) $ a constant. Meanwhile with respect to $B_n(\bx)$, we obtain, still by Chebyshev's inequality, 
\begin{align}
\P \big[ |B_n(\bx)| > \alpha \big] & \leq \frac{1}{n \alpha^2} \frac{\E \left[ K_k(\x, \X)^2 \right]}{(\E \left[ K_k(\x, \X) \right])^2}\label{eq:borne_bn} 
\end{align}
which matches the control made by \cite{scornet2016random}. Consequently, 
\begin{align}
\P \left[ \mathcal{C}^c_{\alpha}(\bx) \right] & \leq \P \big[ |A_n(\bx)| > \alpha \big] + \P \big[ |B_n(\bx)| > \alpha \big]\\
& \leq \frac{C_0+1}{n \alpha^2} \frac{\E \left[ K_k(\x, \X)^2 \right]}{(\E \left[ K_k(\x, \X) \right])^2}. 
\end{align}
Besides, for all $x \in [0,1]^d$, for all $k$, $\Esp{\K} = \frac{1}{2^k}$ (see in \cite{scornet2016random} the proof of theorem VI.1 p.11). Since $K_k(\x, \X) \leq 1$, we know that
\begin{align}
    \Esp{\K} = \frac{1}{2^k} \hspace{0.1cm} \geq \hspace{0.1cm} \Esp{\K^2} \hspace{0.1cm} \geq  \hspace{0.1cm} (\Esp{\K})^2 = \frac{1}{2^{2k}}, \label{eq_kerf_kernel_centered}
\end{align}
which leads to 
\begin{align}
\P \left[ \mathcal{C}^c_{\alpha}(\bx) \right] & \leq 2^{2k} \left(\frac{C_0+1}{n \alpha^2}\right) \E \left[ K_k(\x, \X)^2 \right], \label{eq:borne_proba_bef_lem}
\end{align}
but to pursue, we need a tighter upper bound on $\Esp{\K^2}$ than that obtained from \eqref{eq_kerf_kernel_centered}. Such a control is provided in Lemma~\ref{lem:esp_square} below, which is original, and departs from the  
work of \cite{scornet2016random}. 
\begin{lemma}
\label{lem:esp_square}
For all $d \geq 2$, for all $k$ large enough, for all $x \in [0,1]^d$,  %$\geq \max (2d, 2^{6/(d-1)}),$
\begin{align}
     \Esp{\K^2} &  \leq  2^{-k} k^{- \frac{d-1}{2}} \left( C_1 + C_2\left(  \log_2(k)\right)^d \right),
\end{align}
where 
\begin{align}
C_1  = 1 + \frac{2 d^{d/2}}{(4 \pi )^{(d-1)/2}} \quad \textrm{and} \quad C_2 = 5^d \left( \frac{d-1}{2}\right)^d.
%C_1 = \frac{2 d^{d/2}}{(4 \pi )^{(d-1)/2}} \quad \textrm{and} \quad C_2 = 5^d \left( \frac{d-1}{2}\right)^d.
\end{align}
%$$\E \left[ K_k(\x, \X)^2 \right] \lesssim 2^{-k} \left( k^{-(d-1)/2} +2^{-k^{1/d}} + k^{-(d-3)/2} \right)$$ 
%where $C_2$ is a constant depending only on $d$ and %$v_k \approx  C_1/(2^k k^{(d-1)/2}) $ with $C_1$ also a constant depending only on $d$.
\end{lemma}

\begin{proof}[Proof of Lemma \ref{lem:esp_square}]
From Lemma \ref{lemme_centred_random_forest}, we know that
\begin{align}
     \Esp{\K^2} = \Esp{\left( \multi \frac{k!}{k_1!...k_d!} \left(\frac{1}{d}\right)^k \displaystyle \prod_{j=1}^d \ind{\lceil 2^{k_j}x^{(j)} \rceil = \lceil 2^{k_j}X^{(j)} \rceil}\right)^2}.
\end{align}

Developing the square within the expectation, we obtain two terms, the first one $A$ being the sum of squares and the second one, $B$, being the cross-product terms. The first term $A$ takes the form
\begin{align}
    A &:= \Esp{\multi \left( \frac{k!}{k_1!...k_d!}\right)^2 \left(\frac{1}{d}\right)^{2k} \displaystyle \prod_{j=1}^d \ind{\lceil 2^{k_j}x^{(j)} \rceil = \lceil 2^{k_j}x^{(j)} \rceil}} \\
    &= \multi \left( \frac{k!}{k_1!...k_d!}\right)^2 \left(\frac{1}{d}\right)^{2k} \displaystyle \prod_{j=1}^d \Prob{\lceil 2^{k_j}x^{(j)} \rceil = \lceil 2^{k_j}X^{(j)} \rceil}.
\end{align}
Note that, for all $j$, $\Prob{\lceil 2^{k_j}x^{(j)} \rceil = \lceil 2^{k_j}X^{(j)} \rceil} = 2^{-k_j},$ and
$\displaystyle \prod_{j=1}^d  2^{-k_j} = 2^{-k}.$
Therefore,
\begin{align}
    A &= \multi \left( \frac{k!}{k_1!...k_d!}\right)^2 \left(\frac{1}{d}\right)^{2k} 2^{-k}. 
    %&= \left(\frac{1}{d}\right)^{2k}  \left( \frac{1}{2} \right)^{k} \multi \left( \frac{k!}{k_1!...k_d!}\right)^2.
\end{align}

Thanks to \cite{richmond2008counting}, we know that, for all $d \geq 2$, 
\begin{align}
    \multi \left( \frac{k!}{k_1!...k_d!}\right)^2 \mathop{\sim}_{k \to + \infty} \frac{d^{2k+d/2}}{(4 \pi k)^{(d-1)/2}}.
\end{align}

Therefore, for all $k$ large enough, we have 
\begin{align}
    \multi \left( \frac{k!}{k_1!...k_d!}\right)^2 \leq  \frac{2 d^{2k+d/2}}{(4 \pi k)^{(d-1)/2}}.
\end{align}
 Thus, letting $C_1 = 2 d^{d/2}/(4 \pi )^{(d-1)/2}$, for all $k$ large enough, 
%\begin{align}
%    \multi \left( \frac{k!}{k_1!...k_d!}\right)^2 \leq  \frac{2 d^{2k+d/2}}{(4 \pi k)^{(d-1)/2}}.
%\end{align}
%Consequently, for all $k$, 
%Consequently, there exists a constant $C_1$ such that, for all $k$,
%\begin{align*}
%     \multi \left( \frac{k!}{k_1!...k_d!}\right)^2 \leq  \frac{C_1 d^{2k}}{k^{(d-1)/2}}
%\end{align*}
%and
\begin{align}
    A \leq C_1 2^{-k} k^{-(d-1)/2}.
\end{align}
%since the function $k \mapsto  \multi \left( \frac{k!}{k_1!...k_d!}\right)^2$ is non-decreasing, according to Lemma~\ref{lem_multi_carre}.
%\begin{align}
%C_1 = \frac{2 d^{d/2}}{(4 \pi )^{(d-1)/2}}.
%\end{align}

Regarding the second term $B$,  
\begin{align}
    B &:= \Esp{\sum\limits_{\substack{(k_{1},\hdots,k_{d}) \\ \neq (\ell_1, \hdots, \ell_d), \\ \sum_{j=1}^d k_{j} = \sum_{j=1}^d \ell_{j} =k } } \frac{k!}{k_{1}! \hdots k_{d} !} \frac{k!}{\ell_{1}! \hdots \ell_{d} !} \left( \frac{1}{d}\right)^{2k} \prod_{j=1}^d  \mathds{1}_{ \lceil 2^{k_j}x^{(j)} \rceil = \lceil 2^{k_j}X^{(j)} \rceil} \mathds{1}_{ \lceil 2^{\ell_j}x^{(j)} \rceil = \lceil 2^{\ell_j}X^{(j)} \rceil}} \\
    &= \sum\limits_{\substack{(k_{1},\hdots,k_{d}) \\ \neq (\ell_1, \hdots, \ell_d), \\ \sum_{j=1}^d k_{j} = \sum_{j=1}^d \ell_{j} =k } } \frac{k!}{k_{1}! \hdots k_{d} !} \frac{k!}{\ell_{1}! \hdots \ell_{d} !} \left( \frac{1}{d}\right)^{2k} \Prob{\displaystyle \bigcap_{j=1}^d \left(  (\lceil 2^{k_j}x^{(j)} \rceil = \lceil 2^{k_j}X^{(j)} \rceil) \cap  (\lceil 2^{\ell_j}x^{(j)} \rceil = \lceil 2^{\ell_j}X^{(j)} \rceil) \right)}. \nonumber
\end{align}

A small computation yields
\begin{align}
    &\Prob{\displaystyle \bigcap_{j=1}^d \left(  (\lceil 2^{k_j}x^{(j)} \rceil = \lceil 2^{k_j}X^{(j)} \rceil) \cap  (\lceil 2^{\ell_j}x^{(j)} \rceil = \lceil 2^{\ell_j}X^{(j)} \rceil) \right)} \nonumber \\
    &= \Prob{\displaystyle \bigcap_{j=1}^d \lceil 2^{\ell_j}x^{(j)} \rceil = \lceil 2^{\ell_j}X^{(j)} \rceil \bigg| \forall j,  \lceil 2^{k_j}x^{(j)} \rceil = \lceil 2^{k_j}X^{(j)} \rceil } 2^{-k} \\
    &= 2^{-k}\prod_{j=1}^d \Prob{\lceil 2^{\ell_j}x^{(j)} \rceil = \lceil 2^{\ell_j}X^{(j)} \rceil  \bigg| \lceil 2^{k_j}x^{(j)} \rceil = \lceil 2^{k_j}X^{(j)} \rceil} \\
    &= 2^{-k} 2^{-\sum_{j=1}^d (\ell_j-k_j) \ind{\ell_j \geq k_j}} \\
    &= 2^{- \sum_{j=1}^d k_j (\ind{\ell_j \geq k_j}+ \ind{\ell_j<k_j}) - \sum_{j=1}^d (\ell_j-k_j) \ind{\ell_j \geq k_j}} \\
    &= 2^{- \sum_{j=1}^d k_j \ind{\ell_j<k_j} - \sum_{j=1}^d \ell_j \ind{\ell_j \geq k_j}} \\
    &= 2^{- \sum_{j=1}^d \max(k_j, \ell_j)}.
\end{align}

Therefore,
\begin{align}
    B &= \left( \frac{1}{d}\right)^{2k} \sum\limits_{\substack{(k_{1},\hdots,k_{d}) \\ \neq (\ell_1, \hdots, \ell_d), \\ \sum_{j=1}^d k_{j} = \sum_{j=1}^d \ell_{j} =k } } \frac{k!}{k_{1}! \hdots k_{d} !} \frac{k!}{\ell_{1}! \hdots \ell_{d} !} \left(\frac{1}{2}\right)^{\sum_{j=1}^d \max(k_j,\ell_j)}. \label{eq:crossed_term_sum} \\
    &= \left( \frac{1}{d}\right)^{2k} \sum\limits_{\substack{(k_{1},\hdots,k_{d}) \\ \neq (\ell_1, \hdots, \ell_d), \\ \sum_{j=1}^d k_{j} = \sum_{j=1}^d \ell_{j} =k } } \frac{k!}{k_{1}! \hdots k_{d} !} \frac{k!}{\ell_{1}! \hdots \ell_{d} !} \left(\frac{1}{2}\right)^{ k + \frac{1}{2} \sum_{j=1}^d  |k_j-\ell_j|} \\
    &= \left( \frac{1}{2 d^2}\right)^{k} \sum\limits_{\substack{(k_{1},\hdots,k_{d}) \\ \neq (\ell_1, \hdots, \ell_d), \\ \sum_{j=1}^d k_{j} = \sum_{j=1}^d \ell_{j} =k } } \frac{k!}{k_{1}! \hdots k_{d} !} \frac{k!}{\ell_{1}! \hdots \ell_{d} !} \left(\frac{1}{2}\right)^{ \frac{1}{2} \sum_{j=1}^d  |k_j-\ell_j|}.  \label{eq:crossed_term_sum}
\end{align}

%We fix $k_1,\hdots, k_d$ such that $\sum_{j=1}^d k_j = k$.
%Then, considering the sum Line \eqref{eq:crossed_term_sum} over $\ell_1,\hdots, \ell_d$ for the previously fixed $k_1,\hdots, k_d$, 

For all $q >0$, define the set $\mathcal{K}_q = \{ \boldsymbol{\ell} = (\ell_1,\hdots, \ell_d), \boldsymbol{k} = (k_1,\hdots, k_d) | \sum\limits_{j=1}^d |k_j-\ell_j| \geq 2q\}$, so that
\begin{align}
    B &= \left( \frac{1}{2d^2}\right)^k \sum\limits_{\substack{(\boldsymbol{k,\ell}) \in \mathcal{K}_q \\ \boldsymbol{\ell} \neq \boldsymbol{k} \\ \sum_{j=1}^d k_{j} = \sum_{j=1}^d \ell_{j} =k } } \frac{k!}{k_{1}! \hdots k_{d} !} \frac{k!}{\ell_{1}! \hdots \ell_{d} !} \left(\frac{1}{2}\right)^{\frac{1}{2}\sum_{j=1}^d |k_j-\ell_j|} \nonumber \\
    & \quad + \left( \frac{1}{2d^2}\right)^k \sum\limits_{\substack{(\boldsymbol{k,\ell}) \notin \mathcal{K}_q \\ \boldsymbol{\ell} \neq \boldsymbol{k} \\ \sum_{j=1}^d k_{j} = \sum_{j=1}^d \ell_{j} =k } } \frac{k!}{k_{1}! \hdots k_{d} !} \frac{k!}{\ell_{1}! \hdots \ell_{d} !} \left(\frac{1}{2}\right)^{\frac{1}{2}\sum_{j=1}^d |k_j-\ell_j|} \nonumber \\
    &=  B_1 + B_2.
\end{align}

Regarding $B_1$, we have
\begin{align}
    B_1 &\leq \left( \frac{1}{2d^2}\right)^k \sum\limits_{\substack{(\boldsymbol{k,\ell}) \in \mathcal{K}_q \\ \boldsymbol{\ell} \neq \boldsymbol{k} \\ \sum_{j=1}^d k_{j} = \sum_{j=1}^d \ell_{j} =k } } \frac{k!}{k_{1}! \hdots k_{d} !} \frac{k!}{\ell_{1}! \hdots \ell_{d} !} 2^{-q} \\
    &\leq \left( \frac{1}{2d^2}\right)^k 2^{-q} \left(\sum\limits_{\substack{\boldsymbol{k},  \sum_{j=1}^d k_{j} = k } } \frac{k!}{k_{1}! \hdots k_{d} !}\right) \left( \sum\limits_{\substack{\boldsymbol{\ell}, \sum_{j=1}^d \ell_{j} =k} }\frac{k!}{\ell_{1}! \hdots \ell_{d} !}\right) \\
    &\leq 2^{-k -q}, 
\end{align}
as 
\begin{align}\sum\limits_{\substack{\boldsymbol{k},  \sum_{j=1}^d k_{j} = k } } \frac{k!}{k_{1}! \hdots k_{d} !} =d^k.
\end{align}

We now define, for all $\boldsymbol{k}, \mathcal{K}_q(\boldsymbol{k}) := \{ \boldsymbol{\ell} = (\ell_1,\hdots, \ell_d), \sum\limits_{j=1}^d \ell_j = k, \sum\limits_{j=1}^d |k_j-\ell_j| \geq 2q\}$.
Regarding $B_2$, we have
\begin{align}
    B_2 &\leq \left( \frac{1}{2d^2}\right)^k \sum\limits_{\substack{\boldsymbol{k,\ell} \notin \mathcal{K}_q \\ \boldsymbol{\ell} \neq \boldsymbol{k} \\ \sum_{j=1}^d k_{j} = \sum_{j=1}^d \ell_{j} =k } } \frac{k!}{k_{1}! \hdots k_{d} !} \frac{k!}{\ell_{1}! \hdots \ell_{d} !} \\
    &= \left( \frac{1}{2d^2}\right)^k \sum\limits_{\substack{\boldsymbol{k}, \sum_{j=1}^d k_{j} = k } } \frac{k!}{k_{1}! \hdots k_{d} !} \sum\limits_{\substack{\substack{\boldsymbol{\ell} \notin \mathcal{K}_q(\boldsymbol{k}) \\ \boldsymbol{\ell} \neq \boldsymbol{k} \\ \sum_{j=1}^d k_{j} = \sum_{j=1}^d \ell_{j} =k } }} \frac{k!}{\ell_{1}! \hdots \ell_{d} !}. \\
    %&\leq \left( \frac{1}{2d^2}\right)^k \sum\limits_{\substack{\boldsymbol{k}, \sum_{j=1}^d k_{j} = k } } \frac{k!}{k_{1}! \hdots k_{d} !} \sum\limits_{\substack{\substack{\boldsymbol{l} \notin \mathcal{K}_q(\boldsymbol{k}) \\ \boldsymbol{l} \neq \boldsymbol{k} \\ \sum_{j=1}^d k_{j} = \sum_{j=1}^d \ell_{j} =k } }} \frac{k!}{\Gamma(k/d)^d} \\
    %&\leq  \left( \frac{1}{2d}\right)^k \sum\limits_{\substack{\boldsymbol{k}, \sum_{j=1}^d k_{j} = k } } \frac{k!}{k_{1}! \hdots k_{d} !} \sum\limits_{\substack{\substack{\boldsymbol{l} \notin \mathcal{K}_q(\boldsymbol{k}) \\ \boldsymbol{l} \neq \boldsymbol{k} \\ \sum_{j=1}^d k_{j} = \sum_{j=1}^d \ell_{j} =k } }} C k^{-(d-1)/2}
\end{align}
%where the last inequality is obtained from \cite{batir2008inequalities}.

Note that for all $\boldsymbol{\ell}$, $ \frac{k!}{\ell_{1}! \hdots \ell_{d} !}$ is maximal when $\max_i \ell_i$ is minimal. Therefore, for all $k\geq 2d$,
\begin{align}
    \frac{k!}{\ell_{1}! \hdots \ell_{d} !} &= \frac{k!}{\Gamma(\ell_1+1) \hdots \Gamma(\ell_d+1)} \\
    &\leq \frac{k!}{\Gamma(\lfloor k/d \rfloor +1 ) \hdots \Gamma(\lfloor k/d \rfloor +1 )} \\
    &\leq \frac{k!}{\Gamma(k/d)^d}.
\end{align}
Using an inequality from \cite{batir2008inequalities}, we obtain
\begin{align*}
    \frac{k!}{\Gamma(k/d)^d} & \leq \frac{k^{k+1/2}e^{-k} }{k^k d^{-k} e^{-k} k^{d/2}} \\
    &\leq d^{k} k^{-(d-1)/2}.
\end{align*}

Overall, for all $k\geq 2d$,
\begin{align}
    B_2 & \leq \left( \frac{1}{2d}\right)^k \sum\limits_{\substack{\boldsymbol{k}, \sum_{j=1}^d k_{j} = k } } \frac{k!}{k_{1}! \hdots k_{d} !} \sum\limits_{\substack{\substack{\boldsymbol{\ell} \notin \mathcal{K}_q(\boldsymbol{k}) \\ \boldsymbol{\ell} \neq \boldsymbol{k} \\ \sum_{j=1}^d k_{j} = \sum_{j=1}^d \ell_{j} =k } }} k^{-(d-1)/2}\\
    & \leq  k^{-(d-1)/2} \left( \frac{1}{2d}\right)^k \sum\limits_{\substack{\boldsymbol{k}, \sum_{j=1}^d k_{j} = k } } \frac{k!}{k_{1}! \hdots k_{d} !}   \textrm{Card} (\mathcal{K}_q(\boldsymbol{k})).
\end{align}

We now want to upper bound the cardinal of $\mathcal{K}_q(\boldsymbol{k})$.
%As shown in \cite{chandrasekharan1967lattice}, in a ball of radius $c$ in dimension $d$, the number of integer lattice points scales in $O(c^d)$ as $c$ tends to infinity. Therefore, as $c=k^{1/d}$ here,
Denoting by $B_{L_1}(0,2q)$ the ball of radius $2q$ with respect to the $L_1$ norm, note that 
\begin{align}
\textrm{Card} (\mathcal{K}_q(\boldsymbol{k})) & \leq
\textrm{Card} (\{x \in \mathbb{N}^d \cap B_{L_1}(\boldsymbol{k},2q)\})\\
&  \leq \textrm{Card} (\{x \in \mathbb{N}^d \cap B_{L_1}(0,2q)\}).
\end{align}
Since, 
\begin{align*}
B_{L_1}(0,c) \subset  B_{L_{\infty}}(0,c)  \subset B_{L_{\infty}}(0, \lceil c \rceil),
\end{align*}
we have, 
\begin{align*}
\textrm{Card} (\mathcal{K}_q(\boldsymbol{k})) & \leq \textrm{Card} (\{x \in \mathbb{N}^d \cap B_{L_{\infty}}(0, \lceil 2q \rceil)\}) \\
& \leq \left( 2 \lceil 2q \rceil +1\right)^d\\
& \leq \left( 4q  +3\right)^d.
\end{align*}
Thus, we have, for all $k \geq 2d$, 
\begin{align}
    B_2 
    & \leq  k^{-(d-1)/2} \left( \frac{1}{2d}\right)^k \left( 4q  +3\right)^d\sum\limits_{\substack{\boldsymbol{k}, \sum_{j=1}^d k_{j} = k } } \frac{k!}{k_{1}! \hdots k_{d} !}\\
    & \leq k^{-(d-1)/2}  \left( 4q  +3\right)^d 2^{-k},
\end{align}
as 
\begin{align}\sum\limits_{\substack{\boldsymbol{k},  \sum_{j=1}^d k_{j} = k } } \frac{k!}{k_{1}! \hdots k_{d} !} =d^k.
\end{align}
Finally, for all $q$, we have 
\begin{align}
    B & = B_1 + B_2 \\
    & \leq 2^{-k-q} + k^{-(d-1)/2}  \left( 4q  +3\right)^d 2^{-k}.
\end{align}
Let $q = \left( \frac{d-1}{2}\right) \log_2(k)$. For all $q \geq 3$, that is for all $k \geq 2^{6/(d-1)}$, and for all $k \geq 2d$, 
\begin{align}
    B & \leq 2^{-k} \left( k ^{- \frac{d-1}{2}} + k^{-(d-1)/2}  C_2  \left(  \log_2(k)\right)^d\right),
\end{align}
where 
\begin{align}
C_2 = 5^d \left( \frac{d-1}{2}\right)^d.
\end{align}
Finally, for all $k$ large enough 
\begin{align}
     \Esp{\K^2} & \leq A + B_1 + B_2 \\
     & \leq  2^{-k} k^{- \frac{d-1}{2}} \left( C_1 +1 + C_2\left(  \log_2(k)\right)^d \right) .
     %&\leq 2^{-k} k^{- \frac{d-1}{2}} \left( C_1' + C_2\left(  \log_2(k)\right)^d \right)
\end{align}
%where $C'1 = C_1+1$

%Choosing $q = k^{1/d}$, we obtain, for all $k \geq 3^d$, 
%Thus, the number of points in $B_{L_2}(0,c)$ with integer coordinates is upper 
%\begin{align}
 %   B_2 
    %&\leq \left( \frac{1}{2d}\right)^k \sum\limits_{\substack{\boldsymbol{k}, \sum_{j=1}^d k_{j} = k } } \frac{k!}{k_{1}! \hdots k_{d} !}  k^{-(d-1)/2} k  \\
%    & \leq 5^d 2^{-k}k^{-(d-3)/2}
%\end{align}

%Finally, for $k$ large enough,
%\begin{align}
%    \E \left[ K_k(\x, \X)^2 \right] \leq C_1 2^{-k} \left( k^{-(d-1)/2} +2^{-k^{1/d}} + k^{-(d-3)/2} \right)
%\end{align}
%where $C_1 = \max\left(\frac{2d^{d/2}}{(4\pi)^{(d-1)/2}}, 5^d\right)$ is a constant depending only on $d$.
\end{proof}

According to inequality~\eqref{eq:borne_proba_bef_lem} and Lemma~\ref{lem:esp_square}, we have, for all $k$ large enough 
\begin{align}
\P \left[ \mathcal{C}^c_{\alpha}(\bx) \right] & \leq  \frac{C_0+1}{n \alpha^2} 2^{k} k^{- \frac{d-1}{2}} \left( C_1 + C_2\left(  \log_2(k)\right)^d \right). 
\end{align}
%\begin{align}
%    \Prob{|A_n(\bx)| > \alpha} &\leq  C_0 \frac{2^k}{n\alpha^2 }  k^{- \frac{d-1}{2}} \left( C_1' + C_2\left(  \log_2(k)\right)^d \right)
%\end{align}
%where $v_k \approx C_d/(k^{(d-1)/2})$, and
%and
%\begin{align*}
%    \Prob{|B_n(x)| > \alpha} &\leq \frac{2^{k}}{n\alpha^2}  k^{- \frac{d-1}{2}} \left( C_1' + C_2\left(  \log_2(k)\right)^d \right).
%\end{align*}
%Thus, 
%\begin{align*}
%\P \big[\mathcal{C}_{\alpha}^c(\bx) \big] & \leq   2 C_0 \frac{2^{k}}{\alpha^2 n} k^{- \frac{d-1}{2}} \left( C_1' + C_2\left(  \log_2(k)\right)^d \right).
%\end{align*}
Consequently, according to inequality (\ref{proba_centred_ineq}), we obtain, for all $k$ large enough 
\begin{align*}
\E \Big[ |\infforest^{\mathrm{KeRF}}(\x) - f^\star(x)|^2 \mathds{1}_{\mathcal{C}^c_{\alpha}(\bx)} \Big] 
& \leq C'' \sigma^2 \log n \left( \frac{C_0+1}{n \alpha^2} 2^{k} k^{- \frac{d-1}{2}} \left( C_1 + C_2\left(  \log_2(k)\right)^d \right) \right)^{1/2}\\
& \leq C'' \sigma^2 (C_0+1)^{1/2} (\max(C_1, C_2))^{1/2} \frac{\log n}{n^{1/2} \alpha} 2^{k/2} k^{- \frac{d-1}{4}} \left(    \left( 1 + \left(  \log_2(k)\right)^d \right) \right)^{1/2}\\
& \leq C_3  \frac{\log n}{n^{1/2} \alpha} 2^{k/2} k^{- \frac{d-1}{4}} \left(  \log_2(k)\right)^{d/2}  ,
%& \leq C'' \log n \frac{2^{k/2}}{\alpha n^{1/2}}\left( C_0 k^{- \frac{d-1}{2}} \left( C_1' + C_2\left(  \log_2(k)\right)^d \right)\right)^{1/2}. 
\end{align*}
where $C_3 = C'' \sigma^2 (C_0+1)^{1/2} (2 \max(C_1, C_2))^{1/2}$.
Then using inequality (\ref{equation_proof_rate_consistency_centred}), for all $k$ large enough 
\begin{align*}
& \E \Big[ \infforest^{\mathrm{KeRF}}(\x) - f^\star(x) \Big]^2 \\
& \leq \E \Big[ |\infforest^{\mathrm{KeRF}}(\x) - f^\star(x)|^2 \mathds{1}_{\mathcal{C}_{\alpha}(\bx)} \Big] + \E \Big[ |\infforest^{\mathrm{KeRF}}(\x) - f^\star(x)|^2 \mathds{1}_{\mathcal{C}^c_{\alpha}(\bx)} \Big] \\
& \leq 8L^2 d^2 \left( 1 - \frac{1}{2d} \right)^{2k} + 8 \alpha^2 (1 + \|f^{\star}\|_{\infty})^2 \\
& \quad +C_3 \sigma^2 (\log n) \frac{2^{k/2}}{\alpha n^{1/2}} k^{- \frac{d-1}{4}} (\log_2 k)^{d/2}.
\end{align*}
Optimizing the right hand side in $\alpha$, that is choosing 
% \begin{align}
%  \alpha^3= (\log n)  \frac{2^{k/2}}{ n^{1/2}}  k^{- \frac{d-1}{4}} (\log_2 k)^{d/2} \frac{C_3}{(1+\|f^{\star}\|_{\infty})^{2}},   
% \end{align}
\begin{align}
 \alpha^3= (\log n)  \frac{2^{k/2}}{ n^{1/2}}  k^{- \frac{d-1}{4}} (\log_2 k)^{d/2} \frac{C_3}{8(1+\|f^{\star}\|_{\infty})^{2}},   
\end{align}
%\cb{j'ai un 16 au dénominateur}
%\la{Pourquoi un 16 au dénominateur ?}
%\es{ici, je pense qu'on s'était contenté d'écrire l'égalité des 2e et 3e terme, ce qui est optimal à une constante près. Anyway, tout calcul me convient}
we get
% \begin{align*}
%  \E \Big[ \infforest^{\mathrm{KeRF}}(\x) - f^\star(x) \Big]^2
% &\leq 8L^2 d^2 \left( 1 - \frac{1}{2d} \right)^{2k} + 9  C_3^{2/3} (1+\|f^{\star}\|_{\infty})^{2/3} (\log n)^{2/3}  \frac{2^{k/3}}{n^{1/3}}k^{- \frac{d-1}{6}} (\log_2 k)^{d/3}.
% \end{align*}
\begin{align*}
 \E \Big[ \infforest^{\mathrm{KeRF}}(\x) - f^\star(x) \Big]^2
&\leq 8L^2 d^2 \left( 1 - \frac{1}{2d} \right)^{2k} + 4 C_3^{2/3} (1+\|f^{\star}\|_{\infty})^{2/3} (\log n)^{2/3}  \frac{2^{k/3}}{n^{1/3}}k^{- \frac{d-1}{6}} (\log_2 k)^{d/3}.
\end{align*}
Choosing $k_n = \log_2(n)$, we obtain, for all $n$ large enough, 
% \begin{align}
%     \E \Big[ \infforest^{\mathrm{KeRF}}(\x) - f^\star(x) \Big]^2 
%      \leq 8L^2 d^2 n^{2\log(1-\frac{1}{d})} + 9  C_3^{2/3} (1+\|f^{\star}\|_{\infty})^{2/3} (\log n)^{2/3} (\log_2 n)^{-\frac{d-1}{6}} (\log_2(\log_2 n))^{d/3}.
% \end{align}
\begin{align}
    \E \Big[ \infforest^{\mathrm{KeRF}}(\x) - f^\star(x) \Big]^2 
     \leq 8L^2 d^2 n^{2\log_2(1-\frac{1}{d})} + 4 C_3^{2/3} (1+\|f^{\star}\|_{\infty})^{2/3} (\log n)^{2/3} (\log_2 n)^{-\frac{d-1}{6}} (\log_2(\log_2 n))^{d/3}.
\end{align}
Finally,
\begin{align}
    \E \Big[ \infforest^{\mathrm{KeRF}}(\x) - f^\star(x) \Big]^2 &\leq 8L^2 d^2 n^{2\log_2(1-\frac{1}{d})} + C_4 (\log_2 n)^{-\frac{d-5}{6}} (\log_2(\log_2 n))^{d/3}.
\end{align}
with 
\begin{align}
 C_4 =  18 \times 2^{2/3} \times  (\log 2)^{2/3}  C''^{2/3} (\|f^\star\|_{\infty}^2+\sigma^2 +1) ( \max(C_1, C_2))^{1/3}.
\end{align}
% \begin{align}
%  C_4 =  4 \times  (\log 2)^{2/3}  C''^{2/3} (\|f^\star\|_{\infty} +1)^{2/3} ( \max(C_1, C_2))^{1/3}.
% \end{align}

\end{proof}

%\begin{lemma} \label{lem_multi_carre}
%The function $k \mapsto  \multi \left( \frac{k!}{k_1!...k_d!}\right)^2$ is non-decreasing on $\mathbb{N}^{\star}$. 
%\end{lemma}

%\begin{proof}[Proof of Lemma~\ref{lem_multi_carre}]
%Consider the set $\mathcal{S}_k:= \{\boldsymbol{k} = (k_1, \hdots, k_d)~: ~ \sum_{j=1}^d k_j=k \}$ and define the function $h : \mathcal{S}_k \to \mathcal{S}_{k+1}$ such that
%\[   
%h(\boldsymbol{k})_j = 
%     \begin{cases}
 %      k_j+1 &\quad\text{if j is the smallest index such that } k_j = \max\limits_{1\leq \ell\leq d} k_\ell \\
 %      k_j &\quad\text{otherwise}
%     \end{cases}
%\]
%It is clear that $h(\mathcal{S}_k) \subset \mathcal{S}_{k+1}$. Besides, the function $h$ is injective. Now, denote $g: \boldsymbol{k} \to \left(\frac{k!}{k_1!...k_d!}\right)^2$, and, without loss of generality, let $\boldsymbol{k}$ such that $k_1 = \max_{1 \leq \ell \leq d} k_{\ell}$. Then, we have
%\begin{align*}
 %   g(h(\boldsymbol{k})) &= g(k_1+1, k_2, ..., k_d) \\
 %   &= \left(\frac{(k+1)!}{(k_1+1)!k_2!...k_d!}\right)^2 \\
 %   &\geq \left(\frac{k!}{k_1!...k_d!}\right)^2 \\
%    &= g(\boldsymbol{k)}.
%\end{align*}

%Therefore,
%\begin{align*}
%    \sum_{\boldsymbol{k} \in \mathcal{S}_k} g(\boldsymbol{k}) &\leq \sum_{\boldsymbol{k} \in \mathcal{S}_k} g(h(\boldsymbol{k})) \\
%    &\leq \sum_{\boldsymbol{k} \in \mathcal{S}_{k+1}} g(\boldsymbol{k})
%\end{align*}
%by injectivity of $h$.
    
%\end{proof}

\begin{lemma} \label{lem:max_Gaussian_variables}
    Consider $n$ i.i.d.  random variables $\varepsilon_1, \hdots, \varepsilon_n$, distributed as $\mathcal{N}(0, 1)$. Then, for all $n \geq 21$, 
\begin{align*}
\E \Big[ \max\limits_{1 \leq i \leq n} \varepsilon_i^4 \Big] \leq 32 e (\log n)^2.
\end{align*}
\end{lemma}

\begin{proof}
 We have, for all $p \geq 1$, 
 \begin{align}
     \mathbb{E} \left[ \max_{1 \leq i \leq n}  |\varepsilon_i|^4 \right] 
     & \le \left(\mathbb{E} \left[ \max_{1 \leq i \leq n}  |\varepsilon_i|^{4p}\right] \right)^{1/p}  \leq \left(\mathbb{E} \left[ \sum_{i=1}^n |\varepsilon_i|^{4p} \right] \right)^{1/p},
 \end{align}
 using Jensen's inequality (by concavity of $x \mapsto x^{1/p}$ for $p \geq 1$).
 The $p$-th moment of a Gaussian variable $\mathcal{N}(0,1)$ can be computed as follows
\begin{align}
    \mathbb{E}\left[ |\varepsilon_1|^p \right] & = \int_0^{\infty} \mathbb{P} \left[ |\varepsilon|^p \geq u \right] \textrm{d}u \\
    & = \int_0^{\infty} \mathbb{P} \left[ |\varepsilon| \geq t \right] p t^{p-1} \textrm{d}t \\
    & \leq \int_0^{\infty} 2 \exp(-t^2/2) p t^{p-1} \textrm{d}t,
\end{align}
using classical tail inequalities for Gaussian variables. Now, setting $s = t^2/2$ and recalling that $\Gamma(z) = \int_0^{\infty} \exp(-t) t^{z-1} \textrm{d}t$, we have
\begin{align}
    \int_0^{\infty} 2 \exp(-t^2/2) p t^{p-1} \textrm{d}t & = 2p \int_0^{\infty} \exp(-s) (2s)^{\frac{p-2}{2}}  \textrm{d}s \\
    & = 2p 2^{\frac{p-2}{2}} \Gamma(p/2).
\end{align}
According to Theorem~2.2 in \citet{batir2008inequalities}, we have, for all $x > 0$
\begin{align}
    \Gamma(x+1) < \sqrt{2\pi} x^x \exp(-x) \left( x^2 +  \frac{x}{3} + \frac{1}{18} \right)^{1/4}.
\end{align}
Let 
\begin{align}
    f : x \mapsto \exp(-x) \left( x^2 +  \frac{x}{3} + \frac{1}{18} \right),
\end{align}
one can show that $f$ is non-increasing on $[1/2, \infty)$. Thus, for all $x \geq 1/2$,
\begin{align}
 \Gamma(x+1) & < \sqrt{2\pi} x^x f(1/2)^{1/4} \\
 & < \sqrt{2\pi} x^x \exp(-1/2) \left( \frac{1}{2} \right)^{1/4}\\
 & < 2 x^x.
\end{align}

Hence, for all $p \geq 3$, 
\begin{align}
    \mathbb{E}\left[ |\varepsilon_1|^p \right] & \leq 4p 2^{\frac{p-2}{2}} (p/2)^{p/2},
\end{align}
which leads to 
\begin{align}
     \mathbb{E} \left[ \max_{1 \leq i \leq n}  |\varepsilon_i|^4 \right] 
& \leq \left(\mathbb{E} \left[ \sum_{i=1}^n |\varepsilon_i|^{4p} \right] \right)^{1/p}\\
& \leq n^{1/p} \left(  16p 2^{\frac{4p-2}{2}} (2p)^{2p} \right)^{1/p}\\
& \leq 16 n^{1/p} p^2 \left( \frac{p}{2}\right)^{1/p}\\
& \leq 32 n^{1/p} p^2. 
 \end{align}
 Choosing $p = \log n$ yields, for all $ n \geq e^3,$
%  \begin{align}
% \mathbb{E} \left[ \max_{1 \leq i \leq n}  |\varepsilon_i|^4 \right] & \leq 32 (\log n)^2.
%  \end{align}
  \begin{align}
\mathbb{E} \left[ \max_{1 \leq i \leq n}  |\varepsilon_i|^4 \right] & \leq 32 e (\log n)^2.
 \end{align}

\end{proof}

%%%%%%%%%%%%%%%%%%%%%%%%%%%%%%%%%%%%%%%%%%%%%%%%%%%%%%%%%%%%%%%%%%%%%%%%%%%%%%%%%%%%%%%%%%%%%%%%%%%

%%%%%%%%%%%%%%%%%%%%%%%%%%%%%%%%%%%%%%%%%%%%%%%%%%%%%%%%%%%%%%%%%%%%%%%%%%%%%%%%%%%%%%%%%%%%%%
%%%%%%%%%%%%%%%%%%%%%%%%%%%%%%%%%%%%%%%%%%%%%%%%%%%%%%%%%%%%%%%%%%%%%%%%%%%%%%%%%%%%%%%%%%%%%%

\subsection{Proofs of Section \ref{sec:AdaCRF} (Semi-adaptive forests)}

\begin{lemma}%[Depth of an adaptive centered tree]
\label{lemma:control_proba}
For all $\alpha \in [0,1)$, the depth $k_n^{\mathrm{AdaCT}}$ of a semi-adaptive centered tree verifies
\begin{align*}
\lim\limits_{n \to \infty} \Prob{k_n^{\mathrm{AdaCT}}(X, \Theta) \in [\log_2(n) \pm  \log_2^{1-\alpha}(n)]} = 1.
\end{align*}
\end{lemma}

Lemma~\ref{lemma:control_proba} states that the asymptotic behavior of $k_n^{\mathrm{AdaCT}}(X, \Theta)$ is equivalent to $\log_2 n$ up to a negligible factor.
The $\log(n)$ equivalent matches the condition for the mean interpolation regime in the case of CRF exhibited in Section \ref{sec:CRF}.

\subsubsection{Proof of Lemma \ref{lemma:control_proba}}
\label{proof:lemma_adacrf_depth}
For all $0 \leq j \leq k$, we let $A_{j,n}(X, \Theta)$ be the cell containing $X$ in the tree truncated at level $j$. Similarly, we let $N_{j,n}(X, \Theta)$ the number of observations in this cell. Then, 
\begin{align}
    \Prob{k_n(X, \Theta) \geq k} %&= \Prob{\bigcap_{i=1}^{k-1} N_{i,\Theta}(X) \geq 2} \\
    %&= \Prob{N_{1, \Theta} \geq 2} \prod_{i=2}^{k-1} \Prob{N_{i,\Theta}(X) \geq 2 | N_{i-1,\Theta}(X) \geq 2} \\
    &= \Prob{N_{k-1,n}(X, \Theta) \geq 2} \\
    &= \Esp{\Prob{N_{k-1,n}(X, \Theta) \geq 2 | X, \Theta}} \\
    &= 1 - \left(1 - \frac{1}{2^{k-1}}\right)^n - \frac{n}{2^{k-1}}\left(1 - \frac{1}{2^{k-1}}\right)^{n-1}. \label{eq:proba_controled}
\end{align}

%\item 
%We denote $F_{k_n(X)}$ the CDF of $k_n(X)$. We have, for all $k$, for all $x \in [k-1, k[$,
%\begin{align}
%    F_{k_n(X)} (x) &=  \left(1 - \frac{1}{2^{k-1}}\right)^n + \frac{n}{2^{k-1}}\left(1 - \frac{1}{2^{k-1}}\right)^{n-1}  \\*
%    &\leq  \left(1 - \frac{1}{2^{x}}\right)^n + \frac{n}{2^{x}}\left(1 - \frac{1}{2^{x}}\right)^{n-1}.
%\end{align}
%We denote $G : x \in [0, + \infty) \to (1-2^{-x})^{n} - n2^{-x}(1-2^{-x})^{n-1}$. Using that %$\Esp{k_n(X)} = \int 1 - F_{k_n(X)}$, we have
%\begin{align}
%    \Esp{k_n(X)} &= \int_0^\infty 1 - F_{k_n(X)}(x) dx \\
%    &\geq  \int_0^\infty 1 - G(x) dx \\
%    &= \int_0^\infty 1 - \left(1 - \frac{1}{2^{x}}\right)^n + \frac{n}{2^{x}}\left(1 - \frac{1}{2^{x}}\right)^{n-1} dx . 
%\end{align}
%Setting $u = 2^{-x}$ yields
%\begin{align}
%    \Esp{k_n(X)} &\leq \frac{1}{\log 2} \int_0^1 \frac{1}{u} - \frac{1}{u} \left(1 - u\right)^n - n ( 1-u)^{n-1} du \\
%    &= \frac{1}{\log 2} \int_0^1 \frac{1 - \left(1 - u\right)^n}{u} du - \frac{1}{\log 2} \\
%    &= \frac{s(n+1,2)}{n! \log 2}  - \frac{1}{\log 2}
%\end{align}
%where $s(.,.)$ denotes the Stirling number of the first kind.
Using the inequality $\log(1-x) \leq - x$ for all $x \in [0, 1)$ yields, 
\begin{align}
	 \Prob{k_n(X, \Theta) \geq k} & \geq  1 -  \exp \left(-\frac{n}{2^{k-1}} \right) - \frac{n}{2^{k-1}} \exp \left(-\frac{n-1}{2^{k-1}} \right)\\
	 & \geq  1 -  \left( 1 + \frac{n}{2^{k-1}} \right) \exp \left(-\frac{n}{2^{k-1}} \right). \label{eq_adaCRF_depth_lowbound}
\end{align}
Letting $k = (1-\varepsilon_n)\log_2(n)$ in \eqref{eq_adaCRF_depth_lowbound}
yields
\begin{align}
	 \Prob{k_n(X, \Theta) \geq k} 
	& \geq  1 -  \left( 1 + 2n^{\varepsilon_n}  \right) \exp \left(-2n^{\varepsilon_n}  \right).
\end{align}
Note that, setting $\varepsilon_n = c_1 (\log_2 n)^{- \alpha}$ for any $\alpha \in [0,1)$ implies that
\begin{align}
n^{\varepsilon_n} = \exp \left( \varepsilon_n \log n\right)
\end{align}
tends to infinity. Therefore, for all $c_1 > 0$ and all $\alpha _in [0,1)$,  
 \begin{align}
 	\lim\limits_{n \to \infty} \Prob{k_n(X, \Theta) \geq \log_2(n) - c_1 (\log_2 n)^{-\alpha}} = 1.
 \end{align}
Besides, 
\begin{align}
\Prob{k_n(X, \Theta) \leq k} 
&= 1 - \Prob{k_n(X, \Theta) > k}\\
& =  \left(1 - \frac{1}{2^{k}}\right)^n - \frac{n}{2^{k}}\left(1 -+ \frac{1}{2^{k}}\right)^{n-1}. 
\end{align}
Using the inequality $\log (1-  x) \geq - x / (1-x)$ for all $x \in [0,1)$, we have
\begin{align}
	\Prob{k_n(X, \Theta) \leq k} & \geq  \exp \left(- \frac{n}{2^{k} - 1} \right) + \frac{n}{2^{k} }\exp \left( - \frac{n-1}{2^{k} - 1}\right)\\
	& \geq \left( 1 + \frac{n}{2^{k} } \right) \exp \left(- \frac{n}{2^{k} - 1} \right). \label{eq_adaCRF_depth_uppbound}
\end{align}
Letting $k = (1+\varepsilon_n)\log_2(n)$ in \eqref{eq_adaCRF_depth_uppbound}
yields
\begin{align}
	\Prob{k_n(X, \Theta) \geq k} 
	& \geq   \left( 1 + 2n^{-\varepsilon_n}  \right) \exp \left(-\frac{n}{n^{1+ \varepsilon_n}- 1}  \right)\\
	& \geq \left( 1 + 2n^{-\varepsilon_n}  \right) \exp \left(-\frac{n^{-\varepsilon_n}}{1 - \frac{1}{n^{1 + \varepsilon_n}}}  \right),
\end{align}
which tends to $1$ for the choice $\varepsilon_n = c_2 (\log_2 n)^{- \alpha}$, for any $\alpha \in [0,1)$ and any $c_2 >0$.

%
%
%\begin{align}
%	\log (1 - x) \geq - x / (1-x)\\
%	\log(1 - \frac{1}{2^{k-1}}) \geq - \frac{\frac{1}{2^{k-1}}}{1 - \frac{1}{2^{k-1}}}\\
%	n \log(1 - \frac{1}{2^{k-1}}) \geq - \frac{n}{2^{k-1} - 1}
%	\label{eq_adaCRF_depth_uppbound}
%\end{align}
%
%Letting $k = (1-\log_2^{-\alpha} (n))\log_2(n)$ in \eqref{eq_adaCRF_depth_lowbound} and \eqref{eq_adaCRF_depth_uppbound} yields the result.   %$$\exp\left(-\frac{n}{2^{k-1}-1} \right) \leq \exp\left(n \log \left(1-\frac{1}{2^{k-1}}\right)\right) \leq \exp \left(-\frac{n}{2^{k-1}} \right)$$ yields the result.

%%%%%%%%%%%%%%%%%%%%%%%%%%%%%%%%%%%%%%%%%%%%%%%%%%%%%%%%%%%%%%%%%%%%%%%%%%%%%%%%%%%%%%%%%%%%%%%%%%

%%%%%%%%%%%%%%%%%%%%%%%%%%%%%%%%%%%%%%%%%%%%%%%%%%%%%%%%%%%%%%%

\subsubsection{Proof of Theorem \ref{th:Median_RF_consistency} (Consistency of Median RF)}

\paragraph{Preliminary results} 

In all the preliminary results, we use the fact that the spacing between two consecutive order statistics, that originate from an i.i.d.\ sample uniformly distributed on $[0,1]$ of size $n_j$ is distributed as a beta distribution $\mathcal{B}(1, n_j)$. We also recall that, for all $\alpha, \beta$, 
\begin{align}
    \mathds{V}\left[ \mathcal{B}(\alpha, \beta) \right] = \frac{\alpha \beta}{(\alpha + \beta)^2 (\alpha + \beta + 1)} \quad \textrm{and} \quad \E\left[ \mathcal{B}(\alpha, \beta) \right] = \frac{\alpha}{\alpha + \beta}.
\end{align}

\begin{lemma}[Control of a cell side of a fully-developed median RF]
\label{Lemme_cell_length_beta} 
Assume that $n \geq 16$ is a power of two. For all $\bx \in [0,1]^d$, for all $\ell \in \{1, \hdots, d\}$ and depth $k \in \N^*$, with $k \leq \lfloor \log_2 n \rfloor$, we have
\begin{align}
\E \left[ \mu \left( A_{k,n}^{(\ell)}(\bx, \T) \right)^2 \right] 
& \leq C_1 \left( 1 - \frac{3}{4d} \right)^k,
\end{align}
with $C_1 \leq 256 \exp \left( \frac{ 42 + \sqrt{5}}{2 - \sqrt{2}} \right)$.
\end{lemma}

\begin{proof}[Proof of Lemma \ref{Lemme_cell_length_beta}]
Fix $x \in [0,1]^d$. For all $\ell$, let $\bm{\delta}_{\ell}(x, \Theta)$ be the vector whose components are defined as $\delta_{j,\ell}(x, \Theta) = 1 $ if the $j$-th cut is made along direction $\ell$ and $0$ otherwise.
Without loss of generality, we let $\ell =1$ and fix $\bx \in [0,1]^d$. For all $j \in \{0, \hdots, k\}$, we denote $A_{j,n}^{(1)}(\bx, \T)$ the cell containing $\bx$ at level $j$, projected onto the first direction, and $n_j = n 2^{-j}$ the number of observations falling into this cell.

Recall that we consider the median forest in which splits are performed at the middle of two consecutive order statistics in a cell, so that each resulting cell contains exactly the same number of observations. 
With these notations in mind, we want to upper bound, for all $j$,
\begin{align*}
\E \bigg[ \mu \left(A_{j,n}^{(1)}(\bx, \T)\right)^2 | \bm{\delta}_{1}(x, \Theta) \bigg],
\end{align*}
where, for now, the split randomization $\bm{\delta}_1(x, \Theta)$ is considered fixed and may be omitted in the notations.
Let us fix $j \leq k-1$, define 
\begin{align*}
A_{j,n}^{(1)}(\bx, \T) = [M_{1,j}, M_{2,j}], %\times \hdot
\end{align*}
%with $j_1, j_2 < j$. 
and assume that the next cut is made along the first axis at position $M_j$. Then, 
\begin{align}
& \mu \left( A_{j+1,n}^{(1)}(\bx, \T) \right)^2 \nonumber \\
&  = (M_j - M_{1,j})^2 \ind{\bx \in [M_{1,j}, M_j]} + ( M_{2,j} - M_j)^2 \ind{\bx \in [M_{j}, M_{2,j}]}\\
&  = (M_j - M_{1,j})^2  + \left( ( M_{2,j} - M_j)^2 - (M_j - M_{1,j})^2 \right)  \ind{\bx \in [M_{j}, M_{2,j}]}\\
&  = (M_j - M_{1,j})^2  +   \left( M_{1,j} + M_{2,j}  - 2M_j \right) (M_{2,j} - M_{1,j})   \ind{\bx \in [M_{j}, M_{2,j}]}. \label{lemma_bias_eq3}
\end{align}

We denote $X_1', ..., X_{n_j}'$ the points contained in the cell $A_{j,n}^{(1)}(\bx, \T)$.
Note that the second term in \eqref{lemma_bias_eq3} can be decomposed as
\begin{align}
& \left( M_{1,j} + M_{2,j}  - 2M_j \right) (M_{2,j} - M_{1,j})   \ind{\bx \in [M_{j}, M_{2,j}]} \nonumber \\
= &\left( X'_{(1)} + X'_{(n_j)}    -  X'_{(n_j/2)} - X'_{(n_j/2 + 1)} + M_{1,j} - X'_{(1)} + M_{2,j} - X'_{(n_j)} \right) (M_{2,j} - M_{1,j})   \ind{\bx \in [M_{j}, M_{2,j}]}\\
 = &\left( \frac{X'_{(1)} + X'_{(n_j)}}{2}    -  X'_{(n_j/2)} + \frac{X'_{(1)} + X'_{(n_j)}}{2}  -  X'_{(n_j/2 + 1)} + M_{1,j} - X'_{(1)} + M_{2,j} - X'_{(n_j)} \right) \nonumber \\
 & \quad \times (M_{2,j} - M_{1,j})   \ind{\bx \in [M_{j}, M_{2,j}]}\\
  \leq & \left( \frac{X'_{(1)} + X'_{(n_j)}}{2}    -  X'_{(n_j/2)} + \frac{X'_{(1)} + X'_{(n_j)}}{2}  -  X'_{(n_j/2 + 1)} +  M_{2,j} - X'_{(n_j)} \right) (M_{2,j} - M_{1,j}). \label{eq:lemma_bias_decomp1}
\end{align}

Injecting \eqref{eq:lemma_bias_decomp1} into \eqref{lemma_bias_eq3}, taking the expectation and  using Cauchy-Schwarz inequality leads to
\begin{align}
\E \left[ \mu \left( A_{j+1,n}^{(1)}(\bx, \T) \right)^2 \right] &\leq  \E \left[ (M_j - M_{1,j})^2 \right]  \nonumber\\
& \quad + \left( \E\left[ \left( \frac{X'_{(1)} + X'_{(n_j)}}{2}    -  X'_{(n_j/2)} \right)^2 \right] \E \left[ (M_{2,j} - M_{1,j})^2 \right]\right)^{1/2} \nonumber \\
& \quad + \left( \E\left[ \left( \frac{X'_{(1)} + X'_{(n_j)}}{2}    -  X'_{(n_j/2+1)} \right)^2 \right] \E \left[ (M_{2,j} - M_{1,j})^2 \right]\right)^{1/2} \nonumber \\
& \quad + \left( \E\left[ \left( M_{2,j} - X'_{(n_j)} \right)^2 \right] \E \left[ (M_{2,j} - M_{1,j})^2 \right]\right)^{1/2}.
\end{align}

Considering the second term, we have
\begin{align}
  \E \left[ \left( \frac{X'_{(1)} + X'_{(n_j)}}{2}    -  X'_{(n_j/2)} \right)^2 \right] 
  & = \E \left[ \left( \frac{ X'_{(n_j)} - X'_{(1)}}{2}    -  (X'_{(n_j/2)} - X'_{(1)}) \right)^2 \right] \\
  & = \E \left[ (X'_{(n_j)} - X'_{(1)})^2 \E \left[ \left(\frac{1}{2} -  \frac{(X'_{(n_j/2)} - X'_{(1)})}{(X'_{(n_j)} - X'_{(1)})} \right)^2 | X'_{(1)}, X'_{(n_j)} \right] \right].
  \end{align}
where  
\begin{align*}
\frac{(X'_{(n_j/2)} - X'_{(1)})}{(X'_{(n_j)} - X'_{(1)})} | X'_{(1)}, X'_{(n_j)} \sim \mathcal{B} \left(\frac{n_j}{2}-1,  \frac{n_j}{2} \right),
\end{align*}
with $\E [\mathcal{B}(\frac{n_j}{2}-1,  \frac{n_j}{2})] = \frac{n_j - 2}{2(n_j - 1)}$. 

Thus, 
\begin{align}
\E \left[ \left(\frac{1}{2} -  \frac{(X'_{(n_j/2)} - X'_{(1)})}{(X'_{(n_j)} - X'_{(1)})} \right)^2 | X'_{(1)}, X'_{(n_j)} \right]  & = \left( \frac{1}{2} -  \frac{n_j - 2}{2(n_j - 1)}\right)^2 + \mathds{V}\left[ \mathcal{B}\left(\frac{n_j}{2}-1,  \frac{n_j}{2}\right) \right]\\
& = \frac{1}{4(n_j -1)^2}   + \frac{1}{4} \frac{n_j - 2   }{(n_j-1)^2}\\
& = \frac{1}{4(n_j -1)}.
\end{align}

Consequently, 
\begin{align}
  \E \left[ \left( \frac{X'_{(1)} + X'_{(n_j)}}{2}    -  X'_{(n_j/2)} \right)^2 \right] 
  & = \frac{1}{4(n_j -1)} \E \left[ (X'_{(n_j)} - X'_{(1)})^2 \right]\\
  & \leq \frac{1}{4(n_j -1)} \E \left[ (M_{2,j} - M_{1,j})^2 \right].
\end{align}
  
Similarly, 
\begin{align}
  \E \left[ \left( \frac{X'_{(1)} + X'_{(n_j)}}{2}    -  X'_{(n_j/2+1)} \right)^2 \right] 
  & = \frac{1}{4(n_j -1)} \E \left[ (X'_{(n_j)} - X'_{(1)})^2 \right]\\
  & \leq \frac{1}{4(n_j -1)} \E \left[ (M_{2,j} - M_{1,j})^2 \right].
\end{align}

%\begin{align}
 % \E \left[ \left( \frac{X'_{(1)} + X'_{(n_j)}}{2}    -  X'_{(n_j/2+1)} \right)^2 \right] 
%  & = \E \left[ (X'_{(n_j)} - X'_{(1)})^2 \E \left[ \left(\frac{1}{2} -  \frac{(X'_{(n_j/2+1)} - X'_{(1)})}{(X'_{(n_j)} - X'_{(1)})} \right)^2 | X'_{(1)}, X'_{(n_j)} \right] \right],
 % \end{align}
%where  
%\begin{align*}
%\frac{(X'_{(n_j/2+1)} - X'_{(1)})}{(X'_{(n_j)} - X'_{(1)})} | X'_{(1)}, X'_{(n_j)} \sim \mathcal{B} \left(\frac{n_j}{2},  \frac{n_j}{2}-1 \right),
%\end{align*}
%with $\E [\mathcal{B}(\frac{n_j}{2},  \frac{n_j}{2}-1)] = \frac{n_j}{2(n_j - 1)}$. Thus, 
%\begin{align}
%& \E \left[ \left(\frac{1}{2} -  \frac{(X'_{(n_j/2+1)} - X'_{(1)})}{(X'_{(n_j)} - X'_{(1)})} \right)^2 | X'_{(1)}, X'_{(n_j)} \right]  \\
%& = \left( \frac{1}{2} -  \frac{n_j}{2(n_j - 1)} \right)^2 + \mathds{V}\left[ \mathcal{B}(\frac{n_j}{2},  \frac{n_j}{2}-1) \right]\\
%& = \frac{1}{4(n_j -1)^2}   + \frac{1}{4} \frac{n_j - 2   }{(n_j-1)^2}\\
%& = \frac{1}{4(n_j -1)}.
%\end{align}
By Lemma ~\ref{lem:tech_lemma1}, 
\begin{align*}
    \E\left[\left(M_{2,j} - X'_{(n_j)} \right)^2\right]  \leq \frac{5}{(n_j-1)^2} \E\left[  \left( M_{2,j} - M_{1,j} \right)^2 \right].
\end{align*}
Gathering all previous inequalities into \eqref{lemma_bias_eq3} yields
\begin{align}
\E \left[ \mu \left( A_{j+1,n}^{(1)}(\bx, \T) \right)^2 \right] & \leq  \E \left[ (M_j - M_{1,j})^2 \right] + \frac{1}{\sqrt{n_j -1}}    \E \left[ (M_{2,j} - M_{1,j})^2 \right] \nonumber\\
 & \quad + \frac{\sqrt{5}}{(n_j-1)} \E\left[  \left( M_{2,j} - M_{1,j} \right)^2 \right].
\label{lemma_eq_bias_6}
\end{align}

Considering the first term in \eqref{lemma_eq_bias_6}, we have
\begin{align}
(M_j - M_{1,j})^2 & = \left( \frac{X'_{(n_j/2)} + X'_{(n_j/2 + 1)}}{2} - X'_{(1)} + X'_{(1)} - M_{1,j}\right)^2 \\
& \leq \left( X'_{(n_j/2 + 1)} - X'_{(1)} + X'_{(1)} - M_{1,j}\right)^2 \\
& \leq \left( X'_{(n_j/2 + 1)} - X'_{(1)} \right)^2 + \left( X'_{(1)} - M_{1,j}\right)^2  + 2 \left( X'_{(n_j/2 + 1)} - X'_{(1)} \right) \left( X'_{(1)} - M_{1,j}\right). \nonumber
\end{align}

Taking the expectation and using Cauchy-Schwarz inequality, we obtain
\begin{align}
\E \left[ (M_j - M_{1,j})^2  \right] & \leq \E \left[ \left( X'_{(n_j/2 + 1)} - X'_{(1)} \right)^2 \right] + \E \left[ \left( X'_{(1)} - M_{1,j}\right)^2 \right] \nonumber \\
& \quad + 2 \left( \E \left[ \left( X'_{(n_j/2 + 1)} - X'_{(1)} \right)^2 \right] \E \left[ \left( X'_{(1)} - M_{1,j}\right)^2 \right] \right)^{1/2}. \label{eq_lemma_bias_1}
\end{align}

Now, 
\begin{align}
\E \left[ \left( X'_{(n_j/2 + 1)} - X'_{(1)} \right)^2 \right] & = \E \left[ \left( X'_{(n_j)} - X'_{(1)} \right)^2 \E \left[ \left( \frac{X'_{(n_j/2 + 1)} - X'_{(1)}}{X'_{(n_j)} - X'_{(1)} } \right)^2 | X'_{(1)}, X'_{(n_j)}\right] \right], 
\end{align}
where
\begin{align}
\E \left[ \left( \frac{X'_{(n_j/2 + 1)} - X'_{(1)}}{X'_{(n_j)} - X'_{(1)} } \right)^2 | X'_{(1)}, X'_{(n_j)}\right] %& = \E \left[ \mathcal{B} \left(\frac{n_j}{2},  n_j - 2 - \frac{n_j}{2} +1 \right)^2 \right] \\ 
 & = \E \left[ \mathcal{B} \left(\frac{n_j}{2},   \frac{n_j}{2} - 1 \right)^2 \right] \\ 
 & = \mathds{V} \left[ \mathcal{B} \left(\frac{n_j}{2},   \frac{n_j}{2} - 1 \right)\right] + \left(\E \left[\mathcal{B} \left(\frac{n_j}{2},   \frac{n_j}{2} - 1 \right) \right] \right)^2 \\
 &= \frac{\frac{n_j}{2} \left( \frac{n_j}{2} - 1\right)}{(n_j - 1)^2 n_j} + \left(\frac{n_j/2}{n_j-1}\right)^2\\
 & = \frac{1}{4} \frac{n_j-2}{(n_j - 1)^2} + \left( \frac{1}{2} \frac{n_j}{n_j-1}\right)^2\\
 & = \frac{1}{4} \frac{n_j^2 +n_j - 2}{(n_j - 1)^2}\\
 & \leq \frac{1}{4} \frac{\left( n_j + 1/2\right)^2}{(n_j -1)^2}.
\end{align}

Therefore, 
\begin{align}
\E \left[ \left( X'_{(n_j/2 + 1)} - X'_{(1)} \right)^2 \right] & \leq \frac{1}{4} \frac{\left( n_j + 1/2\right)^2}{(n_j -1)^2} \E \left[ \left( X'_{(n_j)} - X'_{(1)} \right)^2  \right].
\end{align}
Injecting this expression into \eqref{eq_lemma_bias_1}, we have
\begin{align}
\E \left[ (M_j - M_{1,j})^2  \right] & \leq \frac{1}{4} \frac{\left( n_j + 1/2\right)^2}{(n_j -1)^2} \E \left[ \left( X'_{(n_j)} - X'_{(1)} \right)^2  \right] + \E \left[ \left( X'_{(1)} - M_{1,j}\right)^2 \right] \nonumber \\
& \quad + \frac{\left( n_j + 1/2\right)}{(n_j -1)} \left( \E \left[ \left( X'_{(n_j)} - X'_{(1)} \right)^2 \right] \E \left[ \left( X'_{(1)} - M_{1,j}\right)^2 \right] \right)^{1/2}. 
\end{align}
According to Technical Lemma~\ref{lem:tech_lemma1}, 
we have
\begin{align*}
    \E\left[\left(X'_{(1)} - M_{1,j}\right)^2\right]  \leq \frac{5}{(n_j-1)^2} \E\left[  M_{2,j} - M_{1,j}\right].
\end{align*}

Hence, 
\begin{align}
& \E \left[ (M_j - M_{1,j})^2  \right] \nonumber \\
& \leq \frac{1}{4} \frac{\left( n_j + 1/2\right)^2}{(n_j -1)^2} \E \left[ \left( M_{2,j} - M_{1,j} \right)^2  \right] + \frac{5}{(n_j-1)^2} \E\left[  \left( M_{2,j} - M_{1,j} \right)^2 \right] \nonumber \\
& \quad + \frac{\left( n_j + 1/2\right)}{(n_j -1)} \left( \E \left[ \left( M_{2,j} - M_{1,j}  \right)^2 \right] \frac{5}{(n_j-1)^2} \E\left[  \left( M_{2,j} - M_{1,j} \right)^2  \right] \right)^{1/2}\\
& \leq \left(  \frac{1}{4} \frac{\left( n_j + 1/2\right)^2}{(n_j -1)^2} + \frac{5}{(n_j-1)^2} + \frac{\left( n_j + 1/2\right) \sqrt{5}}{(n_j -1)^2} \right)\E\left[  \left( M_{2,j} - M_{1,j} \right)^2  \right] \\
& \leq \frac{1}{4} \frac{\left( n_j + 1/2\right)^2}{(n_j -1)^2} \left(  1 + \frac{20}{\left( n_j + 1/2\right)^2} + \frac{ 4 \sqrt{5}}{\left( n_j + 1/2\right)} \right)\E\left[  \left( M_{2,j} - M_{1,j} \right)^2  \right] \\
& \leq \frac{1}{4} \left( 1 + \frac{3}{2(n_j - 1)} \right)^2   \left(  1 + \frac{20}{\left( n_j + 1/2\right)^2} + \frac{ 4 \sqrt{5}}{\left( n_j + 1/2\right)} \right)\E\left[  \left( M_{2,j} - M_{1,j} \right)^2  \right] \\
& \leq \frac{1}{4} \left( 1 + \frac{9}{2(n_j - 1)} \right)   \left(  1 + \frac{20}{\left( n_j + 1/2\right)^2} + \frac{ 4 \sqrt{5}}{\left( n_j + 1/2\right)} \right)\E\left[  \left( M_{2,j} - M_{1,j} \right)^2  \right],
\end{align}
for all $n_j \geq 4$, since $(1+x)^2 \leq 1+3x$ if $x \leq 1$.

Consequently, 
\begin{align}
 & \E \left[ (M_j - M_{1,j})^2   \right] \nonumber \\  
 & \leq \frac{1}{4} \left( 1 + \frac{9}{2(n_j - 1)} \right)   \left(  1 +   \frac{ 30}{ n_j - 1  } \right)\E\left[  \left( M_{2,j} - M_{1,j} \right)^2  \right] \\
 & \leq \frac{1}{4} \left( 1 + \frac{69}{2(n_j - 1)} + \frac{90}{(n_j-1)^2} \right) \E\left[  \left( M_{2,j} - M_{1,j} \right)^2  \right] \\
  & \leq \frac{1}{4} \left( 1 + \frac{35 + 6}{n_j - 1}   \right) \E\left[  \left( M_{2,j} - M_{1,j} \right)^2  \right] \\
   & \leq \frac{1}{4} \left( 1 + \frac{41}{n_j - 1}   \right) \E\left[  \left( M_{2,j} - M_{1,j} \right)^2  \right],
\end{align}
for all $n_j \geq 16$. 
Recall that, until now, we have fixed $\bm{\delta}_1(x, \Theta)$ and omitted the explicit conditioning in the proof to lighten notations. Thus, plugging-in the previous  inequality into \eqref{lemma_eq_bias_6} yields, for all $n_j \geq 16$, 
\begin{align}
\E \left[ \mu \left( A_{j+1,n}^{(1)}(\bx, \T) \right)^2 | \bm{\delta}_1(x, \Theta) \right] & \leq  \frac{1}{4} \left( 1 + \frac{41}{n_j - 1}   \right) \E\left[  \left( M_{2,j} - M_{1,j} \right)^2 | \bm{\delta}_1(x, \Theta) \right] \nonumber \\
& \quad + \frac{1}{\sqrt{n_j -1}}    \E \left[ (M_{2,j} - M_{1,j})^2 \big| \boldsymbol{\delta}_1(x,\Theta) \right] + \frac{\sqrt{5}}{(n_j-1)} \E\left[  \left( M_{2,j} - M_{1,j} \right)^2 \big| \boldsymbol{\delta}_1(x,\Theta) \right] \\
& \leq \frac{1}{4} \left( 1 + \frac{42 + \sqrt{5}}{\sqrt{n_j - 1}}   \right) \E\left[  \mu \left( A_{j,n}^{(1)}(\bx, \T) \right)^2 \big| \boldsymbol{\delta}_1(x,\Theta)  \right].
\end{align}

Recall that $\bm{\delta}_1(x, \Theta)$ is the vector whose components are defined as $\delta_{j,1}(x, \Theta) = 1 $ if the $j$-th cut is made along the first direction  and $0$ otherwise. We let $K_1 = \|\bm{\delta}_1(x, \Theta) \|_1$ be the number of times the first direction is split. By induction, we have
\begin{align}
\E \left[ \mu \left( A_{k,n}^{(1)}(\bx, \T) \right)^2 \right] 
& = \E \left[ \E \left[ \mu \left( A_{k,n}^{(1)}(\bx, \T) \right)^2 | \bm{\delta}_{1}(x, \Theta) \right] \right] \\
& \leq \E \left[  \prod_{j :  \delta_{j,1}=1, \atop j \leq k-4} \frac{1}{4} \left( 1 + \frac{42 + \sqrt{5}}{\sqrt{n_j - 1}}   \right)  \right]\\
& \leq 4^4 \E \left[ 4^{-K_1} \prod_{j: \delta_{j,1}=1,\atop j \leq k-4}  \left( 1 + \frac{42 + \sqrt{5}}{\sqrt{n_j - 1}}   \right)  \right].
\end{align}

The product can be upper bounded as follows, with $C = 42 + \sqrt{5}$,
\begin{align}
\log \left( \prod_{j, \delta_{j,l}=1, j \leq k-4} \left( 1 + \frac{C}{\sqrt{n_j - 1}}   \right) \right) & \leq  \log \left( \prod_{j: \delta_{j,1}=1, \atop j \leq k-4} \left( 1 + \frac{C}{\sqrt{n_{j+1}}}   \right) \right)\\
& = \sum_{j: \delta_{j,1}=1, \atop j \leq k - 4} \log \left( 1 + \frac{ C\sqrt{2} \cdot 2^{j/2}}{n^{1/2}}  \right)\\
& \leq \frac{C \sqrt{2}}{n^{1/2}}  \sum_{j=0}^{k-4}  2^{j/2}\\
%& \leq \frac{C \sqrt{2}}{n^{1/2}} \frac{1 - 2^{(k-3)/2}}{1 - \sqrt{2}} 
& \leq \frac{C \sqrt{2}}{\sqrt{2}-1} \frac{2^{(k-3)/2}}{n^{1/2}}\\
& \leq \frac{C}{2\sqrt{2} - 2}.
\end{align}

Thus,
\begin{align}
\E \left[ \mu \left( A_{k,n}^{(1)}(\bx, \T) \right)^2 \right] 
& \leq 4^4 \exp\left( \frac{C}{2\sqrt{2} - 2} \right) \E \left[ 4^{-K_1}   \right].
\end{align}
Since $K_1 \sim \textrm{Bin} (k, 1/d)$, we have
\begin{align}
 \E \left[ 4^{-K_1}   \right] & = \left( 1 - \frac{1}{d} + \frac{1}{4d} \right)^k \\
 & = \left( 1 - \frac{3}{4d} \right)^k.
\end{align}

Finally, 
\begin{align}
\E \left[ \mu \left( A_{k,n}^{(1)}(\bx, \T) \right)^2 \right] 
& \leq 4^4 \exp\left( \frac{C}{2\sqrt{2} - 2} \right) \left( 1 - \frac{3}{4d} \right)^k,
\end{align}
with $C = 42 + \sqrt{5}$.

\end{proof}

\begin{lemma}[Technical Lemma]
\label{lem:tech_lemma1}

\smallskip 

% With the assumptions of Lemma~\ref{Lemme_cell_length_beta} and the notations of the proof of Lemma~\ref{Lemme_cell_length_beta}, we have, for all $j\leq k-1$,
\begin{enumerate}
\item Let $x\in[0,1]^d$ and consider the cell $A_{n,j}(x,\Theta)$ containing $x$ at depth $j\leq k-1$. W.l.o.g.\ restrict the study to the one-dimensional cell $A_{n,j}^{(1)}(x,\Theta)$ corresponding to the cell $A_{n,j}(x,\Theta)$ along the first dimension only, and set $A_{n,j}^{(1)}(x,\Theta) = [M_{1,j};M_{2,j}]$. The one-dimensional cell 
$A_{n,j}^{(1)}(x,\Theta)$ contains $n_j$ points denoted $X'_1, \hdots X'_{n_j}$ (random subsample of the initial training sample). Call $X'_{(1)},\hdots , X'_{(n_j)}$ the ordered version of $X'_1, \hdots X'_{n_j}$.
Then,
%\es{il me semble qu'on n'a pas besoin du premier point}
%\begin{enumerate}
%    \item \begin{align*}
%    \E\left[X'_{(1)} - M_{1,j}\right] + \E\left[M_{2,j} - X'_{(n_j)}  \right]   
%    & \leq \frac{2}{n_{j}+1} \E\left[  M_{2,j} - M_{1,j}\right],
%\end{align*}
\begin{align*}
    \E\left[\left(X'_{(1)} - M_{1,j}\right)^2\right]  \leq \frac{5}{(n_j-1)^2} \E\left[  \left( M_{2,j} - M_{1,j} \right)^2 \right],
\end{align*}
and
\begin{align*}
    \E\left[\left(M_{2,j} - X'_{(n_j)}\right)^2\right]  \leq \frac{5}{(n_j-1)^2} \E\left[  \left( M_{2,j} - M_{1,j} \right)^2 \right].
\end{align*}

\item Consider now the cell $A_{n,j}(X_1,\Theta)$ containing $X_1$ at depth $j\leq k-1$. W.l.o.g.\ restrict the study to the one-dimensional cell $A_{n,j}^{(1)}(X_1,\Theta)$ corresponding to the cell $A_{n,j}(X_1,\Theta)$ along the first dimension only, and set $A_{n,j}^{(1)}(X_1,\Theta) = [M_{1,j};M_{2,j}] = [M_{1,j}(X_1,\Theta);M_{2,j}(X_1,\Theta)]$. The one-dimensional cell 
$A_{n,j}^{(1)}(X_1,\Theta)$ contains $n_j$ points denoted $\{X'_1, \hdots X'_{n_j-1}\} \cup \{X_1\}$ (random subsample of the initial training sample containing $X_1$ and projected on the first axis). Call $X'_{(1)},\hdots , X'_{(n_j-1)}$ the ordered version of $X'_1, \hdots X'_{n_j-1}$.
Then,
    \begin{align*}
    \E\left[X'_{(1)} - M_{1,j}|X_1\right] 
    &\leq  \frac{1}{n_j}\E\left[  (M_{2,j}-M_{1,j}) | X_1 \right],
\end{align*}
and
\begin{align*}
    \E\left[M_{2,j} - X'_{(n_j-1)}|X_1\right] 
    &\leq  \frac{1}{n_j}\E\left[  (M_{2,j}-M_{1,j}) | X_1 \right].
\end{align*}
\end{enumerate}
\end{lemma}

\begin{proof}[Proof of Lemma~\ref{lem:tech_lemma1}]\hfill 

\textbf{Notations} 
{
W.l.o.g.\ consider the following development according to the first direction only. Let $x\in [0,1]^d$. Recall that we consider the cell $A_{j,n}^{(1)}(x,\Theta) = [M_{1,j} ; M_{2,j}]$ containing $x$ at depth $j\leq k-1$.
The cut at $M_{1,j}$ (resp.\ $M_{2,j}$) has been obtained at an anterior depth $j_1\leq j$ (resp.\ $j_2\leq j$), as the middle of two order statistics of a previous subsample:
\begin{align*}
    M_{1,j} = \frac{M_{1,j,-} + M_{1,j,+}}{2} \quad \text{and} \quad
    M_{2,j} = \frac{M_{2,j,-} + M_{2,j,+}}{2},
\end{align*}
with $M_{1,j,-} < M_{1,j,+}$ and $M_{2,j,-} < M_{2,j,+}$.  The following computations can be also conducted in a similar way when $M_{1,j}=0$ or $M_{2,j}=1$.
The current cell $A^{(1)}_{j,n}(x,\Theta)$, includes now $n_j$ points of the original training sample, which are denoted by $X'_1, \hdots , X'_{n_j}$. Remark that as 
$M_{1,j,-}$ and $M_{2,j,+}$ refer to anterior order statistics of a previous subsample (including the points $X'_1, \hdots , X'_{n_j}$), then $X'_1, \hdots , X'_{n_j}$ are i.i.d.\ uniformly distributed in $[M_{j,1,-}; M_{2,j,+}]$. Denote by $X'_{(1)}, \hdots , X'_{(n_j)}$, the ordered statistics of the current subsample $X'_1, \hdots , X'_{n_j}$ in $A^{(1)}_{j,n}(x,\Theta)$ for some fixed $x$.

% Control of $\E\left[X'_{(1)} - M_{1,j}\right]$.
%Remark that
%\begin{align}
 %   \E\left[X'_{(1)} - M_{1,j}\right] + \E\left[X'_{(n_j)} - X'_{(1)}\right] + 
 %   \E\left[M_{2,j} - X'_{(n_j)}  \right] &= \E\left[M_{2,j} - M_{1,j}\right].
%\end{align}
%Therefore,
%\begin{align}
 %   \E\left[X'_{(1)} - M_{1,j}\right] + 
  %  \E\left[M_{2,j} - X'_{(n_j)}  \right] 
   % &=  \E\left[M_{2,j} - M_{1,j}\right] - \sum_{\ell=2}^{n_j} \E\left[X'_{(\ell)} - X'_{(\ell-1)}\right]
  %  \\
   % &= \E\left[M_{2,j} - M_{1,j}\right] - (n_j-1) \E\left[X'_{(2)} - X'_{(1)}\right]
   % .
%\end{align}
%We focus now on the term $\E\left[X'_{(2)} - X'_{(1)}\right]$. One has
%\begin{align}
 %   \E\left[X'_{(2)} - X'_{(1)}\right] &= \E\left[ \E \left[\frac{X'_{(2)} - X'_{(1)}}{M_{2,j,+}-M_{1,j,-}} | M_{1,j,-} , M_{2,j,+} \right] (M_{2,j,+}-M_{1,j,-})\right]. 
%\end{align}
%As, conditional on $M_{1,j,-} , M_{2,j,+}$, $\frac{X'_{(2)} - X'_{(1)}}{M_{2,j,+}-M_{1,j,-}}$ follows a beta distribution $\mathcal{B}(1,n_j)$, then
%\begin{align}
 %   \E\left[X'_{(2)} - X'_{(1)}\right] &=
%\E\left[ \frac{1}{1+n_j} (M_{2,j,+}-M_{1,j,-})\right]\\
%&\geq \frac{1}{1+n_j} \E\left[ M_{2,j}-M_{1,j}\right],
%\end{align}
%which gives
%\begin{align}
%    \E\left[X'_{(1)} - M_{1,j}\right] + \E\left[M_{2,j} - X'_{(n_j)}  \right]  
%    &\leq \left( 1-\frac{n_j-1}{n_j+1} \right)\E\left[M_{2,j} - M_{1,j}\right] \\
%    &\leq \frac{2}{n_j+1} \E\left[M_{2,j} - M_{1,j}\right].
%\end{align}
\paragraph{First statement - Control of $\E[(X'_{(1)} - M_{1,j})^2]$.}

We have
\begin{align}
\E\left[\left(X'_{(1)} - M_{1,j}\right)^2\right] \leq 2 \E\left[ \left(M_{1,j,+} - M_{1,j}\right)^2\right] + 2 \E\left[\left(X'_{(1)} - M_{1,j,+}\right)^2\right]. \label{tech_lemma_eq2} 
\end{align}
Note that, by definition of $ M_{1,j}$, the quantity $M_{1,j,+} - M_{1,j}$ corresponds to a half spacing between two points in the cell previously built by cutting on the first direction at depth $j_1$, denoted $A_{j_1,n}^{(1)}(x, \Theta)$.
By construction, the spacings between two consecutive points in $A_{j_1,n}^{(1)}(x, \Theta)$ were the same in distribution. Since points have been removed between $A_{j,n}^{(1)}(x, \Theta)$ and $A_{j_1,n}^{(1)}(x, \Theta)$, the spacings are larger between consecutive points in $A_{j,n}^{(1)}(x, \Theta)$ than between consecutive points in $A_{j_1,n}^{(1)}(x, \Theta)$. This leads to 
\begin{align*}
M_{1,j,+} - M_{1,j} 
 =   \frac{M_{1,j,+} - M_{1,j,-}}{2} \leq  \frac{X'_{(2)} - X'_{(1)}}{2}.
\end{align*}
Therefore, since all variables are bounded, 
\begin{align*}
\E\left[ \left( M_{1,j,+} - M_{1,j} \right)^2 \right] & \leq  \frac{1}{4}  \E \left[ \left(X'_{(2)} - X'_{(1)} \right)^2 \right]  \\
& \leq  \frac{1}{4} \E\left[ (X'_{(n_j)} - X'_{(1)})^2 \E \left[ \frac{\left( X'_{(2)} - X'_{(1)} \right)^2}{\left( X'_{(n_j)} - X'_{(1)} \right)^2} | X'_{(1)}, X'_{(n_j)}\right] \right].
\end{align*}

Regarding the inner expectation, 
\begin{align}
\E \left[ \frac{\left( X'_{(2)} - X'_{(1)} \right)^2}{\left( X'_{(n_j)} - X'_{(1)} \right)^2} | X'_{(1)}, X'_{(n_j)}\right] 
 & = \E \left[ \mathcal{B} \left(1,   n_j-2 \right)^2 \right] \\ 
 & = \mathds{V} \left[ \mathcal{B} \left(1,   n_j-2 \right)\right] + \left(\E \left[\mathcal{B} \left(1,   n_j-2 \right) \right] \right)^2\\
&= \frac{n_j-2}{(n_j - 1)^2 n_j} + \left(\frac{1}{n_j-1} \right)^2\\
%& = \frac{2n_j - 2}{(n_j - 1)^2 n_j}\\
& \leq \frac{2}{(n_j-1)^2}.
\end{align}

Finally, 
\begin{align}
\E\left[ \left( M_{1,j,+} - M_{1,j} \right)^2 \right] & \leq  \frac{1}{2(n_j-1)^2}  \E\left[ \left(X'_{(n_j)} - X'_{(1)} \right)^2  \right]\\
& \leq \frac{1}{2(n_j-1)^2}  \E\left[ \left(M_{2,j} - M_{1,j} \right)^2  \right].
\end{align}

Regarding the second term in \eqref{tech_lemma_eq2}, we have
\begin{align*}
\E\left[ \left( X'_{(1)} - M_{1,j,+} \right)^2 \right] & =  \E \left[ \E\left[ \left(X'_{(1)} - M_{1,j,+}\right)^2 | M_{1,j,+} , X_{M_{2,j}, -}\right]\right]   \\
& = \E \left[ \left(M_{2,j, -} - M_{1,j,+}\right)^2 \E\left[ \left(\frac{X'_{(1)} - M_{1,j,+}}{M_{2,j, -}  - M_{1,j,+}} \right)^2 | M_{1,j,+} , M_{2,j, -}\right]\right]\\
& \leq \E \left[ \left(M_{2,j, -}  - M_{1,j,+}\right)^2 \E\left[  \mathcal{B} (1, n_j-1)^2 \right]\right]\\
& \leq \frac{2}{n_j^2} \E \left[ (M_{2,j, -}  - M_{1,j,+})^2 \right]\\
& \leq \frac{2}{n_j^2} \E \left[ \left(M_{2,j}  -M_{1,j}\right)^2 \right].
\end{align*}

Finally,
\begin{align*}
    \E\left[\left(X'_{(1)} - M_{1,j}\right)^2\right] & \leq \left( \frac{1}{(n_j-1)^2}  +\frac{4}{n_j^2}  \right) \E\left[  \left( M_{2,j} - M_{1,j} \right)^2\right]\\
    & \leq \frac{5}{(n_j-1)^2} \E\left[  \left( M_{2,j} - M_{1,j} \right)^2 \right].
\end{align*}
The second point of the first statement can be proved in the exact same manner. 

\paragraph{Second statement - Control of $\E[X'_{(1)} - M_{1,j} | X_1]$.}  

In this part, we study the cell $A_{n,j}^{(1)}(X_1,\Theta)$. 
The cell $A_{n,j}^{(1)}(X_1,\Theta)$ contains $n_j$ data points (including $X_1$). We denote by $X_1', \hdots, X_{n_j-1}'$ the observations falling into $A_{n,j}^{(1)}(X_1,\Theta)$, different from $X_1$. Note that, these $n_j-1$ observations are still i.i.d.\ uniformly distributed in $[M_{1,j,-}; M_{2,j,+}]$. We denote by $X'_{(1)}, \hdots , X'_{(n_j-1)}$, the  subsample $X'_1, \hdots , X'_{n_j-1}$. We have
\begin{align}
    M_{2,j} - M_{1,j} = M_{2,j} - X_{(n_j-1)}' + \sum_{q=1}^{n_j-2} \left( X_{(n_j-q)}' - X_{(n_j-q - 1)}' \right) + X_{(1)}' - M_{1,j}.
\end{align}
Thus, 
    \begin{align}
    \E\left[X'_{(1)} - M_{1,j}|X_1\right] + \E \left[ M_{2,j} - X_{(n_j-1)}' | X_1 \right]
    &=  \E\left[M_{2,j} - M_{1,j} |X_1\right] - (n_j-2) \E\left[X'_{(2)} - X'_{(1)} |X_1 \right].
\end{align}
The variables $X'_{(1)}$ and $X'_{(2)}$ being order statistics of a subsample independent of $X_1$, one gets
\begin{align}
    \E\left[X'_{(2)} - X'_{(1)} |X_1 \right] &= 
    \E\left[ \E \left[\frac{X'_{(2)} - X'_{(1)}}{M_{2,j,+}-M_{1,j,-}}  | X_1, M_{1,j,-} , M_{2,j,+} \right] (M_{2,j,+}-M_{1,j,-}) | X_1 \right] \\
    &= \E\left[ \E \left[\mathcal{B}(1,n_j-1)\right] (M_{2,j,+}-M_{1,j,-}) | X_1 \right] \\
    &= \frac{1}{n_j} \E\left[  (M_{2,j,+}-M_{1,j,-}) | X_1 \right] \\
    &\geq \frac{1}{n_j} \E\left[  (M_{2,j}-M_{1,j}) | X_1 \right].
\end{align}
Finally,
\begin{align}
    \E\left[X'_{(1)} - M_{1,j}|X_1\right] + \E \left[ M_{2,j} - X_{(n_j-1)}' | X_1 \right]
    &\leq \left( 1 - \frac{n_j-2}{n_j}\right)\E\left[  (M_{2,j}-M_{1,j}) | X_1 \right], \\
    &= \frac{2}{n_j}\E\left[  (M_{2,j}-M_{1,j}) | X_1 \right],
\end{align}
and, by symmetry, 
\begin{align}
    \E\left[X'_{(1)} - M_{1,j}|X_1\right] 
    &\leq \frac{1}{n_j}\E\left[  (M_{2,j}-M_{1,j}) | X_1 \right],
\end{align}
and
\begin{align}
    \E \left[ M_{2,j} - X_{(n_j-1)}' | X_1 \right]
    &\leq \frac{1}{n_j}\E\left[  (M_{2,j}-M_{1,j}) | X_1 \right].
\end{align}
}

\end{proof}

%\begin{lemma}
%Let $X_1,...,X_n$ be i.i.d. with distribution $\mathcal{U}([0,1])$. Then, denoting $X'_{(\frac{n}{2})}$ the median of the sample, we have, for all $\varepsilon >0$, for all $\lambda >0$,
%\begin{align*}
%    \Prob{X'_{(\frac{n}{2})} - \Esp{X'_{(\frac{n}{2})}} > \varepsilon } \leq e^{[...] -\lambda \varepsilon}.
%\end{align*}
%\end{lemma}

\begin{lemma}[Control of the leaf side and volume of a fully developed median RF] \label{lem:leaf_median_tree}

Assume that $n \geq 4$ is a power of two. Consider a median tree of depth $k$ and denote $A_{n,k}(X_1,\Theta)$ the leaf containing $X_1$. For all $\ell \in \{1, \hdots, d\},$ we denote $K_{\ell}$ the number of splits along the $\ell$-th direction. Let also $\bm{\delta}_{\ell}(X_1, \Theta)$ be the vector whose components are defined as $\delta_{j,\ell}(X_1, \Theta) = 1 $ if the $j$-th cut of the cell $A_n^{(\ell)}(X_1, \Theta)$ is made along direction $\ell$ and $0$ otherwise. Then,
\begin{align}
    \Esp{\mu(A_{n,k}^{(\ell)}(X_1, \Theta)) | X_1, \bm{\delta}_{\ell}(X_1, \Theta)} &\leq 2^{-K_\ell+2} \prod_{j : \delta_{j,\ell}=1, \atop j\leq k-2} \left( 1 + \frac{2}{\sqrt{n_j-1}}\right).
\end{align}
In particular, letting $C_2 = 4 \exp (5/(\sqrt{2}-1))$, we have 
\begin{align}
 \Esp{\mu(A_{n,k}^{(\ell)}(X_1, \Theta)) | X_1, \bm{\delta}_{\ell}(X_1, \Theta)} &\leq C_2 ~ 2^{-K_{\ell}},   
\end{align}
and
\begin{align}
    \label{eq:volume_cell_medianRF}\Esp{\mu(A_{n,k}(X_1, \Theta)) | X_1, \bm{\delta}_1(X_1, \Theta), \hdots, \bm{\delta}_d(X_1, \Theta)} &\leq C_2 2^{-k}.
\end{align}

\end{lemma}

\begin{proof}

We write $A_{n,j}^{(\ell)}(X_1, \Theta) = [M_{1,j}, M_{2,j}]$ the cell of the RF containing $X_1$ along the direction $\ell$, at depth $j$. To lighten the notations, we omit the dependencies in $X_1$, in $\Theta$ and in $\ell$. We also write $X'_1, ..., X'_{n_j}$ the data points contained in the cell $A_{n,j}^{(\ell)}(X_1, \Theta)$, and we denote by $X'_{(1)}, ..., X'_{(n_j-1)}$ the ordered version of $\{X'_1, ..., X'_{n_j}\}\setminus \{X_1\}$.
We suppose that the next cut is occurring on the $\ell$-th direction and  compute the size of the new cell containing $X_1$, $A_{n,j+1}^{(\ell)}(X_1, \Theta)$, so that 4 different events are possible:
\begin{enumerate}
    \item $X_1$ is in the first "part" of the cell, i.e.\ $X_1 \in \left[M_{1,j}, X'_{(n_j/2 -1)}\right]$;
    \item $X_1$ is in the second "part" of the cell, i.e.\ ,  $X_1 \in \left[X'_{(n_j/2+1)}, M_{2,j}\right]$;
    \item $X_1$ is in the "middle (left)" of the cell, i.e.\ $X_1 \in \left[X'_{(n_j/2 -1)}, X'_{(n_j/2)} \right]$;
    \item $X_1$ is in the "middle (right)" of the cell, $X_1 \in \left[X'_{(n_j/2)}, X'_{(n_j/2 +1)} \right]$.
\end{enumerate}

The length of the following cell can be therefore decomposed with respect to the previous events:
\begin{align}
    & \nonumber
    \mu\left(A_{n,j+1}^{(\ell)}(X_1, \Theta)\right) \\
    = & \nonumber
    \left(\frac{X'_{(n_j/2 -1)} + X'_{(n_j/2)}}{2} - M_{1,j} \right) \ind{X_1 \in \left[M_{1,j}, X'_{(n_j/2 -1)}\right]} \\
    &+ \left(M_{2,j} - \frac{X'_{(n_j/2)} + X'_{(n_j/2+1)}}{2} \right) \ind{X_1 \in \left[X'_{(n_j/2+1)}, M_{2,j}\right]} \nonumber \\
    & + \left(\frac{X_{1} + X'_{(n_j/2)}}{2} - M_{1,j} \right) \ind{X_1 \in \left[X'_{(n_j/2 -1)}, X'_{(n_j/2)} \right]} \nonumber\\
    &+ \left(M_{2,j} - \frac{X'_{(n_j/2)} + X_{1}}{2} \right) \ind{X_1 \in \left[X'_{(n_j/2)}, X'_{(n_j/2 +1)} \right]} \\
    \leq &  \left(X'_{(n_j/2)} - M_{1,j}\right) \ind{X_1 \in \left[M_{1,j}, X'_{(n_j/2 -1)}\right]} + \left(M_{2,j} - X'_{(n_j/2)}\right) \ind{X_1 \in \left[X'_{(n_j/2+1)}, M_{2,j}\right]}  \\
    &  + \left(X'_{(n_j/2)} - M_{1,j}\right) \ind{X_1 \in \left[X'_{(n_j/2 -1)}, X'_{(n_j/2)} \right]} + \left(M_{2,j} - X'_{(n_j/2)}\right) \ind{X_1 \in \left[X'_{(n_j/2)}, X'_{(n_j/2 +1)} \right]} \\
    \end{align}
    \begin{align}
    %&\leq \left(\frac{X'_{(n_j/2 -1)} + X'_{(n_j/2)}}{2} - M_{1,j} \right) \ind{X_1 \in [M_{1,j}, X'_{(n_j/2)}]} + \left(M_{2,j} - \frac{X'_{(n_j/2)} + X'_{(n_j/2+1)}}{2} \right) \ind{X_1 \in [X'_{(n_j/2)}, M_{2,j}]}  \nonumber \\
    %&+ \left(\frac{X_{1} + X'_{(n_j/2)}}{2} - M_{1,j} \right) \ind{E_2} + \left(M_{2,j} - \frac{X'_{(n_j/2)} + X_{1}}{2} \right) \ind{E_2'} \\
    & \nonumber
    \mu\left(A_{n,j+1}^{(\ell)}(X_1, \Theta)\right)\\
    \leq &  \left(X'_{(n_j/2)} - M_{1,j}\right) \ind{X_1 \in [M_{1,j}, X'_{(n_j/2)}]}  + \left(M_{2,j} - X'_{(n_j/2)}\right) \ind{X_1 \in [X'_{(n_j/2)}, M_{2,j}]} \\
    = &  \left(X'_{(n_j/2)} - M_{1,j}\right) + 2 \left(\frac{M_{1,j} + M_{2,j}}{2} - X'_{(n_j/2)} \right) \ind{X_1 \in [X'_{(n_j/2)}, M_{2,j}]} \\
    = & \left(X'_{(n_j/2)} - X'_{(1)}\right) + \left( X'_{(1)} - M_{1,j}\right) + 2 \left(\frac{X'_{(1)} + X'_{(n_j-1)}}{2} - X'_{(n_j/2)} \right) \ind{X_1 \in [X'_{(n_j/2)}, M_{2,j}]} \nonumber \\
    & -  \left( \left(X'_{(1)} + X'_{(n_j-1)}\right)) - (M_{1,j} + M_{2,j}) \right) \ind{X_1 \in [X'_{(n_j/2)}, M_{2,j}]}\\
    \leq &  \left(X'_{(n_j/2)} - X'_{(1)}\right) + \left( X'_{(1)} - M_{1,j}\right) + 2 \left(\frac{X'_{(1)} + X'_{(n_j-1)}}{2} - X'_{(n_j/2)} \right) \ind{X_1 \in [X'_{(n_j/2)}, M_{2,j}]} \nonumber\\
    & + (M_{2,j} - X'_{(n_j-1)}).
\end{align}

Therefore,
\begin{align}
    &\Esp{ \mu\left(A_{n,j+1}^{(\ell)}(X_1, \Theta)\right) \bigg| X_1, \bm{\delta}_{\ell}(X_1, \Theta) } \nonumber \\
    &\leq \Esp{ (X'_{(n_j-1)} - X'_{(1)}) \Esp{\frac{X'_{(n_j/2)} - X'_{(1)}}{X'_{(n_j-1)} - X'_{(1)}} | X_1, \bm{\delta}_{\ell}(X_1, \Theta), X'_{(1)}, X'_{(n_j-1)}} \big| X_1, \bm{\delta}_{\ell}(X_1, \Theta) } \nonumber \\
    & \quad + 2 \Esp{\Esp{ \left| \frac{X'_{(n_j-1)} + X'_{(1)}}{2} - X'_{(n_j/2)} \right| \bigg| X_1, \bm{\delta}_{\ell}(X_1, \Theta), X_{(1)}', X_{(n_j-1)}' }\bigg| X_1, \bm{\delta}_{\ell}(X_1, \Theta) } \nonumber \\
    & \quad + \Esp{ X'_{(1)} - M_{1,j} \bigg| X_1, \bm{\delta}_{\ell}(X_1, \Theta) } + \Esp{  M_{2,j} - X'_{(n_j-1)} \bigg| X_1, \bm{\delta}_{\ell}(X_1, \Theta) }. \label{eq:volume_leaf_repartition}
\end{align}
Regarding the first term of \eqref{eq:volume_leaf_repartition}, remark that by property of the uniform distribution, conditional on $X'_{(1)}$ and $X'_{(n_j-1)}$, the order statistics  between $1$ and $n_j-1$ follow Beta distributions independently from $X_1$ and $ \bm{\delta}_{\ell}(X_1, \Theta)$. Therefore,
\begin{align*}
    \frac{X'_{(n_j/2)} - X'_{(1)}}{X'_{(n_j-1)} - X'_{(1)}} \big| X'_{(1)}, X'_{(n_j-1)} \sim \mathcal{B}(n_j/2-1, n_j/2-1),
\end{align*}
so that
\begin{align*}
    \Esp{\frac{X'_{(n_j/2)} - X'_{(1)}}{X'_{(n_j-1)} - X'_{(1)}} \bigg| X_1, \bm{\delta}_{\ell}(X_1, \Theta), X'_{(1)}, X'_{(n_j-1)}} &= \frac{n_j/2-1}{2(n_j/2-1)} = \frac{1}{2}.
\end{align*}

Overall the first term 
of \eqref{eq:volume_leaf_repartition} verifies
\begin{align}
    &\Esp{(X'_{(n_j-1)} - X'_{(1)}) \Esp{\frac{X'_{(n_j/2)} - X'_{(1)}}{X'_{(n_j-1)} - X'_{(1)}} \big| X_1, \bm{\delta}_\ell, X'_{(1)}, X'_{(n_j-1)}} \bigg| X_1, \bm{\delta}_{\ell}(X_1, \Theta)} \nonumber \\
    &= \Esp{ \frac{X'_{(n_j -1)} - X'_{(1)}}{2} \big| X_1, \bm{\delta}_{\ell}(X_1, \Theta)} \\
    &\leq \Esp{\frac{M_{2,j} - M_{1,j}}{2} \big| X_1, \bm{\delta}_{\ell}(X_1, \Theta)} \label{eq:first_term}.
\end{align}

Regarding the second term of \eqref{eq:volume_leaf_repartition}, we have
\begin{align}
    &  \Esp{ \left| \frac{X'_{(n_j-1)} + X'_{(1)}}{2} - X'_{(n_j/2)} \right| \bigg| X_1, \bm{\delta}_{\ell}, X'_{(1)}, X'_{(n_j-1)}}  \nonumber \\
    &= \Esp{ \left| \frac{X'_{(n_j-1)} - X'_{(1)}}{2} - \left(X'_{(n_j/2)} - X'_{(1)}\right) \right| \bigg| X_1, \bm{\delta}_{\ell}, X'_{(1)}, X'_{(n_j-1)}} \\
    &= \left(X'_{(n_j-1)} - X'_{(1)} \right) \Esp{  \left| \frac{1}{2} - \frac{X'_{(n_j/2)} - X'_{(1)}}{X'_{(n_j-1)} - X'_{(1)}} \right| \bigg| X_1, \bm{\delta}_{\ell}, X'_{(1)}, X'_{(n_j-1)}} \\
    &= (X'_{(n_j-1)} - X'_{(1)}) \Esp{ \left| \frac{1}{2}- \mathcal{B}(n_j/2-1,n_j/2-1)  \right|} \\
    &\leq (X'_{(n_j-1)} - X'_{(1)}) \sqrt{\Esp{ \left|\mathcal{B}(n_j/2-1,n_j/2-1) - \frac{1}{2} \right|^2}} \\
    &\leq \frac{M_{2,j} - M_{1,j}}{2 \sqrt{n_j-1}},
\end{align}
where the last inequality is simply obtained by computing the variance of a Beta distribution.

Therefore,
\begin{align}
    &2 \Esp{\Esp{ \left| \frac{X'_{(n_j-1)} + X'_{(1)}}{2} - X'_{(n_j/2)} \right| \bigg| X_1, \bm{\delta}_{\ell}(X_1, \Theta), X_{(1)}', X_{(n_j-1)}' }\bigg| X_1, \bm{\delta}_{\ell}(X_1, \Theta) } \nonumber \\
    &\leq \frac{1}{2 \sqrt{n_j-1}} \Esp{ M_{2,j} - M_{1,j}  \bigg| X_1,  \bm{\delta}_{\ell}(X_1, \Theta) } \label{eq:second_term}.
\end{align}

The third and fourth terms of \eqref{eq:volume_leaf_repartition} have the same expression, controlled by Lemma \ref{lem:tech_lemma1}:
\begin{align}
     \Esp{ X'_{(1)} - M_{1,j} \bigg| X_1, \bm{\delta}_{\ell}(X_1, \Theta)} &= \Esp{  M_{2,j} - X'_{(n_j-1)} \bigg| X_1, \bm{\delta}_{\ell}(X_1, \Theta)} \\
     &\leq \frac{1}{n_j} \Esp{M_{2,j} - M_{1,j} \big| X_1, \bm{\delta}_{\ell}(X_1, \Theta) } \label{eq:third_term}.
\end{align}
Finally, gathering \eqref{eq:first_term}, \eqref{eq:second_term} and \eqref{eq:third_term} yields
\begin{align}
    \Esp{ \mu\left(A_{n,j+1}^{(\ell)}(X_1, \Theta)\right) \bigg| X_1, \bm{\delta}_{\ell}(X_1, \Theta)}
    &\leq \Esp{M_{2,j} - M_{1,j} \big| X_1, \bm{\delta}_{\ell}(X_1, \Theta) } \left( \frac{1}{2} + \frac{1}{2\sqrt{n_j-1}} + \frac{2}{n_j}\right) \\
    &\leq \frac{1}{2} \left( 1 + \frac{5}{\sqrt{n_j-1}}\right) \Esp{M_{2,j} - M_{1,j} \big| X_1, \bm{\delta}_{\ell}(X_1, \Theta)} \\
    &= \frac{1}{2} \left( 1 + \frac{5}{\sqrt{n_j-1}}\right) \Esp{ \mu\left(A_{n,j}^{(\ell)}(X_1, \Theta)\right) \big| X_1, \bm{\delta}_{\ell}(X_1, \Theta)}
\end{align}
for all $n_j \geq 4$. An iterative product yields
\begin{align}
    \Esp{\mu\left(A_{n,k}^{(\ell)}(X_1, \Theta)\right) \big| X_1, \bm{\delta}_{\ell}(X_1, \Theta)} &\leq \Esp{\prod_{j : \delta_{j,\ell}=1, \atop j\leq k-2} \frac{1}{2} \left( 1 + \frac{5}{\sqrt{n_j-1}}\right) \big| X_1, \bm{\delta}_{\ell}(X_1, \Theta)} \\
    & = \prod_{j : \delta_{j,\ell}=1, \atop j\leq k-2} \frac{1}{2} \left( 1 + \frac{5}{\sqrt{n_j-1}}\right) \\
    &= 2^{-K_\ell+2} \prod_{j : \delta_{j,\ell}=1, \atop j\leq k-2} \left( 1 + \frac{5}{\sqrt{n_j-1}}\right),
    %\\
    %&\leq 2^{-K_\ell+2} \prod_{j \leq k-2} \left( 1 + \frac{8}{\sqrt{n_j-1}}\right),
\end{align}
which proves the first statement. Recalling that $n_j=n2^{-j}$,
\begin{align}
    \sum_{j=0}^k \log\left( 1 + \frac{5}{\sqrt{n}}2^{j/2}\right)
    &\leq \frac{5}{\sqrt{n}} \frac{2^{(k+1)/2}-1}{\sqrt{2}-1} \\
    &\leq \frac{5}{\sqrt{2}-1} \frac{2^{(\log_2 n) /2}}{\sqrt{n}} \\
    &= \frac{5}{\sqrt{2}-1},
\end{align}
we have
\begin{align}
    \Esp{\mu\left(A_{n,k}^{(\ell)}(X_1, \Theta)\right) \big| X_1, \bm{\delta}_{\ell}(X_1, \Theta)} &\leq  2^{-K_\ell+2} \exp\left( \frac{5}{\sqrt{2}-1}\right), 
    %\\
    %&\leq 2^{-K_\ell+2} \prod_{j \leq k-2} \left( 1 + \frac{8}{\sqrt{n_j-1}}\right),
\end{align}
which proves the second statement. Note that 
\begin{align}
    & \Esp{\mu\left(A_{n,k}(X_1, \Theta)\right) \big| X_1, \bm{\delta}_1(X_1, \Theta), \hdots, \bm{\delta}_d(X_1, \Theta)}\\
    = & \E \left[ \prod_{\ell=1}^d \mu\left( A_{n,k}^{(\ell)}(X_1, \Theta) \right) \big| X_1, \bm{\delta}_1(X_1, \Theta), \hdots, \bm{\delta}_d(X_1, \Theta) \right]\\
    = & \prod_{\ell=1}^d \E \left[  \mu\left( A_{n,k}^{(\ell)}(X_1, \Theta) \right) \big| X_1, \bm{\delta}_{\ell}(X_1, \Theta) \right]\\
     \leq & \prod_{\ell = 1}^d \prod_{j : \delta_{j,\ell}=1, \atop j\leq k-2} \frac{1}{2} \left( 1 + \frac{5}{\sqrt{n_j-1}}\right)  \\
     \leq & \prod_{j \leq k-2}  \frac{1}{2} \left( 1 + \frac{5}{\sqrt{n_j-1}}\right)  \\
     \leq & 4 \times 2^{-k} \prod_{j \leq k-2} \left( 1 + \frac{5}{\sqrt{n_j-1}}\right)\\
     \leq & 4 \times 2^{-k} \exp \left(  \frac{5}{\sqrt{2}-1} \right).
\end{align}

\end{proof}

\subsection{Proof of the main result (median RF consistency)}

\begin{theorem}[Upper bound on the risk of the median forest]
Consider a generic pair $(X,Y)$ of random variables such that $Y = f^{\star}(X) + \varepsilon$, where $||\partial_\ell f^\star||_\infty^2$ exists for all $\ell \in \{1,\dots , d\}$, X is uniformly distributed on $[0,1]^d$ and the noise $\varepsilon$ satisfies, almost surely,   $\E[\varepsilon|X] = 0$  and $\mathds{V}[\varepsilon| X] \leq \sigma^2$. Consider $n \geq 16$ i.i.d.\ observations, where $n$ is a power of two, distributed as the generic pair $(X,Y)$. Then, the risk of the infinite median forest trained on this data set satisfies
\begin{align}
    \Esp{\left(\medrf(X) - f^\star(X) \right)^2 } \leq C_{1} d \left(\sum_{\ell =1}^d||\partial_\ell f^\star||_\infty^2 \right) \left( 1 - \frac{3}{4d} \right)^{\log_2 n} + \sigma^2  C_{2,d} (\log_2 n)^{- (d-1)/2},
\end{align}
with 
\begin{align}
C_1 =  1024 \exp \left( \frac{ 42 + \sqrt{5}}{2 - \sqrt{2}} \right) \quad \textrm{and} \quad C_{2,d} = 2 \left(32~\exp \left(  \frac{5}{\sqrt{2}-1} \right)~\right)^d d^{d/2}.
\end{align}
%, where $C_2 = 4 \exp \left(  \frac{10}{\sqrt{2}-1} \right) $ is given by Lemma~\ref{lem:leaf_median_tree}.
In particular, the infinite median forest is consistent, that is 
\begin{align}
\lim\limits_{n \to \infty} \Esp{\left(\medrf(X) - f^\star(X) \right)^2 } = 0.   
\end{align}
\end{theorem}

\begin{proof}

We begin with a simple bias/variance decomposition:
\begin{align*}
    & \Esp{\left(\medrf(X) - f^\star(X) \right)^2 } \\
    & = \Esp{\left( \E_{\Theta} \left[ \sum_{i=1}^n W_{ni}(X, \Theta) Y_i \right] - f^\star(X) \right)^2 } \\
    & = \Esp{\left(  \sum_{i=1}^n \E_{\Theta} \left[ W_{ni}(X, \Theta) \right] (f^\star(X_i) + \varepsilon_i) - f^\star(X) \right)^2 } \\
    & = \Esp{\left(  \sum_{i=1}^n \E_{\Theta} \left[ W_{ni}(X, \Theta) \right] (f^\star(X_i) - f^{\star}(X)) + \sum_{i=1}^n \E_{\Theta} \left[ W_{ni}(X, \Theta) \right]\varepsilon_i)  \right)^2 } \\
    & = \Esp{\left(  \sum_{i=1}^n \E_{\Theta} \left[ W_{ni}(X, \Theta) \right] (f^\star(X_i) - f^{\star}(X)) \right)^2} + \Esp{ \left( \sum_{i=1}^n \E_{\Theta} \left[ W_{ni}(X, \Theta) \right]\varepsilon_i)  \right)^2 }, 
\end{align*}
where the penultimate line comes from the fact that 
\begin{align}
\sum_{i=1}^n \E_{\Theta} \left[ W_{ni}(X, \Theta) \right] =  \E_{\Theta} \left[ \sum_{i=1}^n W_{ni}(X, \Theta) \right] = 1,
\end{align}
(since all leaves contain exactly one observation), and the last line results from a null cross product. 

\paragraph{Controlling the bias}
We have,
\begin{align}
& \Esp{\left(  \sum_{i=1}^n \E_{\Theta} \left[ W_{ni}(X, \Theta) \right] (f^\star(X_i) - f^{\star}(X)) \right)^2} \nonumber \\
& = \Esp{\left(  \E_{\Theta} \left[ \sum_{i=1}^n  W_{ni}(X, \Theta) (f^\star(X_i) - f^{\star}(X)) \right]  \right)^2}\\
& \leq \Esp{\left(   \sum_{i=1}^n  W_{ni}(X, \Theta) (f^\star(X_i) - f^{\star}(X))   \right)^2}\\
& \leq \Esp{\left(   \sum_{i=1}^n  \ind{X \in A_n(X_i, \Theta)} (f^\star(X_i) - f^{\star}(X))   \right)^2}\\
& \leq \Esp{  \sum_{i=1}^n  \ind{X \in A_n(X_i, \Theta)} \left(f^\star(X_i) - f^{\star}(X) \right)^2 },
\end{align}
because $ W_{ni}(X, \Theta)=  \ind{X \in A_n(X_i, \Theta)}$ and by applying twice Jensen inequality (third and fifth lines). Noticing that, 
\begin{align*}
    \ind{X \in A_n(X_i, \Theta)} |f^\star(X) - f^\star(X_i)| &\leq \sum_{\ell =1}^d||\partial_\ell f^\star||_\infty |X_i^{(\ell)} - X^{(\ell)}| \ind{X \in A_n(X_i, \Theta)}\\
    & \leq \sum_{\ell =1}^d||\partial_\ell f^\star||_\infty \mu(A_n^{(\ell)}(X, \Theta)) \ind{X \in A_n(X_i, \Theta)},
\end{align*}
we get, 
\begin{align}
 \Esp{  \sum_{i=1}^n  \ind{X \in A_n(X_i, \Theta)} \left(f^\star(X_i) - f^{\star}(X) \right)^2 } & \leq \Esp{  \sum_{i=1}^n  \ind{X \in A_n(X_i, \Theta)} \left(\sum_{\ell =1}^d||\partial_\ell f^\star||_\infty \mu(A_n^{(\ell)}(X, \Theta)) \right)^2 } \nonumber \\
& \leq \Esp{  \left(\sum_{\ell =1}^d||\partial_\ell f^\star||_\infty \mu(A_n^{(\ell)}(X, \Theta)) \right)^2 } \\
& \leq \left(\sum_{\ell =1}^d||\partial_\ell f^\star||_\infty^2 \right) \sum_{\ell =1}^d \Esp{  \mu(A_n^{(\ell)}(X, \Theta))^2 }.
\end{align}
where the last inequality directly results from Cauchy-Schwarz inequality. By Lemma~\ref{Lemme_cell_length_beta}, since $k = \lfloor \log_2 n \rfloor$,
\begin{align}
& \left(\sum_{\ell =1}^d||\partial_\ell f^\star||_\infty^2 \right) \sum_{\ell =1}^d \Esp{  \mu(A_n^{(\ell)}(X, \Theta))^2 } \leq C d \left(\sum_{\ell =1}^d||\partial_\ell f^\star||_\infty^2 \right) \left( 1 - \frac{3}{4d} \right)^{\log_2 n}, 
\end{align}
with 
\begin{align}
C =  1024 \exp \left( \frac{ 42 + \sqrt{5}}{2 - \sqrt{2}} \right).
\end{align}

\paragraph{Controlling the variance} Following \citet{biau2012analysis}, the variance term of the median forest writes
\begin{align}
    \E \left[ \left( \sum_{i=1}^n \E_{\Theta} \left[ W_{ni}(X, \Theta) \right] \varepsilon_i \right)^2 \right]
    & = \E \left[  \sum_{i=1}^n \left(\E_{\Theta} \left[ W_{ni}(X, \Theta) \right]\right)^2 \varepsilon_i^2 \right]\\
    & = \E \left[  \sum_{i=1}^n \left(\E_{\Theta} \left[ W_{ni}(X, \Theta) \right]\right)^2 \E \left[ \varepsilon_i^2 | X, X_1, \hdots, X_n \right] \right]\\
    & \leq \E \left[  \sum_{i=1}^n \left(\E_{\Theta} \left[ W_{ni}(X, \Theta) \right]\right)^2 \sigma^2 \right]\\
    & \leq  \sigma^2 n \E \left[   \left(\E_{\Theta} \left[ W_{n1}(X, \Theta) \right]\right)^2 \right],
\end{align}
where we have used the fact that the cross products are null (since $\E[\varepsilon_i|X_i]=0$). 
Since each leaf of the median tree contains exactly one observation, denoting $\Theta'$ an i.i.d. copy of $\Theta$, we have
\begin{align*}
    \left(\E_{\Theta} \left[ W_{n1}(X, \Theta) \right]\right)^2 & = \E_{\Theta} \left[ W_{n1}(X, \Theta) \right] \E_{\Theta'} \left[ W_{n1}(X, \Theta') \right]\\
    & = \E_{\Theta, \Theta'} \left[ W_{n1}(X, \Theta)  W_{n1}(X, \Theta') \right]\\
    & = \E_{\Theta, \Theta'} \left[ \ind{X \in A_n(X_1, \Theta)} \ind{X \in A_n(X_1, \Theta')}  \right].
\end{align*}
Consequently, 
\begin{align*}
    \E \left[ \left( \sum_{i=1}^n \E_{\Theta} \left[ W_{ni}(X, \Theta) \right] \varepsilon_i \right)^2 \right] & \leq  \sigma^2 n \E \left[ \ind{X \in A_n(X_1, \Theta)} \ind{X \in A_n(X_1, \Theta')}  \right].
\end{align*}
For all $\ell$, we let $A_n^{(\ell)}(X_1, \Theta)$ be the cell $A_n(X_1, \Theta)$ projected onto the $\ell$-th dimension. %$\mu\left(A_n(X_1,\Theta) \cap A_n(X_1,\Theta')\right)$ be the  and $\nu_\ell$ the length of the $\ell$-th side. %
Let also $\bm{\delta}_{\ell}(X_1, \Theta)$ be the vector whose components are defined as $\delta_{j,l} = 1 $ if the $j$-th cut of the cell $A_n^{(\ell)}(X_1, \Theta)$ is made along direction $\ell$ and $0$ otherwise. We define similarly $\bm{\delta}_{\ell}(X_1, \Theta')$ for the cell $A_n^{(\ell)}(X_1, \Theta')$.  We also let $K_{\ell} = \|\bm{\delta}_{\ell}(X_1, \Theta) \|_1$ (resp. $K_{\ell}'$) be the number of times the $\ell$-th direction is split in the tree built with $\Theta$ (resp. $\Theta'$). Then, 
\begin{align*}
    & \E \left[ \left( \sum_{i=1}^n \E_{\Theta} \left[ W_{ni}(X, \Theta) \right] \varepsilon_i \right)^2 \right] \\
    & \leq  \sigma^2 n \E \left[ \ind{X \in A_n(X_1, \Theta) \cap A_n(X_1, \Theta')}  \right]\\
    &= \sigma^2 n \E \left[ \prod_{\ell=1}^d \mu\left( A_n^{(\ell)}(X_1, \Theta) \cap A_n^{(\ell)}(X_1, \Theta')\right)\right]\\
    &= \sigma^2 n \E \left[ \E \left[ \prod_{\ell=1}^d \mu\left( A_n^{(\ell)}(X_1, \Theta) \cap A_n^{(\ell)}(X_1, \Theta')\right) \Bigg| X_1, \bm{\delta}_{\ell}(X_1, \Theta), \bm{\delta}_{\ell}(X_1, \Theta')\right]\right]\\
    &= \sigma^2 n \E \left[ \prod_{\ell=1}^d \E \left[  \mu\left( A_n^{(\ell)}(X_1, \Theta) \cap A_n^{(\ell)}(X_1, \Theta')\right) \Bigg| X_1, \bm{\delta}_{\ell}(X_1, \Theta), \bm{\delta}_{\ell}(X_1, \Theta')\right]\right].
\end{align*}

%\begin{align*}
%    V(m_n) &\leq \sigma^2 n  \Prob{X \in A_n(X_1,\Theta) \cap A_n(X_1,\Theta')} \\
%    &= \sigma^2 n  \Esp{\nu} \\
%    &= \sigma^2 n  \Esp{\Esp{\prod_{l=1}^d \nu_\ell \bigg| K_1,...,K_d, K_1',...,K_d'}} \\
%sigma^2 n  \Esp{\prod_{l=1}^d \Esp{ \nu_\ell \bigg| K_\ell, K_\ell'}}
%\end{align*}
%where $K_\ell$ (resp $K_\ell'$) denotes the number of splits made along the $l$-th direction in the tree $\Theta$ (resp. $\Theta'$) on the path to $X_1$. 
The last equality is obtained by conditional independence: indeed, as the $X_i$s are uniformly distributed, the positions of the coordinates do not influence each others. Therefore only the number of cuts along the other directions will influence the length the cell along a given direction, hence the conditional independence. Now,
\begin{align*}
    &\E \left[  \mu\left( A_n^{(\ell)}(X_1, \Theta) \cap A_n^{(\ell)}(X_1, \Theta')\right) \Bigg| X_1, \bm{\delta}_{\ell}(X_1, \Theta), \bm{\delta}_{\ell}(X_1, \Theta')\right] \\
    &\leq \Esp{ \min(\mu(A_n^{(\ell)}(X_1,\Theta)),\mu(A_n^{(\ell)}(X_1,\Theta')) \bigg| X_1,  \bm{\delta}_{\ell}(X_1, \Theta), \bm{\delta}_{\ell}(X_1, \Theta')} \\
    &=  \frac{1}{2} \left( \E\left[ \mu(A_n^{(\ell)}(X_1,\Theta)) \bigg| X_1, \bm{\delta}_{\ell}(X_1, \Theta) \right] + \E\left[ \mu(A_n^{(\ell)}(X_1,\Theta')) \bigg| X_1,  \bm{\delta}_{\ell}(X_1, \Theta') \right] \right) \\
    & \quad - \frac{1}{2} \Esp{ | \mu(A_n^{(\ell)}(X_1,\Theta)) - \mu(A_n^{(\ell)}(X_1,\Theta'))| \bigg| X_1, \bm{\delta}_{\ell}(X_1, \Theta), \bm{\delta}_{\ell}(X_1, \Theta')}.
\end{align*}
%by Lemma \textbf{[REF]}.
Moreover,
\begin{align*}
    &\Esp{ | \mu(A_n^{(\ell)}(X_1,\Theta)) - \mu(A_n'^{(\ell)}(X_1,\Theta))| \bigg| X_1, \bm{\delta}_{\ell}(X_1, \Theta), \bm{\delta}_{\ell}(X_1, \Theta')} \\
    &= \Esp{ | \mu(A_n^{(\ell)}(X_1,\Theta)) - \mu(A_n^{(\ell)}(X_1,\Theta'))| \left(\ind{K_\ell < K_\ell'} + \ind{K_\ell\geq K_\ell'} \right) \bigg| X_1, \bm{\delta}_{\ell}(X_1, \Theta), \bm{\delta}_{\ell}(X_1, \Theta')} \\
    &\geq \Esp{ \left(\mu(A_n^{(\ell)}(X_1,\Theta)) - \mu(A_n^{(\ell)}(X_1,\Theta'))\right)  \bigg| X_1, \bm{\delta}_{\ell}(X_1, \Theta), \bm{\delta}_{\ell}(X_1, \Theta')} \ind{K_\ell < K_\ell'}\\
    & \quad + \Esp{ \left(\mu(A_n^{(\ell)}(X_1,\Theta')) - \mu(A_n^{(\ell)}(X_1,\Theta))\right)  \bigg| X_1, \bm{\delta}_{\ell}(X_1, \Theta), \bm{\delta}_{\ell}(X_1, \Theta')} \ind{K_\ell\geq K_\ell'} \\
    & \geq \left( \E \left[ \mu(A_n^{(\ell)}(X_1,\Theta)) | X_1, \bm{\delta}_{\ell}(X_1, \Theta) \right] - \E \left[ \mu(A_n^{(\ell)}(X_1,\Theta')) |  X_1, \bm{\delta}_{\ell}(X_1, \Theta') \right] \right) \ind{K_\ell < K_\ell'}\\
    & \quad + \left( \E \left[ \mu(A_n^{(\ell)}(X_1,\Theta')) | X_1,  \bm{\delta}_{\ell}(X_1, \Theta') \right] - \E \left[ \mu(A_n^{(\ell)}(X_1,\Theta)) | X_1, \bm{\delta}_{\ell}(X_1, \Theta) \right] \right)  \ind{K_\ell\geq K_\ell'}.
    %\\
    %&= 2 \left(2^{-K_\ell} - 2^{-K_\ell'}\right)\ind{K_\ell' > K_\ell} + 2\left(2^{-K_\ell'} - 2^{-K_\ell}\right)\ind_{K_\ell \geq K_\ell'}
\end{align*}
Letting $B_\ell = \E \left[ \mu(A_n^{(\ell)}(X_1,\Theta)) | X_1, \bm{\delta}_{\ell}(X_1, \Theta) \right]$ and $B_{\ell}' = \E \left[ \mu(A_n^{(\ell)}(X_1,\Theta')) | X_1, \bm{\delta}_{\ell}(X_1, \Theta') \right]$, we have
\begin{align*}
& \E \left[  \mu\left( A_n^{(\ell)}(X_1, \Theta) \cap A_n^{(\ell)}(X_1, \Theta')\right) \Bigg| X_1, \bm{\delta}_{\ell}(X_1, \Theta), \bm{\delta}_{\ell}(X_1, \Theta')\right]\\
\leq & ~ \frac{1}{2} \left( B_\ell + B_{\ell}'\right) - \frac{1}{2}(B_{\ell} - B_{\ell}') \ind{K_\ell < K_\ell'} - \frac{1}{2}(B_{\ell}' - B_{\ell}) \ind{K_\ell \geq K_\ell'}\\
 \leq & ~ B_{\ell} \ind{K_\ell \geq K_\ell'} + B_{\ell}' \ind{K_\ell < K_\ell'}.
\end{align*}
Now, according to Lemma~\ref{lem:leaf_median_tree}, letting $C_2 = 4 \exp (5/(\sqrt{2}-1))$, we have $B_{\ell} \leq C_2 2^{-K_{\ell}}$ and $B_{\ell}' \leq C_2 2^{-K_{\ell}'}$. Therefore,
\begin{align*}
\E \left[  \mu\left( A_n^{(\ell)}(X_1, \Theta) \cap A_n^{(\ell)}(X_1, \Theta')\right) \Bigg| X_1, \bm{\delta}_{\ell}(X_1, \Theta), \bm{\delta}_{\ell}(X_1, \Theta')\right] 
&\leq    C_2 2^{-K_{\ell}} \ind{K_\ell \geq K_\ell'} + C_2 2^{-K_{\ell}'}  \ind{K_\ell < K_\ell'} \\
& \leq C_2 2^{-\max(K_\ell, K_\ell')}.
\end{align*}

%\begin{align*}
%    \Esp{ \nu_\ell \bigg| K_\ell, K_\ell'} &\leq  2^{-K_\ell} + 2^{-K_\ell'} - \left(2^{-K_\ell} - 2^{-K_\ell'}\right)\ind{K_\ell' > K_\ell} -  \left(2^{-K_\ell'} - 2^{-K_\ell}\right)\ind_{K_\ell \geq K_\ell'} \\
%    &= 2\cdot 2^{-K_\ell} \ind{K_\ell\geq K_\ell'} + 2\cdot 2^{-K_\ell'}\ind{K_\ell' >K_\ell} \\
%    &= 2\cdot 2^{-\max(K_\ell, K_\ell')}.
%\end{align*}

Overall,
\begin{align}
    \E \left[ \left( \sum_{i=1}^n \E_{\Theta} \left[ W_{ni}(X, \Theta) \right] \varepsilon_i \right)^2 \right] & \leq  \sigma^2 n \E \left[ \prod_{\ell=1}^d \E \left[  \mu\left( A_n^{(\ell)}(X_1, \Theta) \cap A_n^{(\ell)}(X_1, \Theta')\right) \Bigg| X_1, \bm{\delta}_{\ell}(X_1, \Theta), \bm{\delta}_{\ell}(X_1, \Theta')\right]\right] \\
    & \leq \sigma^2 n \Esp{\prod_{\ell=1}^d C_2 2^{-\max(K_\ell, K_\ell')}} \\
    & \leq \sigma^2 C_2^{d} n ~ \Esp{2^{-\sum_{\ell=1}^d \max(K_\ell, K_\ell')}} \\
    & \leq \sigma^2 C_2^{d} n ~~ 2^{-k_n} \Esp{2^{-\sum_{\ell=1}^d |K_\ell-K_\ell'|}}, \label{eq:th_penultimate}
\end{align}
since 
\begin{align*}
\sum_{\ell=1}^d \max(K_\ell, K_\ell') &= \frac{1}{2} \sum_{\ell = 1}^d K_{\ell} + \frac{1}{2} \sum_{\ell = 1}^d K_{\ell}' + \frac{1}{2} \sum_{\ell = 1}^d |K_{\ell} - K_{\ell}'|\\
& = k_n + \frac{1}{2} \sum_{\ell = 1}^d |K_{\ell} - K_{\ell}'|.
\end{align*}
According to Lemma~S.1 from \citet{klusowski2021sharp} (see Supplementary Materials), one has 
\begin{align}
\Esp{2^{-\sum_{\ell=1}^d |K_\ell-K_\ell'|}} \leq \frac{8^d~ d^{d/2}}{k_n^{(d-1)/2}}. \label{eq:th_penultimate1}
\end{align}
Finally, combining \eqref{eq:th_penultimate} and \eqref{eq:th_penultimate1},  the variance of the median forest is upper bounded by 
\begin{align}
    \E \left[ \left( \sum_{i=1}^n \E_{\Theta} \left[ W_{ni}(X, \Theta) \right]  \varepsilon_i \right)^2 \right] 
    %& \leq \sigma^2 C_2^{d} n ~ 2^{-k_n} \Esp{2^{-\sum_{\ell=1}^d |K_\ell-K_\ell'|}}\\
    & \leq \sigma^2 C_2^{d} n ~ 2^{-k_n} \frac{8^d~ d^{d/2}}{k_n^{(d-1)/2}}\\
    & \leq 2 \sigma^2  \left(8~C_2~d^{1/2} \right)^d (\log_2 n)^{- (d-1)/2},
\end{align}
since $k_n = \lfloor \log_2 (n) \rfloor.$ All in all,
\begin{align*}
    & \Esp{\left(\medrf(X) - f^\star(X) \right)^2 } \\
    & \leq C d \left(\sum_{\ell =1}^d||\partial_\ell f^\star||_\infty^2 \right) \left( 1 - \frac{3}{4d} \right)^{\log_2 n} + 2 \sigma^2  \left(8~C_2~d^{1/2} \right)^d (\log_2 n)^{- (d-1)/2}
\end{align*}
with $C_2 = 4 \exp (5/(\sqrt{2}-1))$ and
\begin{align}
C =  1024 \exp \left( \frac{ 42 + \sqrt{5}}{2 - \sqrt{2}} \right).
\end{align}
%\begin{align}
%    \Esp{\nu} &= \Esp{\prod_{l=1}^d \nu_\ell} \\
%    &\leq 2^d\Esp{2^{-\sum_{l=1}^d \max(K_\ell, K_\ell')}} \\
%    &= 2^d2^{-k_n} \Esp{2^{-\sum_{l=1}^d |K_\ell-K_\ell'|}}.
%\end{align}

%The above quantity is controlled in Lemma 1 of  \textbf{[KLUSOWSKI]}:
%\begin{align*}
%    \Esp{\nu} &\leq 2^d 2^{-k_n} \frac{d^{d/2}8^d}{k_n^{(d-1)/2}}.
%\end{align*}
%Choosing $k_n = \log_2 n $ yields
%\begin{align*}
%    V(m_n) & \leq \frac{\sigma^2 d^{d/2}16^d}{(\log_2 n)^{(d-1)/2}}.
%\end{align*}

\end{proof}

\subsubsection{Controlling the variance of an interpolating Median RF in an asymptotic high-dimensional setting}
\label{subsec:var_high_dim}

The following result shows the decrease of the variance of the Median RF under an asymptotic high-dimensional framework. It is also numerically illustrated in Section \ref{sec:exp_median_rf_high_dim}.

\begin{proposition} \label{prop:var_high_dim}

For all $d > \log_2 n$, the variance of the infinite interpolating Median RF $\infforest^{\mathrm{MedRF}}$ verifies
\begin{align*}
    V(\infforest^{\mathrm{MedRF}})  = \E \left[ \left( \sum_{i=1}^n \E_{\Theta} \left[ W_{ni}(X, \Theta) \right]  \varepsilon_i \right)^2 \right]  
    & \leq \frac{4 C_2^2\sigma^2}{n}  + 2 C_2\sigma^2 \left(1- \exp \left( -\frac{\log_2^2 n}{d-\log_2 n} \right) \right),
\end{align*}
where $C_2 = 4 \exp{\left(5/(\sqrt{2}-1)\right)}$.
Suppose that the input dimension $d$ dominates $\log_2 ^2 n$ asymptotically  ($d \gg \log_2 ^2 n$), then the variance tends to 0 (as $n,d$ tends to infinity), with a rate of the order of $\max(\frac{\log^2 n}{d}, \frac{1}{n})$. 
\end{proposition}
The proof is given below. 
This results shows that the Median RF benefits from an increase of the dimension as it will improve its averaging effect and help to reduce the variance. 
Of course, in such a setting, 
the variance is only one part of the story, and a control on the bias  becomes a real hindrance (as the approximation error may explode), unless extra model assumptions are formulated.
For instance, consider for any input dimension $d$ the case of a linear model, i.e.\  $Y = X^\top \theta+ \varepsilon$ for $\theta \in \mathds{R}^d$ and such that $ \|\theta\|_2 \leq  C /  \sqrt{d}$, with $C>0$ a constant.
{ 
One can actually show that in such a setting, the bias term remains bounded as $n$ (and $d$) grows towards infinity (using for example the analysis conducted in the next theorem).}
This echoes in particular the behavior of ridgeless least squares estimator in modern interpolation regimes \citep[see, ][]{hastie2022surprises}.

\begin{proof}[Proof of Proposition~\ref{prop:var_high_dim}]
\label{proof:var_high_dim}
A typical bias-variance decomposition yields (see e.g.\ \cite{biau2012analysis})
\begin{align}
    V(\infforest^{\mathrm{MedRF}}) \leq \sigma^2 n \Prob{X \in A_n(X_1,\Theta) \cap A_n(X_1, \Theta')}
\end{align}
with $\Theta'$ an independent copy of $\Theta$. Recalling that the depth is chosen as $k = \lfloor \log_2 n \rfloor$. Consider the event 
$$E = E(\Theta, \Theta', X_1, k):=\{ \Theta \text{ and } \Theta' \text{ do not cut on common directions on the path to } X_1 \}.$$
Denote $M(\Theta, X_1)$ the number of distinct directions chosen by the tree $\Theta$ to produce the leaf containing $X_1$ (upper bounded by $\log_2n$). Then,
\begin{align}
    \Prob{E} &\geq \Esp{\left(\frac{d - M(\Theta,X_1)}{d}\right)^{\log_2n}} \\
    &\geq \left(\frac{d - \log_2 n}{d} \right)^{\log_2 n} \\
    &= \exp \left(  \log_2 n \log \left(1 - \frac{\log_2 n}{d} \right) \right) \\
    &\geq \exp \left( -\frac{\log_2^2n}{d - \log_2 n } \right),
\end{align}
using, for all $x \in [0,1)$, $\log (1 - x) \geq - x / (1-x)$. The above probability tends to 1 as soon as $d \gg \log_2^2 n $.
Then,
\begin{align}
    & \Prob{X \in A_n(X_1,\Theta) \cap A_n(X_1, \Theta')} \\
    = & ~\Prob{\{X \in A_n(X_1,\Theta) \cap A_n(X_1, \Theta')\} \cap  E}  + \Prob{\{X \in A_n(X_1,\Theta)\cap A_n(X_1, \Theta') \} \cap E^c} \nonumber\\
    \leq & ~ \Prob{X \in A_n(X_1,\Theta) | X \in A_n(X_1, \Theta'), E} \Prob{X \in A_n(X_1, \Theta')}   + \Prob{\{X \in A_n(X_1,\Theta)\} \cap E^c}.
%    \leq & ~C_2 ~ 2^{-k_n} \Prob{X \in A_n(X_1,\Theta) | X \in A_n(X_1, \Theta'), E} + \left(1 - \Prob{E}\right) C_2 ~ 2^{-k_n}\\
%    \leq & ~C_2^{2} ~ 2^{-2 k_n} + \left(1- e^{-\frac{\log_2^2n}{d - \log_2 n }}\right) C_2 ~ 2^{-k_n},
\end{align}

Applying Lemma \ref{lem:leaf_median_tree} (Line \eqref{eq:volume_cell_medianRF}) yields
\begin{align}
    \Prob{X \in A_n(X_1, \Theta')} &= \Esp{\mu\left(A_n(X_1,\Theta'\right)} \\
    & = \Esp{\Esp{\mu(A_{n}(X_1, \Theta)) | X_1, \bm{\delta}_1(X_1, \Theta), \hdots, \bm{\delta}_d(X_1, \Theta)}}\\
    &\leq C_2 2^{-k}
\end{align}
with $C_{2} =  4 \exp \left(  \frac{5}{\sqrt{2}-1} \right)$. Moreover, conditional on $E$, $\{X \in A_n(X_1,\Theta)\}$ and $\{X \in A_n(X_1,\Theta') \}$ are independent as $\Theta$ and $\Theta'$ do not share any common direction on the path to $X_1$ and therefore the splits in $\Theta$ and $\Theta'$ are performed on independent sample components (by uniformity of $X$ and $X_1$). Therefore, also by Lemma \ref{lem:leaf_median_tree}, 
\begin{align}
    \Prob{X \in A_n(X_1,\Theta) | X \in A_n(X_1, \Theta'), E} &= \Esp{\mu\left(A_n(X_1,\Theta)\right)} \\
    & = \Esp{\Esp{\mu(A_{n}(X_1, \Theta)) | X_1, \bm{\delta}_1(X_1, \Theta), \hdots, \bm{\delta}_d(X_1, \Theta)}}\\
    &\leq C_2 2^{-k}.
\end{align}

Similarly, the volume $\mu\left(A_n(X_1,\Theta)\right)$ is independent of the directions chosen to build the leaf, therefore
\begin{align*}
     \Prob{\{X \in A_n(X_1,\Theta)\} \cap E^c} &=  \Prob{X \in A_n(X_1,\Theta)} \Prob{E^c} \\
     &\leq C_2 2^{-k} \left(1- \exp \left( -\frac{\log_2^2n}{d - \log_2 n }\right) \right).
\end{align*}

Overall,
\begin{align*}
    \Prob{X \in A_n(X_1,\Theta) \cap A_n(X_1, \Theta')} &\leq C_2^2 2^{-2k} + C_2 2^{-k} \left(1- \exp \left( -\frac{\log_2^2n}{d- \log_2 n} \right) \right)
\end{align*}
and
\begin{align}
    V(\infforest^{\mathrm{MedRF}}) \leq C_2^2 n \sigma^2 2^{-2k} + n\sigma^2 C_2 \left(1- \exp \left( -\frac{\log_2^2n}{d- \log_2 n} \right) \right) 2^{-k}.
\end{align}

Since $k = \lfloor \log_2 n \rfloor$, we have $2^{-k} \leq 2/n$ and 
\begin{align}
     V(\infforest^{\mathrm{MedRF}}) \leq \frac{4C_2^2 \sigma^2}{n}  + 2C_2 \sigma^2 \left(1- e^{-\frac{\log_2^2 n}{d-\log_2 n}}\right).
\end{align}
\end{proof}

\subsubsection{Proof of Proposition \ref{prop:vol_interp_zone_AdaCRF} (Interpolation volume of Median RF)}

It is possible to conduct a one-dimensional analysis and then to extend the result to the multi-dimensional case by a simple multiplication. Indeed all the leaves are determined coordinate per coordinate, therefore the interpolation area is the product of all interpolation areas along each direction.

Let $Z_1,\hdots, Z_n$ be $n$ i.i.d.\ random variables uniformly distributed over $[0,1]$.
As in the infinite Median RF, the univariate trees, i.e., built by cutting along one direction only, appear almost surely. Then, the length of a leaf of such tree is bounded in expectation by $Z_{(k+1)} - Z_{(k-1)}$ where $Z_{(i)}$ indicates the $i$-the statistical order. Moreover, it is known that $Z_{(k)}$ follows a Beta distribution of parameters $(k, n-k+1)$. Therefore,
\begin{align}
    \Esp{Z_{(k+1)} - Z_{(k-1)}} &= \frac{k+1}{n+1} - \frac{k-1}{n+1} \\
    &\leq \frac{2}{n}.
\end{align}

Now, as $X_1, ..., X_n$ are i.i.d.\ and uniformly distributed over $[0,1]^d$, for any data point $x \in [0,1]^d$ we simply have that
\begin{align*}
    \Esp{\mu(\mathcal{A}_{min, x})} \leq \frac{2^d}{n^d}.
\end{align*}

Finally, since by definition all interpolation zones are disjoint and the interpolation area is the union of $n$ interpolation areas, we have
\begin{align*}
    \Esp{\mu(\mathcal{A}_{min})} \leq \frac{2^d}{n^{d-1}}
\end{align*}
which ends the proof.

%%%%%%%%%%%%%%%%%%%%%%%%%%%%%%%%%%%%%%%%%%%%%%%%%%%%%%%%%%%%%%%%%%

%%%%%%%%%%%%%%%%%%%%%%%%%%%%%%%%%%%%%%%%%%%%%%%%%%%%%%%%%%%%%%%%%%%

\subsection{Proofs of Section \ref{sec:BreimanRF} (Interpolation volume of Breiman RF)}

\begin{proof}[Proof of Proposition \ref{prop:interpol_limit_zone}]
\label{proof:interp_zone_breiman}

Before diving into the computations, let us recall two facts about Breiman RF construction.
First, in CART, each cut is made at the middle of two consecutive points in a given direction.  
Second, considering all univariate trees (trees whose splits are performed along one single direction), the probability of cutting between all pairs of successive points along all dimensions is strictly positive. Therefore, for a given point $X_i$, one can define the minimal interpolation zone around $X_i$ as \begin{align}
\mathcal{A}_{min, X_i} := \bigcap_{M \in \mathds{N}, \boldsymbol{\Theta}_M} \mathcal{A}_{X_i, \boldsymbol{\Theta}_M}.
\end{align}
The boundaries of this area are given for each direction by the cuts between $X_i$ and its \textit{neighbor points} respectively to the considered direction, as illustrated on Figure \ref{fig:interp_zone}. 
%More rigorously said, it is built by cutting along each dimension between $X_i$ and its closest point from below respectively to this dimension as well as between $X_i$ and the closest point from above.

\begin{figure}[h]
    \centering
    \includegraphics[trim={2cm 18.4cm 4cm 2.5cm}, clip, width=0.6\textwidth]{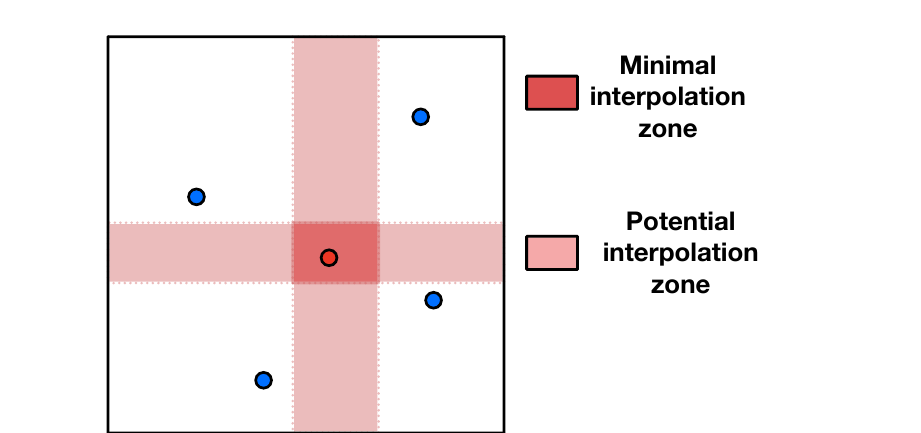}
    \caption{Different interpolation zones of a data point (in red).}
    \label{fig:interp_zone}
\end{figure}

\begin{enumerate}
\item The interpolation zone is the union of $n$ interpolation zones, each one containing a single $X_i$. We denote $\mathcal{A}(m_{M,n}(.,\boldsymbol{\Theta}_M)) = \mathcal{A}_{X_1,\boldsymbol{\Theta}_M} \cup  ... \cup \mathcal{A}_{X_n, \boldsymbol{\Theta}_M}$ with $\mathcal{A}_{X_i, \Theta_M} = \{x \in [0,1]^d, m_{M,n}(x,\boldsymbol{\Theta}_M) = Y_i\}$.
We begin with a one-dimensional analysis, and consider, without loss of generality, the first variable. We let $Z_1 := X_1^{(1)}, ..., Z_n := X_n^{(1)}$ the first components of the observations $X_1, \hdots, X_n$. 
%We denote $X_i^{(j)}$ the $j$-th feature of $X_i$, for all $j \in \{1,\hdots, d\}$ and $i \in \{1, \hdots, n\}$ and we focus on the first variable $X^{(1)}$. 
As $X_1, \hdots, X_n$ are i.i.d.\ and follow a uniform distribution over $[0,1]^d$, $Z_1, ..., Z_n$ are i.i.d.\ and uniformly distributed on $[0,1]$. 
We consider the interpolation area at $x = Z_n$ and we reason conditional on $Z_n$ in the following. The length (volume) of $\mathcal{A}_{min,x}$ restricted to the first dimension is simply given by the sum of the distance from $x$ to its closest point on the left side and to its closest point on the right side (divided by 2 as the cut are made in the middle of two points).
Therefore,
\begin{align}
\mu(A_{min, x}) = \frac{1}{2} \left( x - \underset{\{Z_i, Z_i < x\} \cup \{0\}}{\max} Z_i + \underset{\{Z_i, Z_i > x\} \cup \{1\}}{\min} Z_i - x \right).
\end{align}

All computations are made conditionally on $x$. Denoting $N_x$ the cardinal of the set $\{Z_i: Z_i < x \text{ with } 1 \leq i <n\}$, we have for any $t \in [0,x/2)$,
\begin{align}
    &\Prob{ \frac{1}{2} \left( x- \underset{\{Z_i, Z_i < x\} \cup \{0\} }{\max} Z_i\right)  \leq t \hspace{0.1cm} \big| x} \\
    &= 1 - \Prob{\underset{\{Z_i, Z_i < x\} \cup \{0\} }{\max} Z_i < x- 2t \hspace{0.1cm} \big| x} \\
    %&= 1 - \Prob{\displaystyle \bigcup_{k=0}^n \left( \left(N_{Z_i \leq x} = k \right)  \cap \left( (Z_{i_1} < x- 2t) \cap \hdots (Z_{i_k} < x-2t) | Z_{i_1} \leq x, \hdots, Z_{i_k} \leq x \right) \right)} \\
    &= 1 - \Esp{\Esp{\Prob{(Z_{i_1} < x -2t) \cap ... \cap (Z_{i_{N_x}} < x-2t) \big| N_x, Z_{i_1} < x, ..., Z_{i_{N_x}} < x, x}} \big| x}  \\
    &= 1 - \Esp{\mathbb{P}\left(Z_1 < x-2t | Z_1 \leq x, x\right)^{N_x} \big|x} \\
    &= 1- \displaystyle \sum_{k=0}^{n-1} \Prob{N_{x} = k | x} \Prob{Z_1 < x-2t | Z_1 < x, x}^k \\
    &= 1 - \displaystyle \sum_{k=0}^{n-1} \Prob{N_{x} = k | x} \left(\frac{x-2t}{x}\right)^k  \\
    &= 1 - \left( (1-x) + x \left( \frac{x-2t}{x} \right)\right)^{n-1} \\
    &= 1 - (1-2t)^{n-1}
\end{align}
where the penultimate equality is obtained by noticing that $N_{x}$ is a binomial of parameters $(n-1,x)$ and computing its probability-generating function.
So for all $t \geq 0$,
\begin{align*}
    \Prob{ \frac{1}{2} \left(x- \underset{\{Z_i, Z_i < x\} \cup \{0\} }{\max} Z_i\right)  \leq t |x} = 1 - (1-2t)^{n-1} \ind{t< x/2} .
\end{align*}

By symmetry,
\begin{align*}
    \Prob{ \frac{1}{2} \left( \underset{\{Z_i, Z_i > x\} \cup \{1\} }{\min} Z_i - x \right)  \leq t | x} = 1 - (1-2t)^{n-1} \ind{t> (1-x)/2} .
\end{align*}

Overall, using the fact that for any positive variable $Z$ with cumulative function $F_Z$, $\Esp{Z} = \int (1-F_Z)$, we have
\begin{align*}
    \Esp{\mu(\mathcal{A}_{min,x}) | x} &= \int_0^{x/2} (1-2u)^{n-1} du + \int_0^{(1-x)/2} (1-2u)^{n-1} du \\
    &= \frac{1}{2n} \left( 2 - (1-x)^{n} -x^{n} \right) \\
    &\leq \frac{1}{n} \left( 1 - \frac{1}{2^{n}} \right) .
\end{align*}

Now, as $X_1, ..., X_n$ are i.i.d.\ and uniformly distributed over $[0,1]^d$, for any data point $x \in [0,1]^d$ we simply have that $$\mathcal{A}_{min,x} = \displaystyle \bigtimes_{j=1}^d \mathcal{A}_{min, x^{(j)}}.$$
Therefore, 
\begin{align*}
    \Esp{\mu(\mathcal{A}_{min, x})} \leq \frac{1}{n^d} \left(1- 2^{-n} \right)^d.
\end{align*}

Finally, since by definition all interpolation zones are disjoint, we have
\begin{align*}
    \Esp{\mu(\mathcal{A}_{min})} \leq \frac{1}{n^{d-1}} \left(1- 2^{-n} \right)^d.
\end{align*}
\item It is enough to notice that the minimal interpolation zone is the intersection of all the potential interpolation zones. It is reached when the forest contains all the possible cuts. Then, as the probability of any given cut appearing is strictly greater than 0 by hypothesis, the probability of its appearance in the infinite forest is one. Therefore almost surely, when $M$ grows to infinity, the interpolation zone of the forest reaches the minimal interpolation zone.
\end{enumerate}
\end{proof}

\newpage

\section{Experiments}
\label{sec:exp_setting}

For all experiments, we consider four different regression models, most of which have been already con\-si\-dered in \cite{van2007super}: Model 1 is additive without noise ($d=2$), Model 2 is polynomial with interactions ($d=8$), Model 3 is the sum of elementary terms that contain non-polynomial interactions ($d=6$) and Model 4 ($d=5$) corresponds to a generalized linear model:
\begin{itemize}
    \item \textbf{Model 1:} $d=2$, $Y = 2X_1^2 + \exp(-X_2^2)$
    \item \textbf{Model 2:} $d=6$, $Y = X_1^2 + X_2^2X_3e^{-|X_4|} + X_5 - X_6 + \mathcal{N}(0, 0.5)$
    \item \textbf{Model 3:} $d=8$, $Y = X_1X_2 + X_3^2 - X_4X_5 + X_6X_7 -X_8^2 + \mathcal{N}(0, 0.5)$
    \item \textbf{Model 4:} $d =5$, $Y = 1 / (1+\exp(-10 (\sum_{i=1}^d X_i - 1/2))) + \mathcal{N}(0,0.05)$
    \replace{}{\item \textbf{Model 5}: $d=4$, $Y= - \sin(2X_1 X_2) + X_2^2 + X_3 - e^{X_4} + \mathcal{N}(0,0.5)$ 
    \item \textbf{Model 6}: $d=8$, $Y= \mathds{1}_{\{X_1 \geq 0 \}} + X_2^3 + \mathds{1}_{\{X_3 + X_5 - X_6 - X_7 -X_8\geq 1 \}} + e^{-X_2^2} +  \mathcal{N}(0,0.5)$  
    \item \textbf{Model 7}: $d=4$, $Y= X_1 + 2(X_2-1)^2 + \frac{\sin(2\pi X_3)}{2- \sin(2\pi X_3)} + 2 \sin(2\pi X_4) + 2 \cos(2\pi X_4) + 4 \sin(2 \pi X_4)^2 + 4 \cos(2 \pi X_4)^2 +  \mathcal{N}(0,0.5)$
    \item \textbf{Model 8}: $d=4$, $Y= X_1 + 3 X_2^2 - 2 e^{X_3} + X_4 $.}
\end{itemize}

All the experiments are conducted using Python3. We use Scikit-learn RandomForestRegressor class to implement the Breiman RF model. We coded CRF, KeRF and AdaCRF models ourselves, mainly relying on \textit{numpy} and \textit{joblib} libraries for computation optimisation. Experiments were run on 4 16-cores CPU and took at most a few hours to run.

\subsection{Consistency experiments}

For all consistency experiments, the dataset was divided into a train dataset ($80\%$ of the data) and a test dataset ($20\%$) of the data.

The parameters of the estimators were set as follows:
\begin{itemize}
    \item all RF estimators have 500 \textit{trees} to mimic the behavior of the infinite RF.
    \item parameter \textit{bootstrap} is set to \textit{False} for all estimators in order preserve the interpolation property, or set to \textit{True} when specified.
    \item all other parameters are set to default value.
\end{itemize}

%\subsubsection{Inconsistency of interpolating CART }

%\begin{figure}[h]
%    \centering
%    \includegraphics[clip, width=0.6\textwidth]{cart_consistency.png}
%    \caption{CART consistency results: excess risk w.r.t.\ sample sizes. For each sample size, the experiment is repeated 30 times: we represent the mean over the 30 tries (dot line) and the std (filled zone).}
%    \label{fig:cart_inconsistency}
%\end{figure}
%\la{Voit-on vraiment que CART est inconsistent sur cette figure ? Ca ne saute pas aux yeux...}
%\la{Est-ce qu'on ne ferait pas sauter cette image + section ?}

%Figure \ref{fig:cart_inconsistency} provides an empirical verification of the theoretical inconsistency of an interpolating CART. We can see that the excess risk of CART is high for Models 2 and 3 and non-decreasing for the Model 4.

\subsubsection{Consistency of KeRF in the mean interpolation regime}
\label{sec:kerf_consistency_exp}

We train a centered KeRF (with $M=500$) of depth fixed to $\integer{\log_2 n} +1$ (mean interpolation regime) for different sample sizes $n$ and evaluate the empirical quadratic risk on the test set. %Details about the implementation and supplementary experiments can be found in Appendix \ref{sec:exp_setting}.
\begin{figure}[h]
    \centering
    \includegraphics[clip, width=0.8\textwidth]{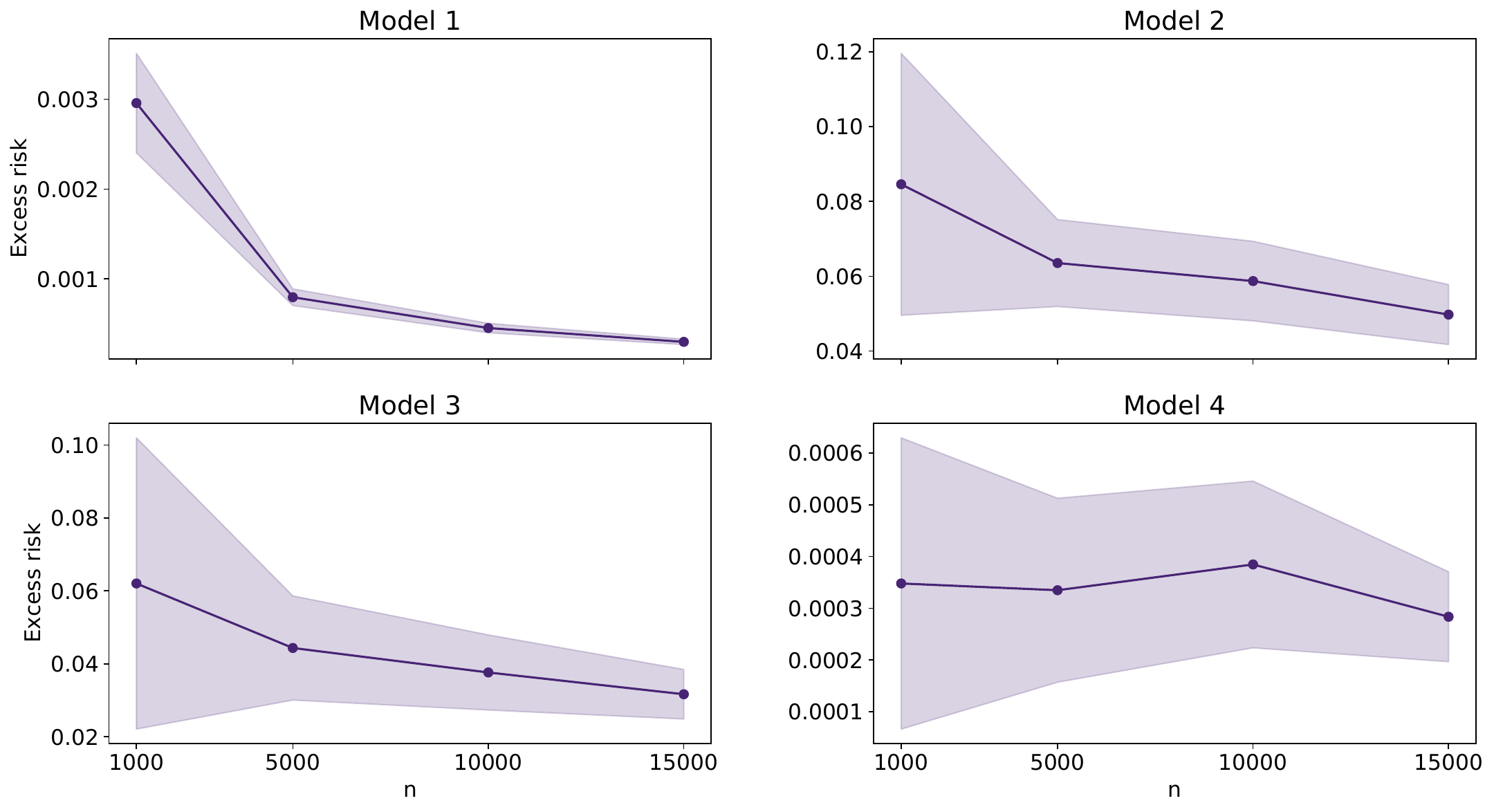}
    \caption{KeRF consistency results: excess risk w.r.t.\ sample sizes. For each sample size $n$, the experiment is repeated 30 times: we represent the mean over the 30 tries (bold line) and the mean $\pm$ std (filled zone).}
    \label{fig:kerf_consistency}
\end{figure}

\paragraph{Results} 
On Figure \ref{fig:kerf_consistency}, for all models, the risk decreases toward zero as the number of samples $n$ increases (with slow convergence rates).
These numerical results, even though obtained for a finite KeRF with a large number $M=500$ of centered trees, support the theoretical consistency of the infinite KeRF in the mean interpolation regime (see Theorem \ref{theoreme_consistency_centred_forest_approximation}).

\subsubsection{Consistency of Median RF in the interpolation regime}
\label{sec:median_rf_empirical}

We analyze the empirical performances of {Median RF} in noiseless and noisy settings on the models specified above.
For each model, given a training set, we train Median RF (with $M=500$ trees) until pure leaves are reached, and measure its excess risk on a test set. 
\begin{figure}[h]
    %\vspace{-1.2cm}
    \centering
    \includegraphics[clip, width=0.8\textwidth]{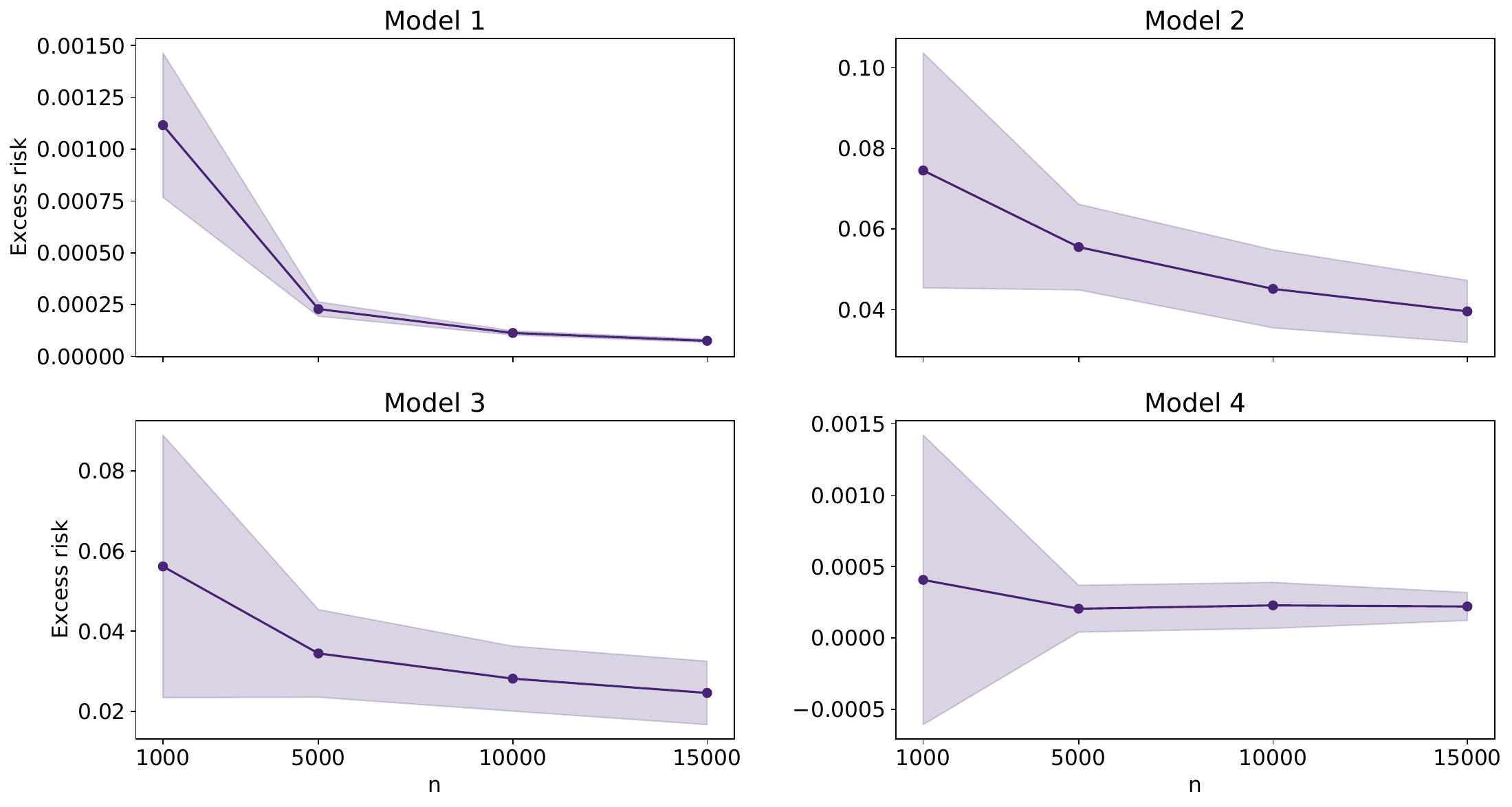}
    \caption{Consistency results for a Median RF with $M=500$ trees: excess risk w.r.t.\ the sample size $n$.
    For each sample size, the experiment is repeated 30 times: we represent the mean over the 30 tries (bold line) and the mean $\pm$ std (filled zone).}
    \label{fig:adacrf_consistency}
\end{figure}

Figure \ref{fig:adacrf_consistency} shows that the excess risk of a Median RF decreases as $n$ grows. These em\-pi\-ri\-cal performances lend support to the idea that Median RF are consistent even with a finite number of trees and beyond the noiseless setting.

%\subsubsection{Decrease of the variance of the Median RF in a high-dimensional setting} \label{sec:exp_median_rf_high_dim}

%Numerical experiments (Figure \ref{fig:median_high_dim}) corroborate the decrease of the variance of interpolating Median RF when $d$ increases. The model involves no signal and only noise (with specified variance $\sigma^2$).
%
%\begin{figure}[ht]
    %\vspace{-1.2cm}
%    \centering
%    \includegraphics[clip, width=0.5\textwidth]{high_dim_MEDIAN_sigma10-100.pdf}
%    \caption{Decrease of the variance of the interpolating Median RF w.r.t.\ dimension $d$. 10 repetitions per boxplot, 5000 training points and 50000 testing points were used for each repetition. The Median RF contains 500 trees.}
%    \label{fig:median_high_dim}
%\end{figure}

\subsubsection{Consistency of Breiman RF, additional models to Figure 1}

\begin{figure}[h]
    \centering
    \includegraphics[clip, width=0.8\textwidth]{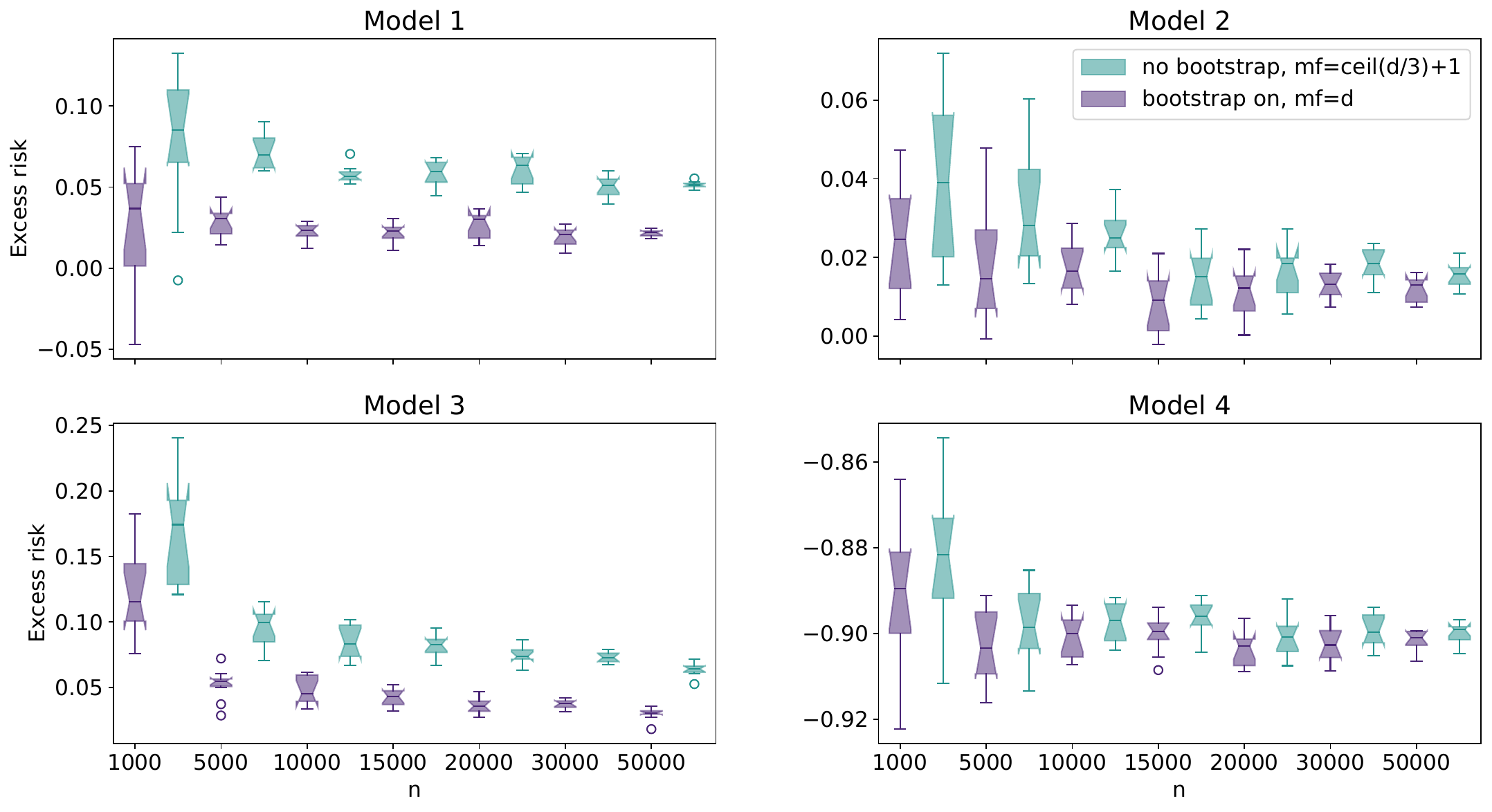}
    \caption{\replace{}{Consistency of Breiman RF: excess risk w.r.t\ the sample size $n$. RF parameters: 2000 trees, max-depth set to None, max-features$=1$. Boxplots over 10 tries.}}
    \label{fig:my_label}
\end{figure}

\begin{figure}[h]
    \centering
    \includegraphics[clip, width=0.8\textwidth]{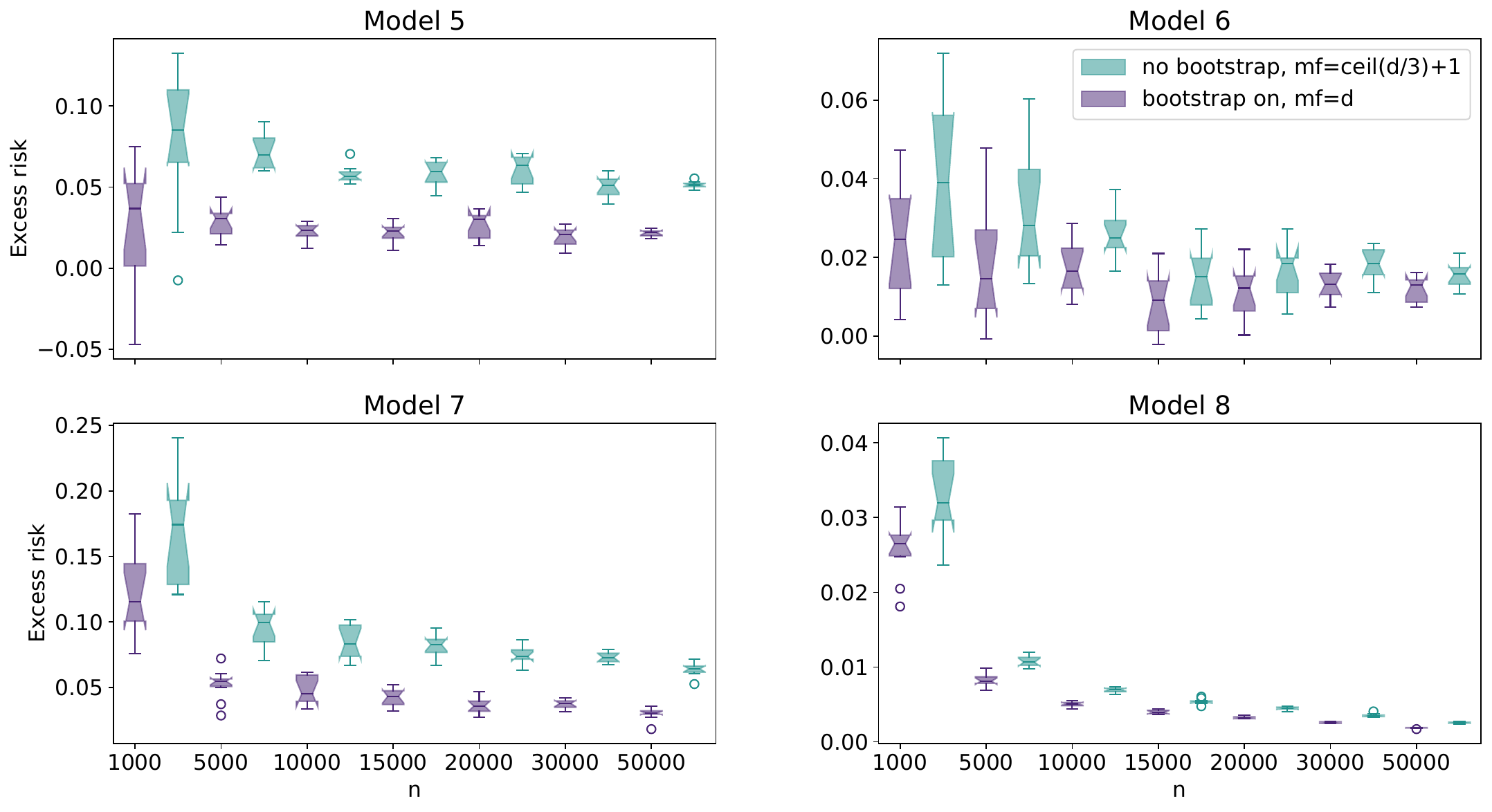}
    \caption{\replace{}{Consistency of Breiman RF: excess risk w.r.t\ the sample size $n$. RF parameters: 2000 trees, max-depth set to None, max-features$=1$. Boxplots over 10 tries.}}
    \label{fig:my_label}
\end{figure}

\FloatBarrier

\subsubsection{Consistency of Breiman RF with max-feature$=1$}

On Figure \ref{fig:breiman_consistency_mtry_one}, we see that the excess risk of a Breiman RF with the max-features parameter set to $1$ is decreasing towards $0$ as $n$ increases. This RF seems consistent for all models.

\begin{figure}[h]
    \centering
    \includegraphics[clip, width=0.8\textwidth]{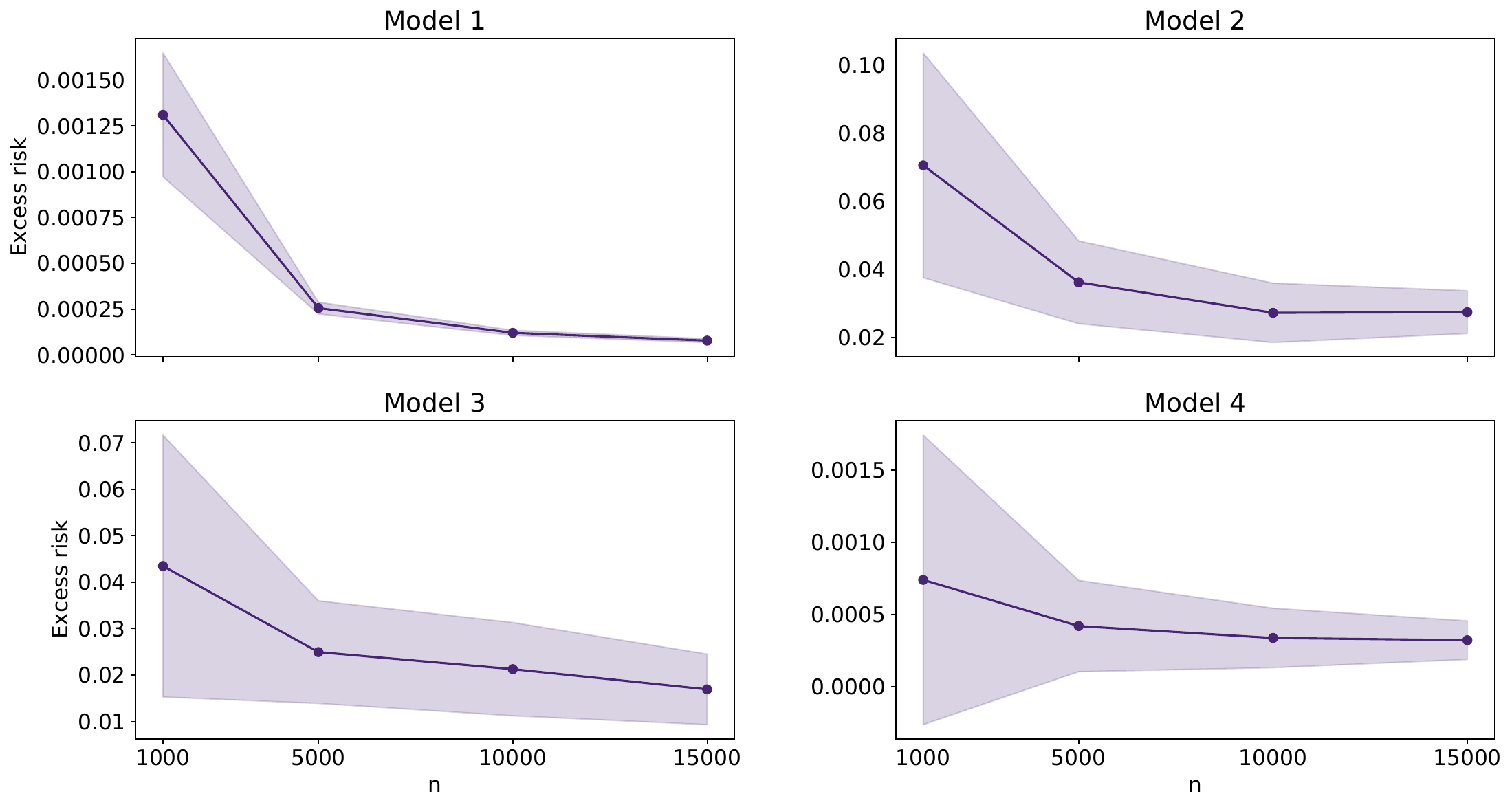}
    \caption{Consistency of Breiman RF: excess risk w.r.t\ sample size. RF parameters: 500 trees, max-depth set to None, max-features$=1$, no bootstrap. Mean over 30 tries (doted line) and std (filled zone).}
    \label{fig:breiman_consistency_mtry_one}
\end{figure}

\subsubsection{Decrease of the variance of the Breiman RF in a high-dimensional setting} \label{sec:exp_median_rf_high_dim}

Numerical experiments show the decrease of the variance of interpolating Breiman RF when $d$ increases. The model involves no signal and only noise (with specified variance $\sigma^2$).

\begin{figure}[ht]
    %\vspace{-1.2cm}
    \centering
    \includegraphics[clip, width=0.5\textwidth]{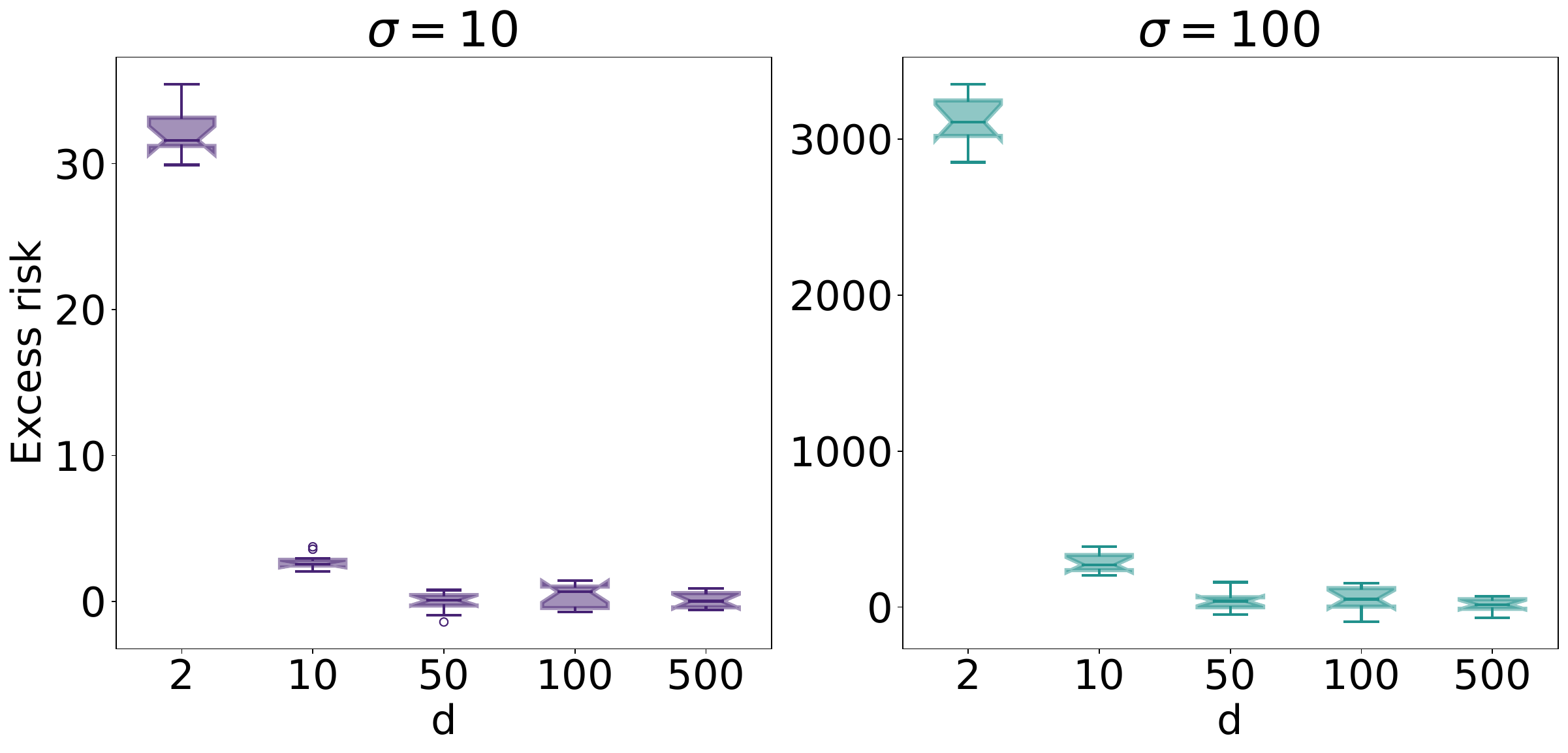}
    \caption{Decrease of the variance of an interpolating Breiman RF with \texttt{max-features=$1$} w.r.t.\ dimension $d$. 10 repetitions per boxplot, 5000 training points and 50000 testing points were used for each repetition. The Breiman RF contains 1000 trees.}
    \label{fig:breiman_high_dim_mf1}
\end{figure}

\begin{figure}[ht]
    %\vspace{-1.2cm}
    \centering
    \includegraphics[clip, width=0.5\textwidth]{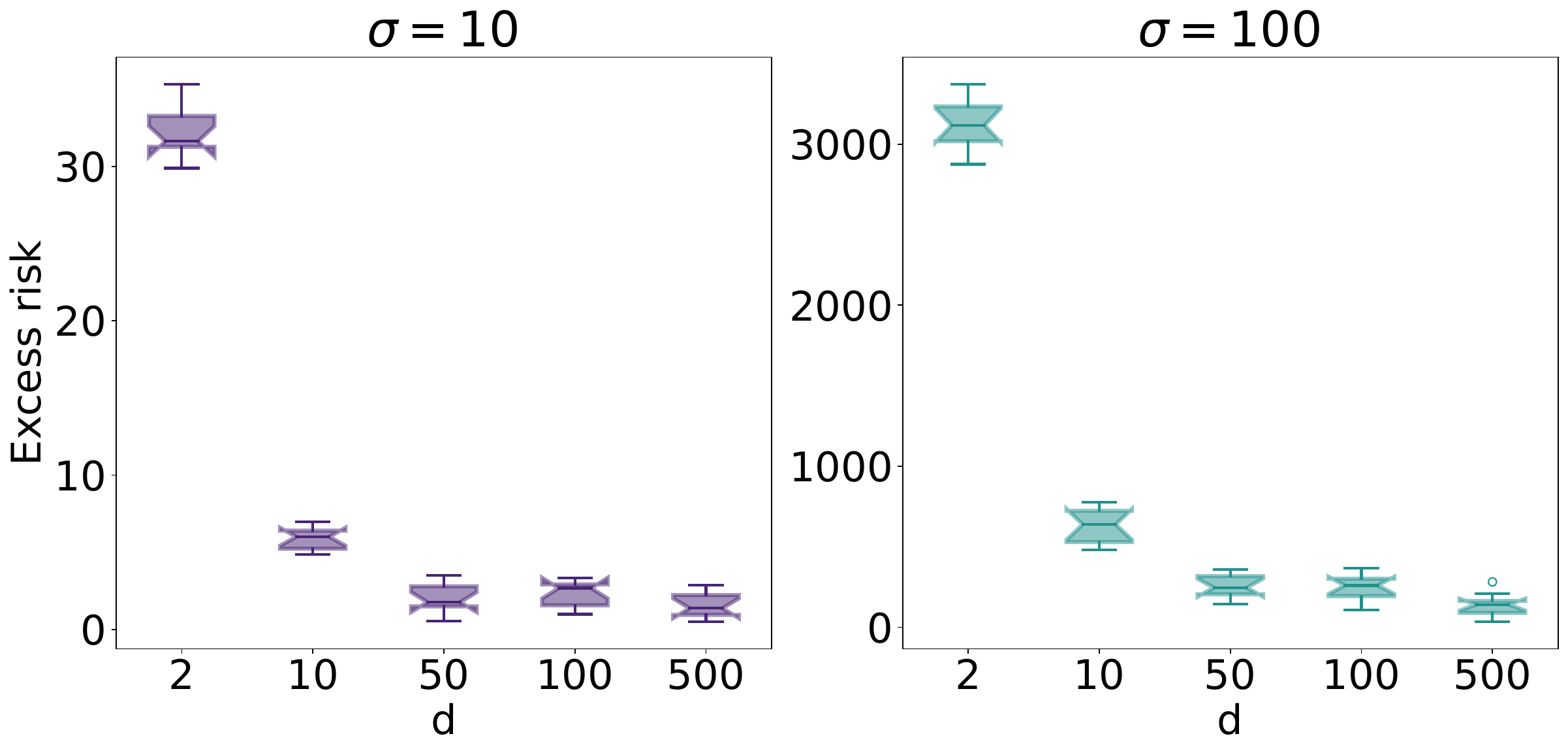}
    \caption{Decrease of the variance of an interpolating Breiman RF with \texttt{max-features=$\lfloor d/3 \rfloor$} w.r.t.\ dimension $d$. 10 repetitions per boxplot, 5000 training points and 50000 testing points were used for each repetition. The Breiman RF contains 1000 trees.}
    \label{fig:breiman_high_dim_mf3}
\end{figure}

\subsubsection{Comparison of Breiman RF with and without bootstrap}

\begin{figure}[h]
    \centering
    \includegraphics[clip, width=0.8\textwidth]{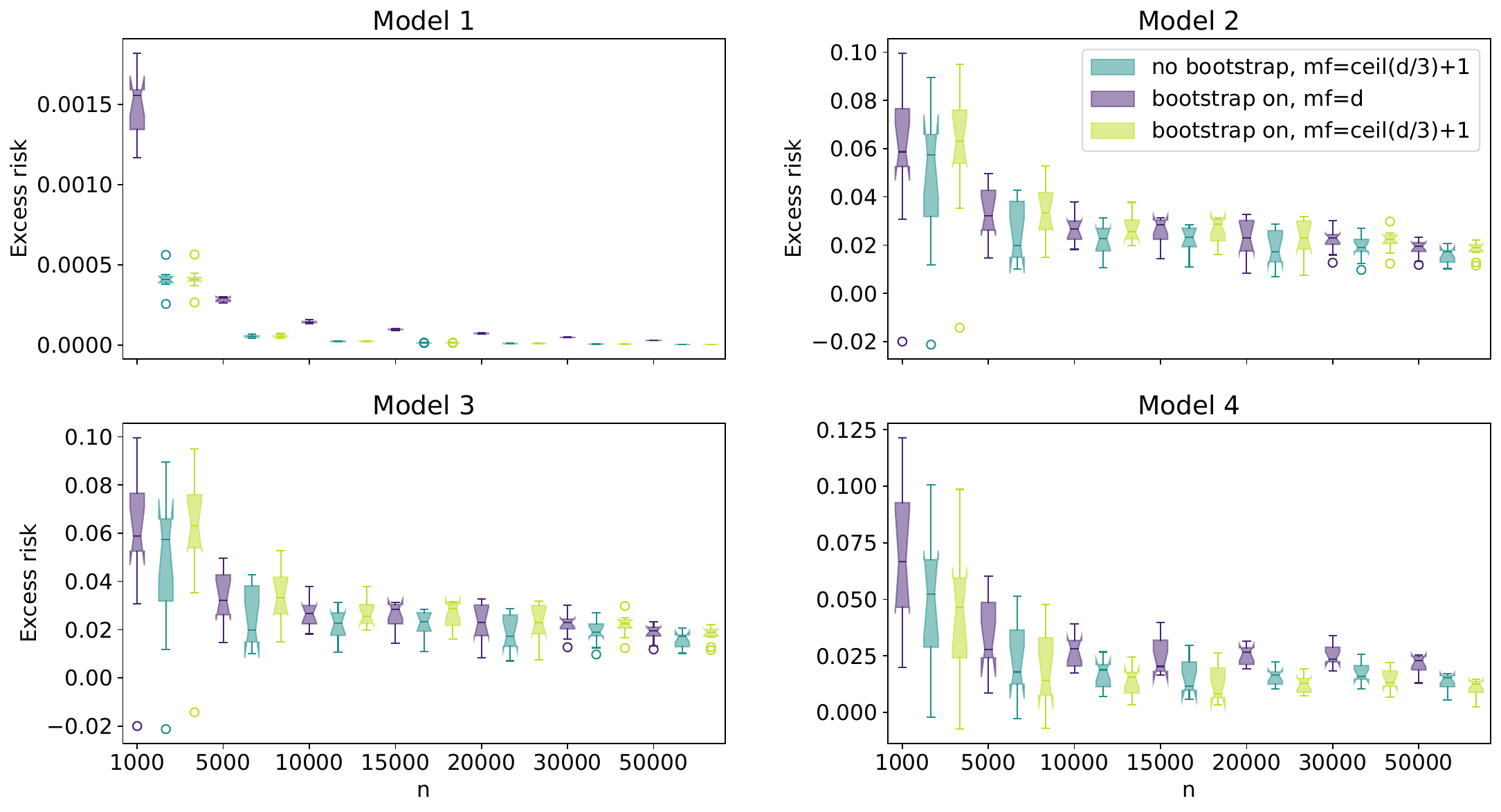}
    \caption{\replace{}{Consistency of Breiman RF: excess risk w.r.t\ sample size. RF parameters: 2000 trees, max-depth set to None, max-features$=1$. Boxplots over 10 tries.}}
    \label{fig:comp_breiman_w_without_bootstrap}
\end{figure}

\subsection{Interpolation experiments}
\label{app:interpolation_exp}

\subsubsection{Volume of the interpolation zone w.r.t\ sample size $n$}
%The Breiman RF used to compute the volume of the interpolation zone have 5000 \textit{trees}, \textit{max features} set to $1$ or specified, \textit{max depth} set to \textit{None} and \textit{bootstrap} set to \textit{false}.

We numerically evaluate the volume of the interpolation area of a Breiman RF (with 5000 trees, see Figure~\ref{fig:volume_interp_zone_tree_varying} in Appendix \ref{app:interpolation_exp} for details about this choice) when the sample size $n$ increases.

\begin{figure}[h]
    %\vspace{-1.8cm}
    \centering
    \includegraphics[clip, width=0.7\textwidth]{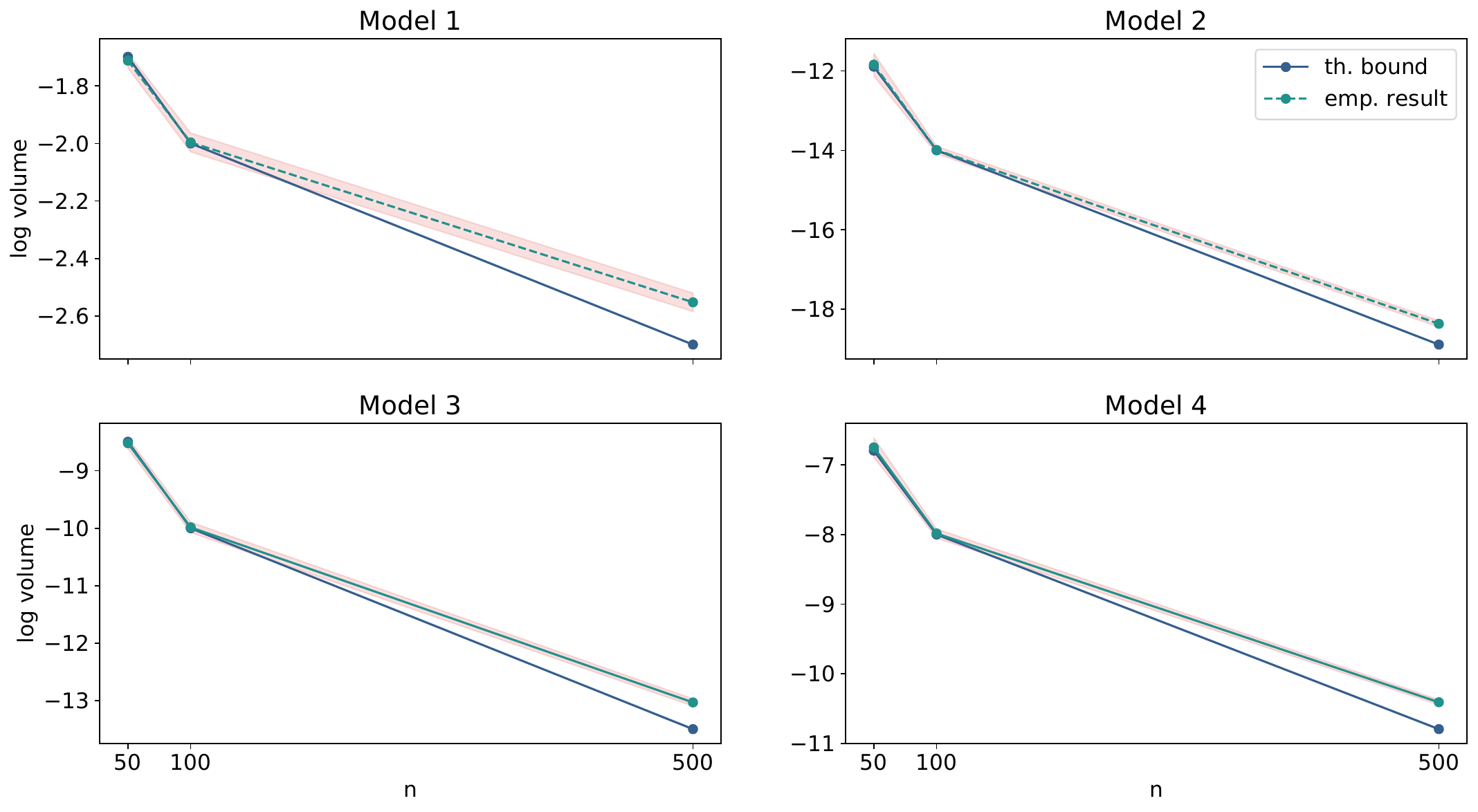}
    \caption{Log volume of the interpolation zone of a Breiman RF with 5000 Trees, max features set to 1, no bootstrap. Mean over 10 tries (red line) and mean $\pm$ std (filled zone). The theoretical bound (Proposition \ref{prop:interpol_limit_zone}) is represented in green.}
    \label{fig:volume_interpolation_zone}
\end{figure}
 
In Figure \ref{fig:volume_interpolation_zone}, the volume of the minimal interpolation zone is shown to tend polynomially fast to 0 (linear in the logarithmic scale) for all considered models as the dataset size increases, matching the behavior of the theoretical bound established in Proposition \ref{prop:interpol_limit_zone}.

One could notice the slight gap between the theoretical and experimental curves, which actually reflects the gap between an infinite forest (for which Proposition \ref{prop:interpol_limit_zone} holds) and its approximation by a finite forest (5000 trees here). This gap naturally tends to increase with $n$ (when the number of trees is fixed) as the approximation of the infinite RF by a finite one deteriorates with $n$.

\paragraph{Increasing \texttt{max-feature} parameter}
We plot on Figure \ref{fig:volume_interp_mf_tierce} the log-volume of the interpolation zone of a Breiman RF with the max-features parameter set to $\lceil d/3 \rceil$ (the default value proposed in \texttt{R} \texttt{randomForest} package). The volume decreases polynomially in $n$ but slower than when max-features$=1$ (Figure \ref{fig:volume_interpolation_zone}) which is to be expected: choosing max-features$=1$ should increase the diversity of the splits and therefore reduce the volume of the interpolation zone.

\begin{figure}[ht]
    \centering
    \includegraphics[clip, width=0.7\textwidth]{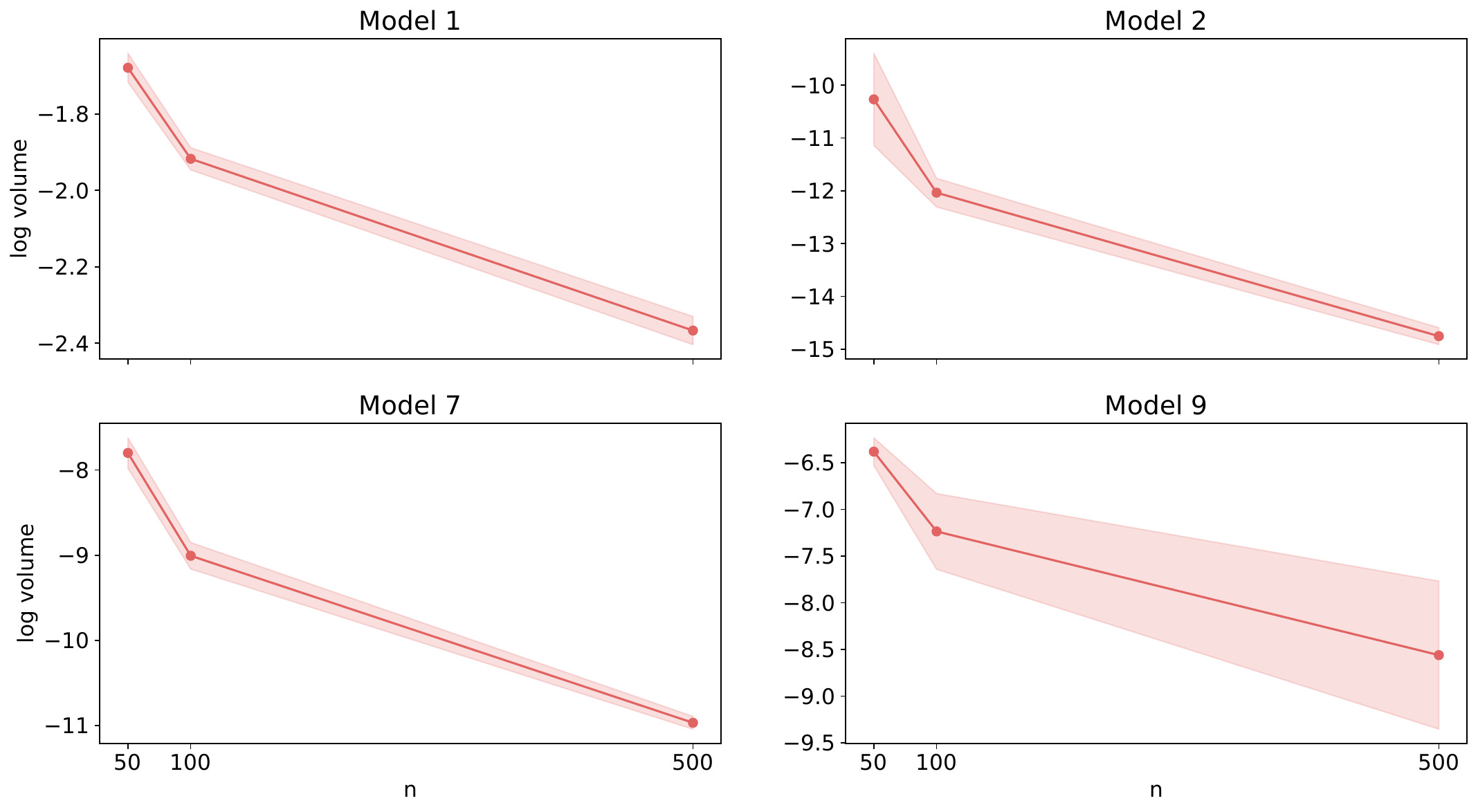}
    \caption{Log volume of Breiman RF interpolation zone w.r.t.\ sample size $n$. RF parameters: 500 trees, no bootstrap, max features $= \lceil d/3 \rceil$. Mean over 10 tries (bold line) and std (filled zone).}
    \label{fig:volume_interp_mf_tierce}
\end{figure}

\subsubsection{Volume of the interpolation zone w.r.t\ number of trees $M$}

In this section, we empirically measure how fast decreases the volume of the interpolation zone of a Breiman RF when its number of trees $M$ increases, and how close the interpolation zone gets from the minimal interpolation zone. 

To this end, for a fixed sample size $n=500$, we numerically evaluate the volume of the interpolation area when the number $M$ of trees in the forest grows. 
This volume is anticipated to be a non-increasing function of $M$ (for $M=1$, note that the interpolation volume is 1, the volume of $[0,1]^d$), but its decrease rate highly depends on the data geometry, making its theoretical evaluation difficult. 
The numerical results in Figure \ref{fig:volume_interp_zone_tree_varying} show a fast decay towards zero of the interpolation volume for all models, already tiny from $M=500$ trees.
Furthermore, it seems to converge to the theoretical bound (dotted line) derived in Proposition \ref{prop:interpol_limit_zone} for an infinite RF with a max-feature parameter equal to $1$.

\begin{figure}[!h]
    \centering
    \includegraphics[clip, width=0.8\textwidth]{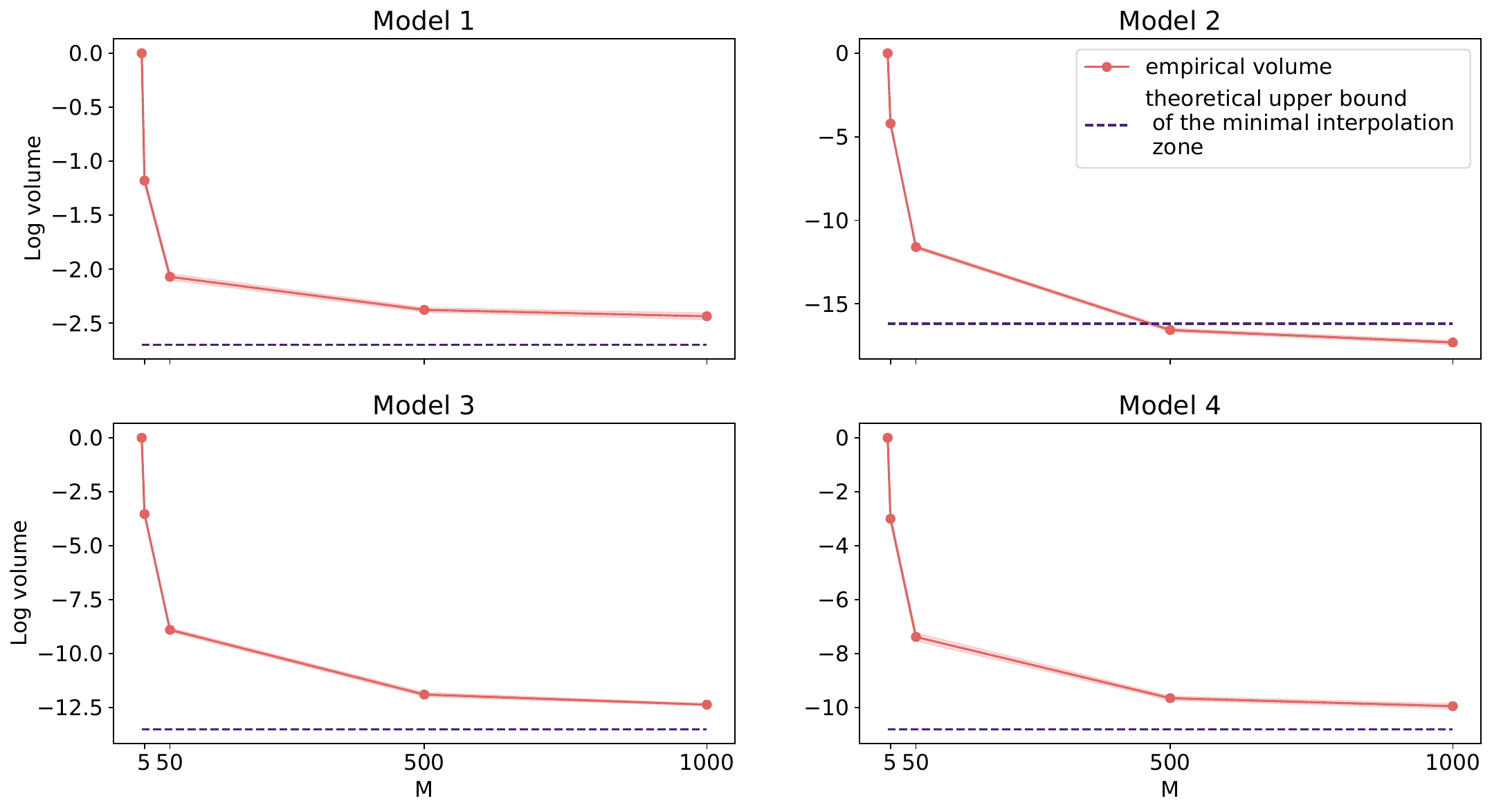}
    \caption{Log volume of Breiman RF interpolation zone w.r.t.\ the number $M$ of trees. RF parameters: no bootstrap, max features $= 1$. Mean over 10 tries (bold line) and std (filled zone). Sample size $n=500$.}
    \label{fig:volume_interp_zone_tree_varying}
\end{figure}

%For the four different models already considered and a sample size of $500$, we numerically measure the volume of interpolation for an increasing number of trees in the forest. 

%As shown by Figure \ref{fig:volume_interp_zone_tree_varying}, for all models, the volume of the interpolation zone decreases as the number of trees increases. 

\subsubsection{Analysis of the interpolation property of Breiman RF with bootstrap}
\label{sec:interp_loss_quant}

In this experiment, we try to measure how close a Breiman RF with bootstrap on is from exactly interpolating (with other parameters being 500 trees, max-depth set to None, max-features$=d$). To this end, we measure the difference between the true train labels (the $Y_i$s) and the predicted ones (the $\hat{Y}_i$s) by computing $$ I_{\text{loss}} := \frac{1}{n} \displaystyle \sum_{i=1}^n \frac{ |Y_i - \hat{Y}_i|}{Y_i}.$$
The closer is this quantity to 0, the closer is the forest from interpolating. On Figure \ref{fig:breiman_train_interp_quant_n_var}, we plot different quantiles of the above quantity as $n$ varies.

\begin{figure}[h]
    \centering
    \includegraphics[clip, width=0.8\textwidth]{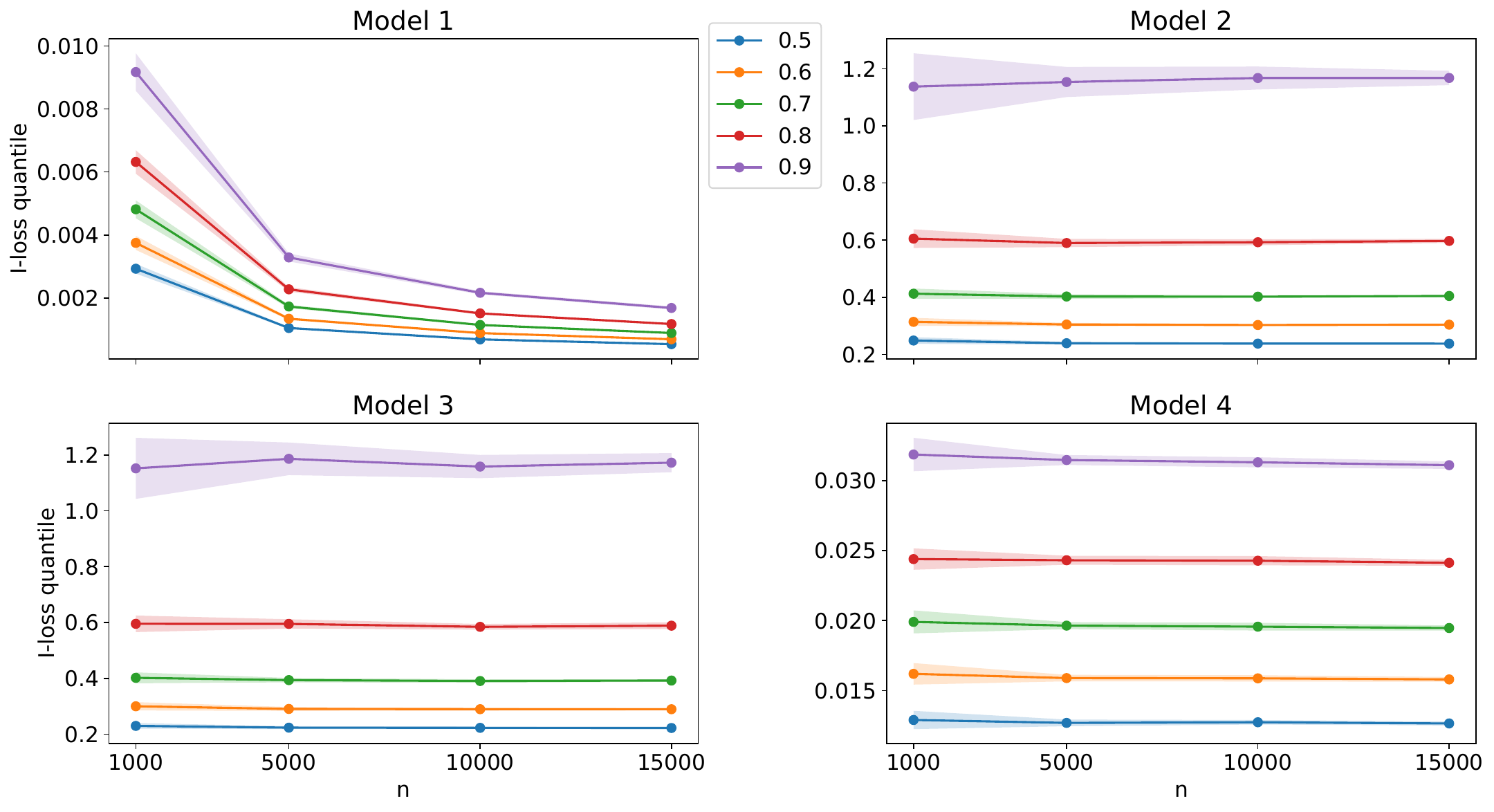}
    \caption{$I_{\text{loss}}$ of a Breiman RF w.r.t\ sample size $n$. RF parameters: 500 trees, bootstrap on, max-features$=d$, max-depth set to None. Mean over 30 tries (doted lines) and std (filled zones).}
    \label{fig:breiman_train_interp_quant_n_var}
\end{figure}

For instance, if we take the $0.8$-quantile in red on Figure \ref{fig:breiman_train_interp_quant_n_var} and look at the upper-right plot (model 2), we read that the $I_{\text{loss}}$ roughly equals $0.6$ for $80\%$ of the points. This quantity seems globally constant in $n$. Finally, the quantiles are smaller in the case of a strong signal-to-noise ratio (models 1 and 4) than in the case of a bigger one (models 2 and 3). 

\begin{figure}[h]
    \centering
    \includegraphics[clip, width=0.8\textwidth]{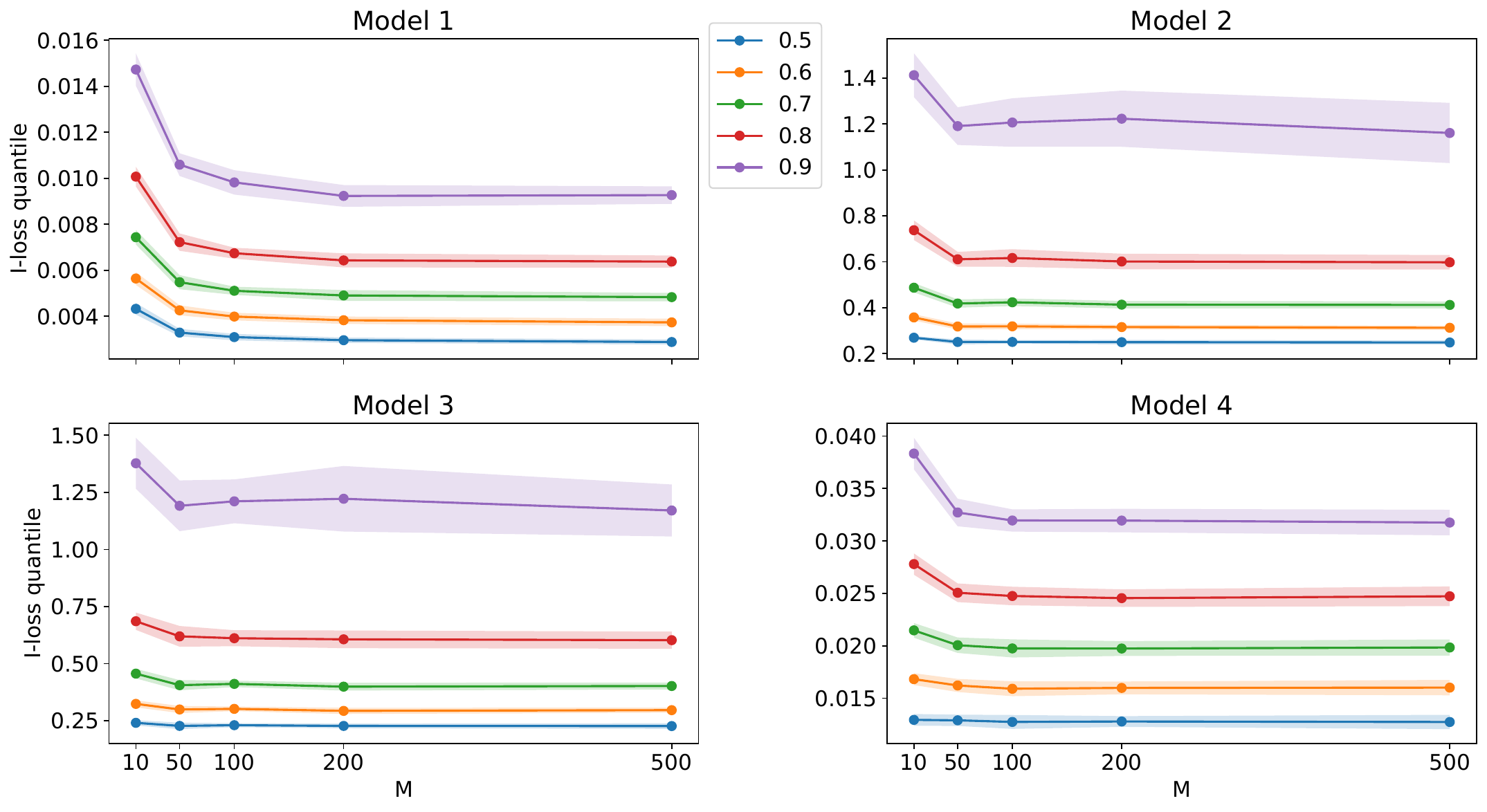}
    \caption{$I_{\text{loss}}$ of a Breiman RF w.r.t\ number of trees. Parameters: bootstrap on, max-features$=d$, max-depth set to None. Sample size $n=1000$. Mean over 30 tries (doted lines) and std (filled zones). }
    \label{fig:breiman_train_interp_quant_tree_var}
\end{figure}

On Figure \ref{fig:breiman_train_interp_quant_tree_var}, we also plot the quantiles of the $I_{\text{loss}}$ for the four different models while the number of trees varies. Adding trees does not significantly change the value of the different quantiles.

\end{document}